\documentclass[11pt]{amsart}
\usepackage{amsmath,amssymb,amsthm,mathrsfs,enumerate,bm,xcolor,multirow,pbox,soul}
\usepackage{graphicx,color,framed,tikz}

\allowdisplaybreaks[4]
\numberwithin{equation}{section}
\newcommand{\N}{\mathbb{N}}
\newcommand{\R}{\mathbb{R}}
\newcommand{\Q}{\mathbb{Q}}
\newcommand{\E}{\mathbb{E}}
\newcommand{\Prob}{\mathbb{P}}
\newcommand{\G}{\mathbb{G}}

\newcommand{\pnorm}[2]{\lVert#1\rVert_{#2}}
\newcommand{\bigpnorm}[2]{\big\lVert#1\big\rVert_{#2}}
\newcommand{\biggpnorm}[2]{\bigg\lVert#1\bigg\rVert_{#2}}
\newcommand{\abs}[1]{\lvert#1\rvert}
\newcommand{\bigabs}[1]{\big\lvert#1\big\rvert}
\newcommand{\biggabs}[1]{\bigg\lvert#1\bigg\rvert}

\renewcommand{\epsilon}{\varepsilon}

\renewcommand{\d}[1]{\mathrm{d}#1}

\newcommand{\beq}{\begin{equation}}
\newcommand{\eeq}{\end{equation}}
\newcommand{\beqa}{\begin{equation} \begin{aligned}}
\newcommand{\eeqa}{\end{aligned} \end{equation}}
\newcommand{\beqas}{\begin{equation*} \begin{aligned}}
\newcommand{\eeqas}{\end{aligned} \end{equation*}}

\newcommand{\bit}{\begin{itemize}}
	\newcommand{\eit}{\end{itemize}}
\newcommand{\bmat}{\begin{bmatrix}}
	\newcommand{\emat}{\end{bmatrix}}

\makeindex

\theoremstyle{definition}\newtheorem{definition}{Definition}
\theoremstyle{remark}\newtheorem{assumption}{Assumption}

\theoremstyle{remark}\newtheorem{remark}{Remark}
\theoremstyle{definition}\newtheorem{example}{Example}
\theoremstyle{plain}\newtheorem{question}{Question}
\theoremstyle{plain}\newtheorem{theorem}{Theorem}
\theoremstyle{plain}\newtheorem{lemma}{Lemma}
\theoremstyle{plain}\newtheorem{proposition}{Proposition}
\theoremstyle{plain}\newtheorem{corollary}{Corollary}
\theoremstyle{plain}

\begin{document}

\title[Convergence rates of LSEs with heavy-tailed errors]{
Convergence rates of Least Squares Regression Estimators with heavy-tailed errors
}
\thanks{Supported in part by NSF Grant DMS-1104832, DMS-1566514 and NI-AID grant R01 AI029168. }

\author[Q. Han]{Qiyang Han}

\address[Q. Han]{
Department of Statistics, Box 354322, University of Washington, Seattle, WA 98195-4322, USA.
}
\email{royhan@uw.edu}

\author[J. A. Wellner]{Jon A. Wellner}

\address[J. A. Wellner]{
Department of Statistics, Box 354322, University of Washington, Seattle, WA 98195-4322, USA.
}
\email{jaw@stat.washington.edu}

\date{\today}

\maketitle

\begin{abstract}

We study the performance of the Least Squares Estimator (LSE) in a general nonparametric regression model, when the errors are independent of the covariates but may only have a $p$-th moment ($p\geq 1$). In such a heavy-tailed regression setting, we show that if the model satisfies a standard `entropy condition' with exponent $\alpha \in (0,2)$, then the $L_2$ loss of the LSE converges at a rate
\begin{align*}
\mathcal{O}_{\mathbf{P}}\big(n^{-\frac{1}{2+\alpha}} \vee n^{-\frac{1}{2}+\frac{1}{2p}}\big).
\end{align*}
Such a rate cannot be improved under the entropy condition alone.

This rate quantifies both some positive and negative aspects of the LSE in a heavy-tailed regression setting. On the positive side, as long as the errors have $p\geq 1+2/\alpha$ moments, the $L_2$ loss of the LSE converges at the same rate as if the errors are Gaussian. On the negative side, if $p<1+2/\alpha$, there are (many) hard models at any entropy level $\alpha$ for which the $L_2$ loss of the LSE converges at a strictly slower rate than other robust estimators. 

The validity of the above rate relies crucially on the independence of the covariates and the errors. In fact, the $L_2$ loss of the LSE can converge arbitrarily slowly when the independence fails.

The key technical ingredient is a new multiplier inequality that gives sharp bounds for the `multiplier empirical process' associated with the LSE. We further give an application to the sparse linear regression model with heavy-tailed covariates and errors to demonstrate the scope of this new inequality.

\end{abstract}


\section{Introduction}

\subsection{Motivation and problems}
Consider the classical setting of nonparametric regression:  
suppose that 
\begin{align}\label{model:nonparametric_reg}
Y_i = f_0 (X_i) + \xi_i \ \ \mbox{for} \ \ i=1, \ldots , n
\end{align}
where $f_0 \in \mathcal{F}$, a class of possible regression functions $f$ where $f : \mathcal{X} \rightarrow \R$, 
$X_1, \ldots , X_n$ are i.i.d. $P$  on $(\mathcal{X},\mathcal{A})$, and $\xi_1, \ldots , \xi_n$ are i.i.d ``errors'' independent
of $X_1, \ldots ,X_n$.  We observe the pairs $\{ (X_i, Y_i ) : \ 1 \le i \le n \}$ and want to estimate $f_0$.

While there are many approaches to this problem, the most classical approach has been to study the 
Least Squares Estimator (or LSE) $\hat{f}_n$ defined by
\begin{align}\label{lse}
	\hat{f}_n = \mbox{argmin}_{f \in \mathcal{F}} \frac{1}{n} \sum_{i=1}^n (Y_i - f(X_i))^2.
\end{align}
The LSE is well-known to have nice properties (e.g. rate-optimality) when:
\begin{enumerate}
	\item[(E)] the errors $\{ \xi_i \}$ are sub-Gaussian or at least sub-exponential;
	\item[(F)] 
	the class $\mathcal{F}$ of regression functions satisfies a condition slightly stronger than a \emph{Donsker} 
	condition: namely, either a uniform entropy condition or a bracketing entropy condition with exponent $ \alpha \in (0,2)$:
	\begin{align*}
	\sup_Q \log \mathcal{N}(\epsilon \| F \|_{L_2 (Q)}, \mathcal{F} , L_2 (Q) ) \lesssim \epsilon^{-\alpha}, 
	\end{align*}
	where the supremum is over all finitely discrete measures $Q$ on $(\mathcal{X},\mathcal{A})$, or
	\begin{align*}
	\log \mathcal{N}_{[\, ]} (\epsilon , \mathcal{F}, L_2 (P) ) \lesssim \epsilon^{-\alpha}.
	\end{align*}
\end{enumerate}
See for example \cite{birge1993rates} and \cite{van2000empirical}, chapter 9, and Section \ref{section:notation} for notation.
In spite of a very large literature, there remains a lack of clear 
understanding of the properties of $\hat{f}_n$ 
in terms of assumptions concerning the heaviness of the tails of the errors and the massiveness or ``size'' 
of the class $\mathcal{F}$.

Our interest here is in developing further tools and methods to study properties of $\hat{f}_n$, especially its convergence rate when the error condition (E) is replaced by:
\begin{enumerate}
	\item[(E$'$)] the errors $\{ \xi_i \}$ have only a $p$-moment for some $1 \leq p < \infty$.
\end{enumerate}

This leads to our first question:

\begin{question}
	What determines the convergence rate $b_n$ of $\hat{f}_n$ with respect to some 
	risk or loss functions?  When is this rate $b_n$ determined by $p$ (and hence the tail behavior of 
	the $\xi_i$'s), and when is it determined by $\alpha$  (and hence the size of $\mathcal{F}$)? 
\end{question} 

There are a variety of measures of loss and risk in this setting.  Two of the most common are:
\begin{enumerate}
	\item[(a)] Empirical $L_2 $ loss:  $\| \hat{f}_n - f_0 \|_{L_2 (\Prob_n)} $\footnote{We write $\Prob_n$ for the empirical measure of the $(X_i, Y_i )$ pairs:  
		$\Prob_n = n^{-1} \sum_{i=1}^n \delta_{(X_i, Y_i)}$.}, and the corresponding risk  
	$\E \| \hat{f}_n - f_0 \|_{L_2 (\Prob_n)} $.
	\item[(b)] Population (or prediction) $L_2 $ loss $\| \hat{f}_n - f_0 \|_{L_2 (P)} $, and the corresponding risk $\E \| \hat{f}_n - f_0 \|_{L_2 (P)} $.
\end{enumerate}
Here we will mainly focus on measuring loss or risk in the sense of the prediction loss (b) 
since it corresponds to the usual choice in the language of Empirical Risk Minimization; see e.g. \cite{barlett2005local,bartlett2006empirical,birge1993rates,koltchinskii2006local,koltchinskii2000rademacher,massart2006risk,van1990estimating,van2000empirical,van1996weak}.
Thus we will (usually) measure loss or risk in $L_2 (P)$ and hence study rates of convergence of 
\begin{align*}
	\| \hat{f}_n - f_0 \|_{L_2 (P)} = \bigg[\int_{\mathcal{X}} | \hat{f}_n (x; (X_1,Y_1), \ldots , (X_n , Y_n ) ) - f_0 (x) |^2\ \d{P}(x)\bigg]^{1/2},
\end{align*}
or, in somewhat more compact notation,
\begin{align*}
\E \| \hat{f}_n - f_0 \|_{L_2 (P)} = \E\bigg[ \int_{\mathcal{X} } | \hat{f}_n (x) - f_0 (x) |^2\ \d{P} (x)\bigg]^{1/2}.
\end{align*}
As we will see in Section 3, the rate of convergence of the LSE $\hat{f}_n$ under conditions (E$'$) and (F) is
\begin{align}\label{ConvergenceRateLSE}
\| \hat{f}_n - f_0 \|_{L_2 (P)}=\mathcal{O}_{\mathbf{P}}\big(n^{-\frac{1}{2+\alpha}}\vee n^{-\frac{1}{2}+\frac{1}{2p}}\big).
\end{align}
So, the dividing line between $p$ and $\alpha$ in determining the rate of convergence of the LSE is given by
\begin{align*}
p = 1+ 2/\alpha
\end{align*}
in the following sense:  
\begin{enumerate}
	\item[($R_{\alpha}$)] If $p\geq 1 + 2/\alpha$, 
	then for any function class with entropy exponent $\alpha$, the rate of convergence of the
	LSE is $\mathcal{O}_{\mathbf{P}} (n^{-1/(2+\alpha)})$. 
	\item[($R_{p}$)] If $p< 1 + 2/\alpha$, then there exist model classes $\mathcal{F}$ with entropy exponent $\alpha$ such that the rate of convergence of the LSE is $\mathcal{O}_{\mathbf{P}} ( n^{-1/2+1/(2p)})$.
\end{enumerate}
These rates in $R_\alpha$ and $R_p$ indicate both some positive and negative aspects of the LSE in a heavy-tailed regression setting:
\begin{itemize}
	\item If $p\geq 1+2/\alpha$, then the heaviness of the tails of the errors (E$'$) does not play a role in the rate of convergence of the LSE,  since the rate in $R_\alpha$ coincides with the usual rate under the light-tailed error assumption (E) and the entropy condition (F).
	\item If $p<1+2/\alpha$, there exist (many) hard models at any entropy level $\alpha$ for which the LSE converges only at a slower rate $\mathcal{O}_{\mathbf{P}} ( n^{-1/2+1/(2p)})$ compared with the faster (optimal) rate $\mathcal{O}_{\mathbf{P}} (n^{-1/(2+\alpha)})$---a rate that can be achieved by other robust estimation procedures. See Section 3 for examples and more details.
\end{itemize} 

It should be noted that the assumption of independence of the errors $\xi_i$'s and the $X_i$'s in the regression model (\ref{model:nonparametric_reg}) is crucial for the above results to hold. In fact, when the errors $\xi_i$'s can be dependent on the $X_i$'s, there is no longer any universal moment condition on the $\xi_i$'s alone that guarantees the rate-optimality of the LSE, as opposed to ($R_\alpha$) (cf. Proposition \ref{prop:impossibility_LSE}).

To briefly introduce the main new tool we develop in Section 2 below, we first recall the classical methods used
to prove consistency and rates of convergence of the LSE (and many other contrast-type estimators).  These methods
are based on a ``basic inequality'' which lead naturally to a multiplier empirical process.  
This is well-known to experts in the area, but we will briefly review the basic facts here.  Since $\hat{f}_n $ minimizes the functional $f\mapsto
\Prob_n (Y - f(X))^2 = n^{-1} \sum_{i=1}^n (Y_i - f(X_i))^2$, it follows that 
\begin{align*}
\Prob_n (Y - \hat{f}_n(X))^2 \leq  \Prob_n (Y - f_0 (X) )^2 .
\end{align*}
Adding and subtracting $f_0$ on the left side, some algebra yields
\begin{align*}
\Prob_n (Y - {f}_0 (X))^2 + 2 \Prob_n (Y - {f}_0)(f_0  - \hat{f}_n  ) 
+ \Prob_n (f_0  - \hat{f}_n  )^2   \leq  \Prob_n (Y - {f}_0 (X))^2 .
\end{align*}
Since $\xi_i = Y_i - f_0 (X_i)$ under the model given by (\ref{model:nonparametric_reg}) we conclude that 
\begin{align}\label{BasicInequalityLSE}
\Prob_n ( \hat{f}_n (X)- f_0 (X))^2 
& \leq 2 \Prob_n \left ( \xi (\hat{f}_n (X) - f_0 (X)) \right )\leq 2 \sup_{f\in \mathcal{F}} \Prob_n \left( \xi (f (X) - f_0 (X) ) \right) 
\end{align}
where the process 
\begin{align}\label{MultEmpProc}
f \mapsto n(\Prob_n-P)(\xi f(X))=n\Prob_n (\xi f(X))=\sum_{i=1}^n \xi_i f(X_i)
\end{align}
is a \emph{multiplier empirical process}.  This is exactly as in Section 4.3 of \cite{van2000empirical}.  When the $\xi_i$'s are Gaussian, the process in (\ref{MultEmpProc}) is even a Gaussian process conditionally on the $X_i$'s, and is relatively easy to analyze.  If the $\{\xi_i \}$'s are integrable and 
$\mathcal{F}$ is a \emph{Glivenko-Cantelli class} of functions, then the inequality (\ref{BasicInequalityLSE}) leads easily to consistency 
of the LSE in the sense of the loss and risk measures (a);  see e.g. \cite{van2000empirical}.  

To obtain rates of convergence we need to consider localized versions of the processes in (\ref{BasicInequalityLSE}),
much as in Section 3.4.3 of \cite{van1996weak}.  As in Section 3.4.3 of \cite{van1996weak}, (but replacing their
$\theta \in \Theta$ and $\epsilon$ by our $f \in \mathcal{F}$ and $\xi$) we consider 
\begin{align*}
\mathbb{M}_n (f) = 2 \Prob_n \xi (f - f_0) - \Prob_n (f-f_0)^2, 
\end{align*}
and note that $\hat{f}_n$ maximizes $\mathbb{M}_n (f)$ over $\mathcal{F}$.  
Since the errors have zero mean and are independent of the $X_i$'s, this process has mean 
$M(f) \equiv - P( f- f_0)^2$.  Since $\mathbb{M}_n (f_0) =0 = M(f_0)$, centering then yields the process
\begin{align*}
f \mapsto \mathbb{Z}_n (f) & \equiv  \mathbb{M}_n (f) - \mathbb{M}_n (f_0) - (M(f) - M(f_0)) \\
& = 2 \Prob_n \xi (f-f_0) - (\Prob_n - P)(f - f_0)^2 .
\end{align*}
Establishing rates of convergence for $\hat{f}_n$ then boils down to bounding 
\begin{align*}
\E \sup_{f \in \mathcal{F}: P(f-f_0)^2 \leq \delta^2} \mathbb{Z}_n (f)
\end{align*}
as a function of $n$ and $\delta$; see e.g. \cite{van1996weak} Theorem 3.4.1, pages 322-323. It is clear at least for $\mathcal{F}\subset L_\infty$ that this can be accomplished if we have good bounds for the multiplier empirical process (\ref{MultEmpProc}) in terms of the empirical process itself
\begin{eqnarray}
f \mapsto n (\Prob_n-P) (f(X)\big) = \sum_{i=1}^n \big( f(X_i)-Pf\big),
\label{EmpProc}
\end{eqnarray}
or, in view of standard symmetrization inequalities (as in Section 2.3 of 
\cite{van1996weak}), its symmetrized equivalent,
\begin{eqnarray}
f \mapsto  \sum_{i=1}^n  \epsilon_i f(X_i),
\label{SymmEmpProc}
\end{eqnarray}
where the $\epsilon_i$ are i.i.d. Rademacher random variables 
$\Prob(\epsilon_i = \pm 1) = 1/2$ independent of the $X_i$'s.
This leads naturally to:

\begin{question}
	Under what moment conditions on the $\xi_i$'s can we assert that the multiplier 
	empirical process (\ref{MultEmpProc}) has (roughly) the same ``size'' as the 
	empirical process (\ref{EmpProc}) (or equivalently the symmetrized empirical process (\ref{SymmEmpProc})
	for (nearly) all function classes $\mathcal{F}$ in a non-asymptotic manner?
\end{question}

In Section \ref{section:multiplier_inequality} below we provide simple moment conditions on the $\xi_i$'s which yield a positive answer to Question 2, when the $\xi_i$'s are independent from the $X_i$'s.  
We then give some comparisons to the existing multiplier inequalities which illustrate the improvement possible via the 
new bounds in non-asymptotic settings, and show that our bounds also yield the asymptotic equivalence required for
multiplier CLT's (cf. Section 2.9 of \cite{van1996weak}). Further impossibility results are demonstrated, showing that there is no positive solution to Question 2 when the $\xi_i$'s and the $X_i$'s can be dependent.

In Section \ref{section:nonparametric_lse_minimax} we address Question 1 by applying the new multiplier inequality to derive the convergence rate of the LSE (\ref{ConvergenceRateLSE}) 
in the context of the nonparametric regression model (\ref{model:nonparametric_reg}), and indicate in greater detail both the positive and negative aspects of the LSE due to this rate. We further show that no solution to Question 1 exists when the errors $\xi_i$'s and the covariates $X_i$'s can be dependent.

 Not surprisingly, the new bounds for the multiplier empirical process have applications to many settings in which 
the Least Squares criterion plays a role, for example the Lasso in the sparse linear regression model.  
In Section \ref{section:sparse_linear_reg} we give an application of the new bounds in
a Lasso setting with both heavy-tailed errors and heavy-tailed covariates. Most detailed proofs are given in Sections \ref{section:proof_multiplier_ineq}-\ref{section:remaining_proof_II}.   
\subsection{Notation}\label{section:notation}
For a real-valued random variable $\xi$ and $1\leq p<\infty$, let $\pnorm{\xi}{p} \equiv \big(\E\abs{\xi}^p\big)^{1/p} $ denote the ordinary $p$-norm. The $L_{p,1}$ norm for a random variable $\xi$ is defined by 
\begin{align*}
\pnorm{\xi}{p,1}\equiv\int_0^\infty {\Prob(\abs{\xi}>t)}^{1/p}\ \d{t}.
\end{align*}
It is well known that $L_{p+\epsilon}\subset L_{p,1}\subset L_{p}$ holds for any underlying probability measure, and hence a finite $L_{p,1}$ condition requires slightly more than a $p$-th moment, but no more than any $p+\epsilon$ moment, see Chapter 10 of \cite{ledoux2013probability}.

For a real-valued measurable function $f$ defined on $(\mathcal{X},\mathcal{A},P)$, $\pnorm{f}{L_p(P)}\equiv \pnorm{f}{P,p}\equiv \big(P\abs{f}^p)^{1/p}$ denotes the usual $L_p$-norm under $P$, and $\pnorm{f}{\infty}\equiv \sup_{x \in \mathcal{X}} \abs{f(x)}$. $f$ is said to be $P$-centered if $Pf=0$, and $\mathcal{F}$ is $P$-centered if all $f \in \mathcal{F}$ are $P$-centered. $L_p(g,B)$ denotes the $L_p(P)$-ball centered at $g$ with radius $B$. For simplicity we write $L_p(B)\equiv L_p(0,B)$. To avoid unnecessary measurability digressions, we will assume that $\mathcal{F}$ is countable throughout the article. As usual, for any $\phi:\mathcal{F}\to \R$, we write $\pnorm{\phi(f)}{\mathcal{F}}$ for $ \sup_{f \in \mathcal{F}} \abs{\phi(f)}$.

Let $(\mathcal{F},\pnorm{\cdot}{})$ be a subset of the normed space of real functions $f:\mathcal{X}\to \R$. For $\epsilon>0$ let $\mathcal{N}(\epsilon,\mathcal{F},\pnorm{\cdot}{})$ be the $\epsilon$-covering number of $\mathcal{F}$, and let $\mathcal{N}_{[\,]}(\epsilon,\mathcal{F},\pnorm{\cdot}{})$ be the $\epsilon$-bracketing number of $\mathcal{F}$; see page 83 of \cite{van1996weak} for more details.

Throughout the article $\epsilon_1,\ldots,\epsilon_n$ will be i.i.d. Rademacher random variables independent of all other random variables. $C_{x}$ will denote a generic constant that depends only on $x$, whose numeric value may change from line to line unless otherwise specified. $a\lesssim_{x} b$ and $a\gtrsim_x b$ mean $a\leq C_x b$ and $a\geq C_x b$ respectively, and $a\asymp_x b$ means $a\lesssim_{x} b$ and $a\gtrsim_x b$ [$a\lesssim b$ means $a\leq Cb$ for some absolute constant $C$]. For two real numbers $a,b$, $a\vee b\equiv \max\{a,b\}$ and $a\wedge b\equiv\min\{a,b\}$. $\mathcal{O}_{\mathbf{P}}$ and $\mathfrak{o}_{\mathbf{P}}$ denote the usual big and small O notation in probability.

\section{The multiplier inequality}\label{section:multiplier_inequality}

Multiplier inequalities have a long history in the theory of empirical processes. Our new multiplier inequality in this section is closest in spirit to the classical multiplier inequality, cf. Section 2.9 of \cite{van1996weak} or \cite{gine2015mathematical}, but strictly improves the classical one in a non-asymptotic setting (see Section \ref{section:comparison_multiplier_ineq_vdvw}).

Our work here is also related to \cite{mendelson2016upper}, who derived bounds for the multiplier empirical process, assuming: (i) $\xi_i$'s have a $2+\epsilon$ moment, and (ii) $\{(\xi_i,X_i)\}$ are i.i.d. (i.e. $\xi_i$ need not be independent from $X_i$). The bounds in \cite{mendelson2016upper} use techniques from generic chaining \cite{talagrand2014upper}, and work particularly well for `sub-Gaussian classes' (defined in \cite{mendelson2016upper}). Our setting here will be different: we assume that: (i) $\xi_i$'s have a $L_{p,1}(p\geq 1)$ moment and (ii) $\xi_i$'s are independent from $X_i$'s, but the $\xi_i$'s need not be independent from each other. 

We make this choice in view of a negative result of Alexander \cite{alexander1985non}, stating that there is no universal moment condition on $\xi_i$'s for a multiplier CLT to hold when $\xi_i$'s need not be independent from $X_i$'s, while a $L_{2,1}$ moment condition is known to be universal in the independent case \cite{gine2015mathematical,ledoux2013probability,van1996weak}. The complication here makes it more hopeful to work in the independent case for a precise understanding of the multiplier empirical process. In fact:
\begin{itemize}
	\item In the independent case we are able to quantify the exact \emph{structural interplay} between the moment of the multipliers and the complexity of the indexing function class in the size of the multiplier empirical process (cf. Theorems \ref{thm:local_maximal_generic}-\ref{thm:lower_bound_mep}), thereby giving a satisfactory answer to Question 2;
	\item Such an interplay fails when the $X_i$'s may not be independent from the $\xi_i$'s. Moreover, no simple moment condition on the $\xi_i$'s alone can lead to a solution to Question 2 in the dependent case (cf. Proposition \ref{prop:impossibility}). 
\end{itemize}

\subsection{Upper bound}

We first state the assumptions.

\begin{assumption}\label{assump:multiplier_function}
	Suppose that $\xi_1,\ldots,\xi_n$ are independent of the random variables $X_1,\ldots,X_n$, and \emph{either} of the following conditions holds:
	\begin{enumerate}
		\item[(A1)]  $X_1,\ldots,X_n$ are i.i.d. with law $P$ on $(\mathcal{X},\mathcal{A})$, and $\mathcal{F}$ is $P$-centered.
		\item[(A2)]  $X_1,\ldots,X_n$ are permutation invariant, and $\xi_1,\ldots,\xi_n$ are independent mean-zero random variables.
	\end{enumerate}
\end{assumption}


\begin{theorem}\label{thm:local_maximal_generic}
	Suppose Assumption \ref{assump:multiplier_function} holds. Let $\{\mathcal{F}_k\}_{k=1}^n$ be a sequence of function classes such that $\mathcal{F}_k\supset \mathcal{F}_n$ for any $1\leq k\leq n$.
	Assume further that there exists a non-decreasing concave function $\psi_n:\R_{\geq 0}\to \R_{\geq 0}$ with $\psi_n(0)=0$ such that
	\begin{align}\label{cond:local_maximal_mulep}
	\E \biggpnorm{\sum_{i=1}^k\epsilon_i f(X_i)}{\mathcal{F}_k}\leq \psi_n(k)
	\end{align}
	holds for all $1\leq k\leq n$. Then
	\begin{align}
	\E \biggpnorm{\sum_{i=1}^n \xi_i f(X_i)}{\mathcal{F}_n}\leq 4\int_0^\infty  \psi_n\bigg(\sum_{i=1}^{n} \Prob(\abs{\xi_i}>t)\bigg) \ \d{t}.
	\end{align}
\end{theorem}

The primary application of Theorem \ref{thm:local_maximal_generic} to non-parametric regression problems in Section \ref{section:nonparametric_lse_minimax} involves a non-increasing sequence of function classes $\mathcal{F}_1\supset\ldots\supset \mathcal{F}_n$. It is also possible to use Theorem \ref{thm:local_maximal_generic} for the case $\mathcal{F}_1=\cdots=\mathcal{F}_n$; see Section \ref{section:sparse_linear_reg} for an application to the sparse linear regression model.

The following corollary provides a canonical concrete application of Theorem \ref{thm:local_maximal_generic}. 

\begin{corollary}\label{cor:local_maximal_ineq_barrier}
	Consider the same assumptions as in Theorem \ref{thm:local_maximal_generic}. Assume for simplicity that $\xi_i$'s have the same marginal distributions.  Suppose that for some $\gamma\geq 1$, and some constant $\kappa_0>0$,
	\begin{align}\label{cond:ex_generic_thm}
	\E \biggpnorm{\sum_{i=1}^k\epsilon_i f(X_i)}{\mathcal{F}_k}\leq \kappa_0 \cdot k^{1/\gamma}
	\end{align}
	holds for all $1\leq k\leq n$. Then for any $p\geq 1$ such that $\pnorm{\xi_1}{p,1}<\infty$, 
	\begin{align*}
	\E \biggpnorm{\sum_{i=1}^n \xi_i f(X_i)}{\mathcal{F}_n}\leq 4\kappa_0 \cdot n^{\max\{1/\gamma,1/p\}} \pnorm{\xi_1}{\min\{\gamma,p\},1}.
	\end{align*}
\end{corollary}

\begin{proof}
	First consider $\gamma\leq p$. In this case, letting $\psi_n(t)\equiv \kappa_0 t^{1/\gamma}$ in Theorem \ref{thm:local_maximal_generic}, we see that $
	\E \pnorm{\sum_{i=1}^n \xi_i f(X_i)}{\mathcal{F}_n}\leq 4\kappa_0\cdot n^{1/\gamma}\pnorm{\xi_1}{\gamma,1}$.
	On the other hand, if $\gamma>p$, we can take $\psi_n(t)\equiv \kappa_0 t^{1/p}$ to conclude that $
	\E \pnorm{\sum_{i=1}^n \xi_i f(X_i)}{\mathcal{F}_n}\leq 4\kappa_0\cdot n^{1/p}\pnorm{\xi_1}{p,1}$. Note that $\gamma\geq 1$ ensures the concavity of $\psi_n$.
\end{proof}

Corollary \ref{cor:local_maximal_ineq_barrier} says that the upper bound for the multiplier empirical process has two components: one part comes from the growth rate of the empirical process; another part comes from the moment barrier of the multipliers $\xi_i$'s.

\begin{remark}
	One particular case for application of Theorem \ref{thm:local_maximal_generic} and Corollary \ref{cor:local_maximal_ineq_barrier} is the following. Let $\delta_1\geq \ldots\geq \delta_n\geq 0$ be a sequence of non-increasing non-negative real numbers, and $\mathcal{F}$ be an arbitrary function class. Let $
	\mathcal{F}_k\equiv \mathcal{F}(\delta_k)\equiv \{f \in \mathcal{F}: Pf^2<\delta_k^2\}$
	be the `local' set of $\mathcal{F}$ with $L_2$-radius at most $\delta_k$. There exists a large literature on controlling such localized empirical processes; a classical device suited for applications in nonparametric problems is to use local maximal inequalities under either the uniform or bracketing entropy conditions (cf. Proposition \ref{prop:local_maximal_ineq}).
	
	An important choice in statistical applications for $\delta_k$ is given by
	\begin{align}\label{def:choice_delta_n}
	\E \biggpnorm{\sum_{i=1}^k \epsilon_i f(X_i)}{\mathcal{F}(\delta_k)}  \lesssim k\delta_k^2.
	\end{align}
	As will be seen in Section \ref{section:nonparametric_lse_minimax}, the above choice $\{\delta_k\}$ corresponds to the rate of convergence of the LSE in the nonparametric regression model (\ref{model:nonparametric_reg}).

	In this case Theorem \ref{thm:local_maximal_generic} and Corollary \ref{cor:local_maximal_ineq_barrier} yield that
	\begin{align}\label{def:choice_delta_n_xi}
	\E \biggpnorm{\sum_{i=1}^n \xi_i f(X_i)}{\mathcal{F}(\delta_n)}  \lesssim n\delta_n^2
	\end{align}
	given sufficient moments of the $\xi_i$'s. 
\end{remark}

\begin{remark}
	Choosing $\gamma\geq 2$ in Corollary \ref{cor:local_maximal_ineq_barrier} corresponds to the bounded Donsker regime\footnote{$\mathcal{F}$ is said to be \emph{bounded Donsker} if $\sup_{n \in \N}\E \sup_{f \in \mathcal{F}} \abs{\frac{1}{\sqrt{n}}\sum_{i=1}^n \epsilon_i f(X_i)}<\infty$.  } for the empirical process. In this case we only need $\pnorm{\xi_1}{2,1}<\infty$ to ensure the multiplier empirical process to also be bounded Donsker. This moment condition is generally unimprovable in view of \cite{ledoux1986conditions}. On the other hand, such a choice of $\gamma$ can fail due to: (i) failure of integrability of the envelope functions of the classes $\{\mathcal{F}_k\}$, or (ii) failure of the classes $\{\mathcal{F}_k\}$ to be bounded Donsker. (i) is related to the classical Marcinkiewicz-Zygmund strong laws of large numbers and the generalizations of those to empirical measures, see  \cite{andersen1988central,mason1983asymptotic,massart1998uniform}. For (ii), some examples in this regard can be found in \cite{strassen1969central}, Chapter 11 of \cite{dudley1999uniform}, see also Proposition 17.3.7 of \cite{shorack1986empirical}.
\end{remark}

Theorem \ref{thm:local_maximal_generic} and Corollary \ref{cor:local_maximal_ineq_barrier} only concern the first moment of the suprema of the multiplier empirical process. For higher moments, we may use the following Hoffmann-J\o rgensen/Talagrand type inequality relating the $q$-th moment estimate with the first moment estimate.

\begin{lemma}[Proposition 3.1 of \cite{gine2000exponential}]\label{lem:p_moment_estimate}
	Let $q\geq 1$. Suppose $X_1,\ldots,X_n$ are i.i.d. with law $P$ and $\xi_1,\ldots,\xi_n$ are i.i.d. with $\pnorm{\xi_1}{2\vee q}<\infty$. Let $\mathcal{F}$ be a class of functions with $\sup_{f \in \mathcal{F}} Pf^2\leq \sigma^2$ such that either $\mathcal{F}$ is $P$-centered, or $\xi_1$ is centered. Then
	\begin{align*}
	\E \sup_{f \in \mathcal{F}}\biggabs{\sum_{i=1}^n \xi_i f(X_i)}^q&\leq K^q\bigg[\bigg(\E\sup_{f \in \mathcal{F}} \biggabs{\sum_{i=1}^n \xi_i f(X_i)} \bigg)^q\\
	&\qquad+q^{q/2}(\sqrt{n}\pnorm{\xi_1}{2}\sigma)^q +q^q\E\max_{1\leq i\leq n} \abs{\xi_i}^q\sup_{f \in \mathcal{F}}\abs{f(X_i)}^q\bigg].
	\end{align*}
	Here $K>0$ is a universal constant.
\end{lemma}

\subsection{Lower bound}

Theorem \ref{thm:local_maximal_generic} and Corollary \ref{cor:local_maximal_ineq_barrier} do not require any structural assumptions on the function class $\mathcal{F}$. \cite{mendelson2016upper} showed that for a `sub-Gaussian' class, a $2+\epsilon$ moment on i.i.d. $\xi_i$'s suffices to conclude that the multiplier empirical process behaves like the canonical Gaussian process. One may therefore wonder if the moment barrier for the multipliers in Corollary \ref{cor:local_maximal_ineq_barrier} is due to an artifact of the proof. Below in Theorem \ref{thm:lower_bound_mep} we show that this barrier is intrinsic for general classes $\mathcal{F}$.

\begin{theorem}\label{thm:lower_bound_mep}
	Let $\mathcal{X}=[0,1]$ and $P$ be a probability measure on $\mathcal{X}$ with Lebesgue density bounded away from $0$ and $\infty$. Let $\xi_1,\ldots,\xi_n$ be i.i.d. random variables such that $\E \max_{1\leq i\leq n}\abs{\xi_i}\geq \kappa_0 n^{1/p}$ for some $p> 1$ and some constant $\kappa_0$ independent of $\xi_1$. Then for any $\gamma>2$, there exists a sequence of function classes $\{\mathcal{F}_k\}_{k=1}^n$ defined on $\mathcal{X}$ with $\mathcal{F}_k\supset \mathcal{F}_n$ for any $1\leq k\leq n$
	such that for $n$ sufficiently large,
	\begin{align*}
	\E\biggpnorm{\sum_{i=1}^k \epsilon_i f(X_i)}{{\mathcal{F}}_k}\leq \kappa_1\cdot  k^{1/\gamma},
	\end{align*}
	holds for all $1\leq k\leq n$, and that
	\begin{align*}
	\E \biggpnorm{\sum_{i=1}^n \xi_i f(X_i)}{{\mathcal{F}}_n}\geq \kappa_1^{-1} n^{\max\{1/\gamma,1/p\}}.
	\end{align*}
	Here $\kappa_1$ is a constant depending on $\kappa_0,\gamma$ and $P$.
\end{theorem}

\begin{remark}
	The condition on the $\xi_i$'s will be satisfied, for example if the $\xi_i$'s are i.i.d. with the tail condition $\Prob(\abs{\xi_i}>t)\geq \kappa_0'/(1+t^p)$ for $t>0$.
\end{remark}

Combined with Corollary \ref{cor:local_maximal_ineq_barrier}, it is seen that the growth rate $n^{\max\{1/\gamma,1/p\}}$ of the multiplier empirical process cannot be improved in general. This suggests an interesting phase transition phenomenon from a worst-case perspective: if the complexity of the function class dominates the effect of the tail of the multipliers, then the multiplier empirical process essentially behaves as the empirical process counterpart; otherwise the tail of the multipliers governs the growth of the multiplier empirical process.

\begin{remark}\label{rmk:lower_bound_mep}
The function class we constructed that witnesses the moment barrier rate $n^{1/p}$ in Theorem \ref{thm:lower_bound_mep} can be simply taken to be the class of indicators over closed intervals on $[0,1]$. Although being the `simplest' function class in the theory of empirical processes, this class serves as an important running example that achieves the bad rate $n^{1/p}$.
\end{remark}

\subsection{Comparison of Theorem \ref{thm:local_maximal_generic} with the multiplier inequality in \cite{van1996weak}}\label{section:comparison_multiplier_ineq_vdvw}
In this section we compare the classical multiplier inequality in Theorem \ref{thm:local_maximal_generic} with the one in Section 2.9 of \cite{van1996weak}, which originates from \cite{gine1984some,gine1986lectures,ledoux1986conditions}; see also \cite{gine2015mathematical}: for i.i.d. mean-zero $\xi_i$'s and i.i.d. $X_i$'s, and for any $1\leq n_0\leq n$,
\begin{align}\label{ineq:mutiplier_ineq_conclusion_2}
\E \biggpnorm{\frac{1}{\sqrt{n}}\sum_{i=1}^n \xi_i f(X_i)}{\mathcal{F}}
&\lesssim (n_0-1) \E \pnorm{f(X_1)}{\mathcal{F}} \frac{\E \max_{1\leq i\leq n} \abs{\xi_i}}{\sqrt{n}}\\
&\qquad + \pnorm{\xi_1}{2,1} \max_{n_0\leq k\leq n} \E \biggpnorm{\frac{1}{\sqrt{k}}\sum_{i=1}^k \epsilon_i f(X_i)}{\mathcal{F}}.\nonumber
\end{align}

\subsubsection{Non-asymptotic setting}
The major drawback of (\ref{ineq:mutiplier_ineq_conclusion_2}) is that it is not sharp in a non-asymptotic setting.  For an illustration, let $\xi_1,\ldots,\xi_n$ be i.i.d. multipliers with $\pnorm{\xi_1}{p,1}<\infty\, (p\geq 2)$, $X_i$'s be i.i.d. uniformly distributed on $[0,1]$, and $\mathcal{F}$ be a uniformly bounded function class on $[0,1]$ satisfying the entropy condition (F) with $\alpha\in (0,2)$.
	We apply (\ref{ineq:mutiplier_ineq_conclusion_2}) with $\mathcal{F}(n^{-1/(2+\alpha)})$ (note that $n^{-1/(2+\alpha)}$ is the usual local radius for $1/\alpha$-smooth problems) and local maximal inequalities for the empirical process (Proposition \ref{prop:local_maximal_ineq} in Section \ref{section:proof_multiplier_ineq} below) to see that
	\begin{align}\label{ineq:nonasymp_loose}
	\E \biggpnorm{\frac{1}{\sqrt{n}}\sum_{i=1}^n \xi_i f(X_i)}{\mathcal{F}(n^{-1/(2+\alpha)})}&\lesssim \inf_{1\leq n_0\leq n}n_0\cdot  n^{-1/2+1/p}+n_0^{-\frac{(2-\alpha)}{2(2+\alpha)}}\\
	&\asymp n^{-\frac{2-\alpha}{6+\alpha}\left(\frac{1}{2}-\frac{1}{p}\right)}\equiv n^{-\delta_1(\alpha,p)}.\nonumber
	\end{align}
	On the other hand, Corollary \ref{cor:local_maximal_ineq_barrier} gives the rate:
	\begin{align}\label{ineq:nonasymp_loose_1}
	\E \biggpnorm{\frac{1}{\sqrt{n}}\sum_{i=1}^n \xi_i f(X_i)}{\mathcal{F}(n^{-1/(2+\alpha)})}\lesssim n^{-\min\{\frac{2-\alpha}{2(2+\alpha)}, 1/2-1/p\}}\equiv n^{-\delta_2(\alpha,p)}.
	\end{align}
    In the above inequalities we used the following bound for the symmetrized empirical process (for illustration we only consider bracketing entropy):
    \begin{align*}
    &\E \biggpnorm{\frac{1}{\sqrt{n}}\sum_{i=1}^n \epsilon_i f(X_i)}{\mathcal{F}(n^{-1/(2+\alpha)})}\\
    & \lesssim J_{[\,]}(n^{-1/(2+\alpha)},\mathcal{F},L_2(P))\bigg(1+\frac{J_{[\,]}(n^{-1/(2+\alpha)},\mathcal{F},L_2(P))}{\sqrt{n} \cdot n^{-2/(2+\alpha)}}\bigg)\lesssim n^{\frac{2-\alpha}{2(2+\alpha)}},
    \end{align*}
    where in the last line of the above display we used
    \begin{align*}
    J_{[\,]}(n^{-1/(2+\alpha)},\mathcal{F},L_2(P))=  \int_0^{n^{-1/(2+\alpha)}} \sqrt{1+\log \mathcal{N}_{[\,]}(\epsilon,\mathcal{F},L_2(P))}\ \d{\epsilon}\lesssim n^{\frac{2-\alpha}{2(2+\alpha)}}.
    \end{align*}
    It is easily seen that the bound (\ref{ineq:nonasymp_loose}) calculated from (\ref{ineq:mutiplier_ineq_conclusion_2}) is worse than (\ref{ineq:nonasymp_loose_1}) because $\delta_1(\alpha,p)<\delta_2(\alpha,p)$ for all $\alpha \in (0,2)$ and $p\geq 2$. Moreover, if $p\geq 1+2/\alpha$, the bound (\ref{ineq:nonasymp_loose_1}) becomes $n^{-\frac{2-\alpha}{2(2+\alpha)}}$, which matches the rate for the symmetrized empirical process $\E \bigpnorm{n^{-1/2}\sum_{i=1}^n \epsilon_i f(X_i)}{\mathcal{F}(n^{-1/(2+\alpha)})}$.

\subsubsection{Asymptotic setting}
The primary application of (\ref{ineq:mutiplier_ineq_conclusion_2}) rests in studying asymptotic equicontinuity of the multiplier empirical process in the following sense. Suppose that $\mathcal{F}$ is Donsker. Then by the integrability of the empirical process (see Lemma 2.3.11 of \cite{van1996weak})\footnote{Here $\mathcal{F}_{\delta}\equiv \{f-g: f,g \in \mathcal{F}, \pnorm{f-g}{L_2(P)}\leq \delta\}$.}, $\E \pnorm{n^{-1/2}\sum_{i=1}^n \epsilon_i f(X_i)}{\mathcal{F}_{\delta}}\to 0$ as $n \to \infty$ followed by $\delta \to 0$.
Now apply (\ref{ineq:mutiplier_ineq_conclusion_2}) via $n \to \infty$, $n_0 \to \infty$ followed by $\delta \to 0$ we see that $ \E \pnorm{n^{-1/2}\sum_{i=1}^n \xi_i f(X_i)}{\mathcal{F}_{\delta}}\to 0$ as $n \to \infty$ followed by $\delta \to 0$
if $\pnorm{\xi_1}{2,1}<\infty$. This shows that $\big(n^{-1/2}\sum_{i=1}^n \xi_i f(X_i)\big)_{f \in \mathcal{F}}$ satisfies a CLT in $\ell^\infty(\mathcal{F})$ if $\mathcal{F}$ is Donsker and the $\xi_i$'s are i.i.d. with $\pnorm{\xi_1}{2,1}<\infty$. 

Our new multiplier inequality, Theorem \ref{thm:local_maximal_generic}, can also be used to study asymptotic equicontinuity of the multiplier empirical process with the help of the following lemma.

\begin{lemma}\label{lem:asymp_concavitify}
	Fix a concave function $\varphi:\R_{\geq 0}\to \R_{\geq 0}$ such that $\varphi(x) \to \infty$ as $x \to \infty$. Let $\{a_n\}\subset \R_{\geq 0}$ be such that $a_n \to 0$ as $n \to \infty$, and $\psi:\R_{\geq 0}\to \R_{\geq 0}$ be the least concave majorant of $\{(n,a_n \varphi(n))\}_{n=0}^\infty$. Then $\psi(t)/\varphi(t) \to 0$ as $t \to \infty$.
\end{lemma}

The proof of this lemma can be found in Section \ref{section:proof_remaining}. Take any sequence $\delta_n \to 0$ and let $a_n\equiv \E \pnorm{n^{-1/2}\sum_{i=1}^n \epsilon_i f(X_i)}{\mathcal{F}_{\delta_n}}$. By Lemma \ref{lem:asymp_concavitify}, the least concave majorant function $\psi:\R_{\geq 0}\to \R_{\geq 0}$ of the map $n\mapsto a_n n^{1/2}(n\geq 0)$ satisfies $\psi(t)/t^{1/2} \to 0$ as $t \to \infty$. Now an application of Theorem \ref{thm:local_maximal_generic} and the dominated convergence theorem shows that
\begin{align*}
\E\biggpnorm{\frac{1}{\sqrt{n}}\sum_{i=1}^n \xi_i f(X_i)}{\mathcal{F}_{\delta_n}}\leq 4\int_0^\infty \frac{\psi(n\Prob(\abs{\xi_1}>t))}{\sqrt{n\Prob(\abs{\xi_1}>t)}}\cdot \sqrt{\Prob(\abs{\xi_1}>t)}\ \d{t}\to 0
\end{align*}
as $n \to \infty$.

We note that the moment conditions of Theorem \ref{thm:local_maximal_generic} and \ref{thm:lower_bound_mep} have a small gap: in essence we require an $L_{p,1}$ moment in Theorem 1, while an $L_p$ moment is required in Theorem 2. In the context of multiplier CLTs discussed above, \cite{ledoux1986conditions} showed that the $L_{2,1}$ moment condition is sharp---there exists a construction of a Banach space of $X$ on which a multiplier CLT fails for $\xi X$ if $\pnorm{\xi_1}{2,1}=\infty$. It remains open in our setting if $L_{p,1}$ (or $L_p$) is the exact moment requirement.

\subsection{An impossibility result}

In this section we formally prove an impossibility result, showing that the independence assumption between the $X_i$'s and the $\xi_i$'s is crucial for Theorem \ref{thm:local_maximal_generic} and Corollary \ref{cor:local_maximal_ineq_barrier} to hold.
\begin{proposition}\label{prop:impossibility}
Let $\mathcal{X}\equiv \R$. For every triple $(\delta,\gamma,p)$ such that $\delta \in (0,1/2)$, $2<\gamma<1+1/(2\delta)$ and $2\leq p<\min\{4/\delta,2\gamma/(1+\gamma\delta)\}$, there exist $X_i$'s and $\xi_i$'s satisfying: (i) $\{(X_i,\xi_i)\}$'s are i.i.d.; (ii) $\xi_i$ is not independent from $X_i$ but $\E[\xi_1|X_1]=0$, $\pnorm{\xi_1}{p,1}<\infty$, and a sequence of function classes $\{\mathcal{F}_k\}_{k=1}^n$ defined on $\mathcal{X}$ with $\mathcal{F}_k\supset \mathcal{F}_n$ for any $1\leq k\leq n$,
such that
\begin{align*}
\E\biggpnorm{\sum_{i=1}^k \epsilon_i f(X_i)}{{\mathcal{F}}_k}\lesssim k^{1/\gamma},
\end{align*}
holds for all $1\leq k\leq n$, and that
\begin{align*}
\E \biggpnorm{\sum_{i=1}^n \xi_i f(X_i)}{{\mathcal{F}}_n}\gtrsim_p \omega(n),
\end{align*}
where $\omega(n)\geq n^{\beta}\cdot n^{\max\{1/\gamma,1/p\}}$ for some $\beta= \beta(\delta,\gamma,p)>0$. In other words, $\omega(n)$ grows faster than $n^{\max\{1/\gamma,1/p\}}$ (= the upper bound in Theorem \ref{thm:local_maximal_generic} and Corollary \ref{cor:local_maximal_ineq_barrier}) by a positive power of $n$.
\end{proposition}

Proposition \ref{prop:impossibility} is a negative result for the multiplier empirical processes in the similar vein as in \cite{alexander1985non}, but more quantitatively: there is no universal moment condition for the multipliers that yield a positive solution to Question 2 when the $X_i$'s and the $\xi_i$'s are allowed to be dependent.

\begin{remark}\label{rmk:impossible}
The basic trouble for removing the independence assumption between the $X_i$'s and the $\xi_i$'s can be seen by the following example. Let $X_i$'s be i.i.d. mean-zero random variables with a finite second moment. Then clearly $\sum_{i=1}^n X_i$ grows at a rate $\mathcal{O}_{\mathbf{P}}(n^{1/2})$ by the CLT. On the other hand, let $\xi_i =\epsilon_i X_i$ where $\epsilon_i$'s are independent Rademacher random variables. Then the multiplier sum $\sum_{i=1}^n \xi_i X_i=\sum_{i=1}^n \epsilon_i X_i^2$ may grow at a rate as fast as $\mathcal{O}_{\mathbf{P}}(n^{1-\delta})$, if $\epsilon_1X_1^2$ is in the domain of attraction of a symmetric stable law with index close to $1$.
\end{remark}

\section{Nonparametric regression: least squares estimation}\label{section:nonparametric_lse_minimax}

In this section, we apply our new multiplier inequalities in Section \ref{section:multiplier_inequality} to study the least squares estimator (LSE) (\ref{lse}) in the nonparametric regression model (\ref{model:nonparametric_reg}) when the errors $\xi_i$'s are heavy-tailed (E$'$), independent of the $X_i$'s (but need not be independent of each other), and the model satisfies the entropy condition (F).

Our results here are related to the recent ground-breaking work of Mendelson and his coauthors  \cite{lecue2016regularization,mendelson2014learning,mendelson2015local,mendelson2015aggregation}. These papers proved rate-optimality of ERM procedures under a $2+\epsilon$ moment condition on the errors, in a general structured learning framework that contains models satisfying sub-Gaussian/small-ball conditions. Their framework also allows arbitrary dependence between the errors $\xi_i$'s and the $X_i$'s. See \cite{mendelson2017extending} for some recent development. Here the reasons for our focus on the different structure---models with entropy conditions, are twofold:
\begin{itemize}
\item Entropy is a standard and well-understood notion for the complexity of a large class of models, see examples in \cite{gine2015mathematical,van1996weak}.
\item The moment condition on the errors needed to guarantee rate-optimality of the LSE in our setting is no longer a $2+\epsilon$ moment. In fact, as we will show, $p\geq 1+2/\alpha$ (cf. Theorems \ref{thm:lse_rate_optimal}-\ref{thm:lse_lower_bound}) moments are needed for such a guarantee.
\end{itemize}
The reason that we work with independent errors is more fundamental: when the errors $\xi_i$'s are allowed to be dependent on the $X_i$'s, there is no universal moment condition on the $\xi_i$'s alone that guarantees the rate-optimality of the LSE (cf. Proposition \ref{prop:impossibility_LSE}). In fact, even in the family of one-dimensional linear regression models with heteroscedastic errors of any finite $p$-th moment, the convergence rate of the LSE can be as slow as specified (cf. Remark \ref{rmk:impossible_lse}).

\subsection{Upper bound for the convergence rates of the LSE}

\begin{theorem}\label{thm:lse_rate_optimal}
	Suppose that $\xi_1,\ldots,\xi_n$ are mean-zero errors independent of $X_1,\ldots,X_n$ with the same marginal distributions, and $\pnorm{\xi_1}{p,1}<\infty$ for some $p\geq 1$. Further suppose that $\mathcal{F}$ is a $P$-centered function class (if the $\xi_i$'s are i.i.d. $\mathcal{F}$ need not be $P$-centered) such that $\mathcal{F}-f_0 \subset L_\infty(1)$ satisfies the entropy condition (F) with some $\alpha \in (0,2)$. Then the LSE $\hat{f}_n$ in (\ref{lse}) satisfies
	\begin{align}\label{ineq:rate_lse}
	\pnorm{\hat{f}_n-f_0}{L_2(P)}=\mathcal{O}_{\mathbf{P}}\big(n^{-\frac{1}{2+\alpha}} \vee n^{-\frac{1}{2}+\frac{1}{2p}}\big).
	\end{align}
    Furthermore, if $\xi_i$'s are i.i.d.  and $p\geq 2$, then (\ref{ineq:rate_lse}) holds in expectation:
    \begin{align}
    \E\pnorm{\hat{f}_n-f_0}{L_2(P)}=\mathcal{O}\big(n^{-\frac{1}{2+\alpha}} \vee n^{-\frac{1}{2}+\frac{1}{2p}}\big).
    \end{align}
\end{theorem}

	One interesting consequence of Theorem \ref{thm:lse_rate_optimal} is a convergence rate of the LSE when the errors only have a $L_{p,1}$ moment $(1<p\leq 2)$.
	\begin{corollary}
		Suppose the assumptions in Theorem \ref{thm:lse_rate_optimal} hold with $p \in (1,2]$. Then
			\begin{align*}
			\pnorm{\hat{f}_n-f_0}{L_2(P)}=\mathcal{O}_{\mathbf{P}}\big(n^{-\frac{1}{2}+\frac{1}{2p}}\big)=\mathfrak{o}_{\mathbf{P}}(1).
			\end{align*}
	\end{corollary}
	Consistency of the LSE has been a classical topic, see e.g. \cite{van1987new,van1996consistency} for sufficient and necessary conditions in this regard under a second moment assumption on the errors. Here Theorem \ref{thm:lse_rate_optimal} provides a quantitative rate of convergence of the LSE when the errors may not even have a second moment (under stronger conditions on $\mathcal{F}$).

The connection between the proof of Theorem \ref{thm:lse_rate_optimal} and the new multiplier inequality in Section \ref{section:multiplier_inequality} is the following reduction scheme. 
\begin{proposition}\label{prop:general_lse}
	Suppose that $\xi_1,\ldots,\xi_n$ are mean-zero random variables independent of $X_1,\ldots,X_n$, and $\mathcal{F}-f_0\subset L_\infty(1)$.  Further assume that
	\begin{align}\label{cond:thm_general_lse_1}
	\E \sup_{f \in \mathcal{F}:\pnorm{f-f_0}{L_2(P)}\leq \delta} \biggabs{\frac{1}{\sqrt{n}}\sum_{i=1}^n \xi_i(f-f_0)(X_i) }\lesssim \phi_n(\delta),
	\end{align}
	and
	\begin{align}\label{cond:thm_general_lse_2}
	\E \sup_{f \in \mathcal{F}:\pnorm{f-f_0}{L_2(P)}\leq \delta} \biggabs{\frac{1}{\sqrt{n}}\sum_{i=1}^n \epsilon_i(f-f_0)(X_i) }\lesssim \phi_n(\delta).
	\end{align}
	hold for some $\phi_n$ such that $\delta\mapsto \phi_n(\delta)/\delta$ is non-increasing. 
	Then  $\pnorm{\hat{f}_n-f_0}{L_2(P)}=\mathcal{O}_{\mathbf{P}}(\delta_n)$ holds for any $\delta_n$ such that $\phi_n(\delta_n)\leq \sqrt{n}\delta_n^2$. 
	Furthermore, if $\xi_1,\ldots,\xi_n$ are i.i.d. mean-zero with $\pnorm{\xi_1}{p}<\infty$ for some $p\geq 2$, then $\E \pnorm{\hat{f}_n-f_0}{L_2(P)} =\mathcal{O}(\delta_n)$ for any $\delta_n\geq n^{-\frac{1}{2}+\frac{1}{2p}}$ such that $\phi_n(\delta_n)\leq \sqrt{n}\delta_n^2$.
\end{proposition}
The remaining task in the proof of Theorem \ref{thm:lse_rate_optimal} is a calculation of the modulus of continuity of the (multiplier) empirical process involved in (\ref{cond:thm_general_lse_1}) and (\ref{cond:thm_general_lse_2}) using Theorem \ref{thm:local_maximal_generic} and local maximal inequalities for the empirical process (see Proposition \ref{prop:local_maximal_ineq}).

\begin{remark}
	Some remarks on the assumptions on $\mathcal{F}$.
	\begin{enumerate}
		\item The entropy condition (F) is standard in nonparametric statistics literature. The condition $\alpha \in (0,2)$ additionally requires $\mathcal{F}$ to be a \emph{Donsker} class. Although the proof applies to non-Donsker function classes with $\alpha\geq 2$, the first term in (\ref{ineq:rate_lse}) becomes \emph{sub-optimal} in general, see \cite{birge1993rates}. 
		\item $\mathcal{F}$ is assumed to be $P$-centered when the errors $\xi_i$'s have an arbitrary dependence structure. It is known from \cite{yang2001nonparametric} (see Theorem 1, page 638) that for a centered function class, the minimax risk of estimating a regression function under arbitrary errors with second moments uniformly bounded, is no worse than that for i.i.d. Gaussian errors.  If the errors are i.i.d., then $\mathcal{F}$ need not be $P$-centered (as stated in the theorem). 
		\item The uniform boundedness assumption on $\mathcal{F}$,  including many classical examples (cf. Section 9.3 of \cite{van2000empirical}),  should be primarily viewed as \emph{a method of proof}: all that we need is $\pnorm{\hat{f}_n}{\infty}=\mathcal{O}_{\mathbf{P}}(1)$. In subsequent work of the authors \cite{han2017robustness}, this method is applied to shape-restricted regression problems in a heavy-tailed regression setting.
	\end{enumerate}
\end{remark}

\begin{remark}
Here in Theorem \ref{thm:lse_rate_optimal} we focus on the regression model (\ref{model:nonparametric_reg}) with errors $\xi_i$'s independent from $X_i$'s. This is crucial: we show below in Proposition \ref{prop:impossibility_LSE} that the independence assumption between the $X_i$'s and $\xi_i$'s cannot be relaxed for the rate in Theorem \ref{thm:lse_rate_optimal} to hold. 

On the other hand, our Theorem \ref{thm:lse_rate_optimal} is useful in handling centered models with arbitrarily dependent errors in the regression model. This complements Mendelson's work \cite{lecue2016regularization,mendelson2014learning,mendelson2015local,mendelson2015aggregation,mendelson2016multiplier} that allows arbitrary dependence between $\xi_i$ and $X_i$'s with independent observations in a learning framework.
\end{remark}

\begin{remark}
In Theorem \ref{thm:lse_rate_optimal} the results are `in probability' and `in expectation' statements. It is easy to see from the proof that a tail estimate can be obtained for $\pnorm{\hat{f}_n-f_0}{L_2(P)}$: if $\pnorm{\xi_1}{p,1}<\infty$ for some $p\geq 2$, then
\begin{align*}
\Prob\big(\delta_n^{-1} \pnorm{\hat{f}_n-f_0}{L_2(P)}>t\big)\leq Ct^{-p},
\end{align*}
where $\delta_n\equiv n^{-\frac{1}{2+\alpha}}\vee n^{-\frac{1}{2}+\frac{1}{2p}}$. Constructing estimators other than the LSE that give rise to \emph{exponential tail bound} under a heavy-tailed regression setting is also of significant interest. We refer the readers to, e.g. \cite{devroye2016subgaussian,lugosi2017regularization,lugosi2017sub} and references therein for this line of research.
\end{remark}

\subsection{Lower bound for the convergence rates of the LSE}

At this point, (\ref{ineq:rate_lse}) only serves as an \emph{upper bound} for the convergence rates of the LSE. Since the rate $n^{-\frac{1}{2+\alpha}}$ corresponds to the optimal rate in the Gaussian regression case \cite{yang1999information}, it is natural to conjecture that this rate cannot be improved. On the other hand, the `noise' rate $n^{-\frac{1}{2}+\frac{1}{2p}}$ is due to the reduction scheme in Proposition \ref{prop:general_lse}, which relates the convergence rate of the LSE to the size of the multiplier empirical process involved. It is natural to wonder if this `noise rate' is a proof artifact due to some possible deficiency in Proposition \ref{prop:general_lse}.

\begin{theorem}\label{thm:lse_lower_bound}
	Let $\mathcal{X}=[0,1]$ and $P$ be a probability measure on $\mathcal{X}$ with Lebesgue density bounded away from $0$ and $\infty$, and $\xi_i$'s are i.i.d. mean-zero errors independent of $X_i$'s. Then for each $\alpha \in (0,2)$ and $2 \vee \sqrt{\log n}\leq p \leq (\log n)^{1-\delta}$ with some $\delta \in (0,1/2)$, there exists a function class $\mathcal{F}\equiv \mathcal{F}_n$, and some $f_0 \in \mathcal{F}$ with $\mathcal{F}-f_0$ satisfying the entropy condition (F), such that the following holds: 
	there exists some law for the error $\xi_1$ with $\pnorm{\xi_1}{p,1}\lesssim \log n$, such that for $n$ sufficiently large, there exists some least squares estimator $f_n^\ast$ over $\mathcal{F}_n$ satisfying
	\begin{align*}
	\E \pnorm{f^\ast_n-f_0}{L_2(P)}\geq \rho \cdot \big(n^{-\frac{1}{2+\alpha}}\vee n^{-\frac{1}{2}+\frac{1}{2p}}\big)(\log n)^{-2}.
	\end{align*}
	Here $\rho>0$ is a (small) constant independent of $n$.
\end{theorem}

Theorem \ref{thm:lse_lower_bound} has two claims. The first claim justifies the heuristic conjecture that the convergence rate for the LSE with heavy-tailed errors under entropy conditions, should be no better than the optimal rate in the Gaussian regression setting. Although here we give an existence statement, the proof is constructive: in fact we use (essentially) a H\"older class. Other function classes are also possible if we can handle the Poisson (small-sample) domain of the empirical process indexed by these classes. 

The second claim asserts that for any entropy level $\alpha \in (0,2)$, there exist `hard models' for which the noise level dominates the risk for the least squares estimator. Here are some examples for these hard models:

\begin{example}
A benchmark model witnessing the worst case rate $\mathcal{O}(n^{-\frac{1}{2}+\frac{1}{2p}})$ (up to logarithmic factors) is (almost) the one we used in Theorem \ref{thm:lower_bound_mep}, i.e. the class of indicators \footnote{\label{note_1}excluding the indicators indexed by intervals that are too short.} over closed intervals in $[0,1]$. 
\end{example} 

\begin{example}
Consider more general classes $^4$
\begin{align*}
\mathcal{F}_k\equiv \bigg\{&\sum_{i=1}^k c_i \bm{1}_{[x_{i-1},x_{i}]}: \abs{c_i}\leq 1,\\
& \quad 0\leq x_0<x_1<\ldots<x_{k-1}<x_k\leq 1\bigg\}, k\geq 1.
\end{align*}
The classes $\mathcal{F}_k$ also witness the worst case rate $\mathcal{O}(n^{-\frac{1}{2}+\frac{1}{2p}})$ (up to logarithmic factors) since they contain all indicators over closed intervals on $[0,1]$, and are closely related to problems in the change-point estimation/detection literature. For instance, the case $k=1$ is of particular importance in epidemic and signal processing applications; see \cite{arias2005near,yao1993tests} from a testing perspective of the problem. From an estimation viewpoint, \cite{boysen2009consistencies} proposed an $\ell_0$-type penalized LSE for estimating regression functions in $\mathcal{F}_k$, where a (nearly) parametric rate is obtained under a sub-Gaussian condition on the errors. Our results here suggest that such least-squares type estimators may not work well for estimating step functions with multiple change-points if the errors are heavy-tailed.
\end{example}

\begin{example}
Yet another class is given by the regression problem involving image restoration (or edge estimation), see e.g. \cite{korostelev1992asymtotically,korestelev1993minimax} or Example 9.3.7 of \cite{van2000empirical} (but we consider a random design). In particular, the class $\mathscr{C}\equiv \{\bm{1}_C: C\subset [0,1]^d\textrm{ is convex}\}$\footnote{excluding the indicators indexed by sets with too small volume.} also witnesses the lower bound $\mathcal{O}(n^{-\frac{1}{2}+\frac{1}{2p}})$ (up to logarithmic factors) since it contains all indicators over hypercubes on $[0,1]^d$.
\end{example}

\subsection{Some positive and negative implications for the LSE}

\begin{figure}
	\begin{tikzpicture}[xscale=4,yscale=0.25]
	\draw[-][draw=black, thick] (0,0) -- (2.3,0);
	\draw[-][draw=black, thick] (0,0)--(0,11);
	\draw[dashed, thick] (0,2) -- (2,2);
	\draw[dashed, thick] (2,0) -- (2,2);
	\draw[purple, ultra thick, domain=0.2:2] plot (\x, {1+2/(\x)} );
	\node [above] at (1.2,6) {Gaussian rate =$n^{-\frac{1}{2+\alpha}}$};
	\node [above] at (0.7,0) {noise rate=$n^{-\frac{1}{2}+\frac{1}{2p}}$};
	\node [below] at (2,-0.1) {2};
	\node [left] at (-0.1,2) {2};
	\node [below] at (2.3,-0.1) {$\alpha$};
	\node [left] at (-0.1,11) {$p$};
	\end{tikzpicture}
	\caption{Tradeoff between the complexity of the function class and the noise level of the errors in the convergence rates for the LSE. The critical curve (purple): $p=1+2/\alpha$. }
	\label{fig: phase_trans_lse}
\end{figure}
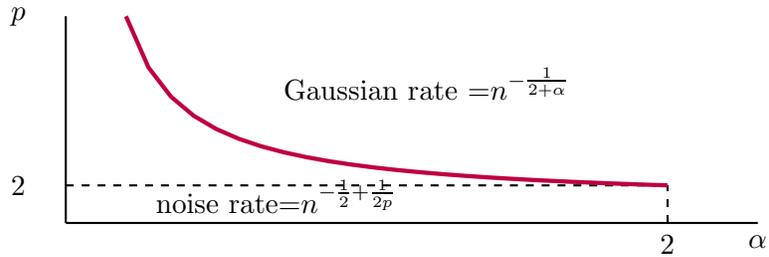

Combining Theorems \ref{thm:lse_rate_optimal} and \ref{thm:lse_lower_bound}, we see that the tradeoff in the size of the multiplier empirical process between the complexity of the function class and the heaviness of the tail of the errors (multipliers) tranlates into the convergence rate of the LSE (cf. Figure \ref{fig: phase_trans_lse}). In particular, Theorems \ref{thm:lse_rate_optimal} and \ref{thm:lse_lower_bound} indicate both some positive and negative aspects of the LSE in a heavy-tailed regression setting:\newline

\noindent \textbf{(Positive implications for the LSE):} 

If $p\geq 1+2/\alpha$, then $\pnorm{\hat{f}_n-f_0}{L_2(P)}=\mathcal{O}_{\mathbf{P}}\big(n^{-\frac{1}{2+\alpha}} \big)$. In this case, the noise level is `small' compared with the complexity of the function class so that the LSE achieves the optimal rate as in the case for i.i.d. Gaussian errors (see \cite{yang1999information}).\newline

\noindent \textbf{(Negative implications for the LSE):}  

If $p<1+2/\alpha$, then $\pnorm{\hat{f}_n-f_0}{L_2(P)}=\mathcal{O}_{\mathbf{P}}\big( n^{-\frac{1}{2}+\frac{1}{2p}}\big)$. In this case, the noise is so heavy-tailed that the \emph{worst-case} rate of convergence of the LSE is governed by this noise rate (see above for examples). The negative aspect of the LSE is that this noise rate reflects a genuine deficiency of the LSE as an estimation procedure, rather than the difficulty due to the `hard model' in such a heavy-tailed regression setting. In fact, we can design simple robust procedures to outperform the LSE in terms of the rate of convergence.

To see this, consider the least-absolute-deviation(LAD) estimator $\tilde{f}_n$ (see e.g. \cite{gao2010asymptotic,pollard1991asymptotics,portnoy1997gaussian}, or page 336 of \cite{van1996weak}) defined by $
\tilde{f}_n =\arg\min_{f \in \mathcal{F}} \frac{1}{n}\sum_{i=1}^n \abs{Y_i-f(X_i)}$.
It follows from a minor modification of the proof \footnote{More specifically, we can proceed by replacing the empirical measure $\Prob_n$ by $P$, slightly restricting the suprema of the empirical process to $1/n\lesssim P(f-f_0)^2<\delta^2$ in the third display on page 336 of \cite{van1996weak}, and noting that Theorem 3.4.1 of \cite{van1996weak} can be strengthened to an expectation since the empirical processes involved are bounded.} in page 336 of \cite{van1996weak} that as long as the errors $\xi_i\equiv M\eta_i$'s for some $\eta_i$ admitting a smooth enough density, median zero and a first moment, and $M>0$ not too small, then under the same conditions as in Theorem \ref{thm:lse_rate_optimal}, the LAD estimator $\tilde{f}_n$ satisfies
\begin{align*}
\sup_{f_0 \in \mathcal{\mathcal{F}}}\E_{f_0} \pnorm{\tilde{f}_n-f_0}{L_2(P)}\leq  \mathcal{O}\big(n^{-\frac{1}{2+\alpha}} \big),
\end{align*}
where clearly the noise rate $\mathcal{O}(n^{-\frac{1}{2}+\frac{1}{2p}})$ induced by the moment of the errors does not occur. For statistically optimal procedures that do not even require a first moment on the errors, we refer the reader to \cite{baraud2017new}. 

It is worthwhile to note that the shortcomings of the LSE quantified here also rigorously justify the motivation of developing other robust procedures (cf. \cite{audibert2011robust,brownlees2015empirical,bubeck2013bandits,catoni2012challenging,catoni2016pac,devroye2016subgaussian,hsu2016loss,joly2017estimation,lugosi2017regularization,lugosi2016risk,lugosi2017sub,minsker2015geometric}).

\begin{remark}
Our Theorems \ref{thm:lse_rate_optimal} and \ref{thm:lse_lower_bound} show that the moment condition
\begin{align*}
p\geq 1+2/\alpha
\end{align*}
that guarantees the LSE to converge at the optimal rate (as in the case for Gaussian errors), is the best one can hope \emph{under entropy conditions alone}. On the other hand, this condition may be further improved if additional structure is available. For instance, in the isotonic regression case $(\alpha=1)$, our theory requires $p\geq 3$ to guarantee an optimal $n^{-1/3}$ rate for the isotonic LSE, while it is known (cf. \cite{zhang2002risk}) that a second moment assumption on the errors $(p=2)$ suffices. The benefits of this extra structure due to shape constraints are investigated in further work by the authors \cite{han2017robustness}.
\end{remark}

\subsection{An impossiblility result}
In this section, dual to the impossibility result in Proposition \ref{prop:impossibility} for the multiplier empirical process, we formally prove that the independence assumption between the $X_i$'s and the $\xi_i$'s is necessary for the rate in Theorems \ref{thm:lse_rate_optimal} and \ref{thm:lse_lower_bound} to hold.
\begin{proposition}\label{prop:impossibility_LSE}
Consider the regression model (\ref{model:nonparametric_reg}) without assuming independence between the $X_i$'s and the $\xi_i$'s. Let $\mathcal{X}\equiv \R$. For every triple $(\delta,\alpha,p)$ such that $\delta \in (0,1/2)$, $4\delta<\alpha<2$ and $2\leq p< \min\{4/\delta, (2+4/\alpha)/(1+(1+2/\alpha)\delta)\}$, there exist
\begin{itemize}
	\item $X_i$'s and $\xi_i$'s satisfying: (i) $\{(X_i,\xi_i)\}$'s are i.i.d.; (ii) $\xi_i$ is not independent from $X_i$ but $\E[\xi_1|X_1]=0$, $\pnorm{\xi_1}{p,1}<\infty$;
	\item a function class $\mathcal{F}\equiv \mathcal{F}_n$, and some $f_0 \in \mathcal{F}$ with $\mathcal{F}-f_0$ satisfying the entropy condition (F),
\end{itemize}
such that the following holds: 
for $n$ sufficiently large, there exists some least squares estimator $f_n^\ast$ over $\mathcal{F}_n$ satisfying
\begin{align*}
\E \pnorm{f^\ast_n-f_0}{L_2(P)}\geq \delta_n
\end{align*}
where $\delta_n \geq n^{\beta}\cdot (n^{-1/(2+\alpha)}\vee n^{-1/2+1/(2p)})$ for some $\beta =\beta(\delta,\alpha,p)>0$. In other words, $\delta_n$ shrinks to $0$ slower than $n^{-1/(2+\alpha)}\vee n^{-1/2+1/(2p)}$ (= the rate of the LSE in Theorems \ref{thm:lse_rate_optimal} and \ref{thm:lse_lower_bound}) by a positive power of $n$.
\end{proposition}

Proposition \ref{prop:impossibility_LSE} is a negative result on the LSE: there is no universal moment condition on $\xi_i$'s that guarantees the rate-optimality of the LSE when the errors $\xi_i$'s can be dependent on the $X_i$'s. 

\begin{remark}\label{rmk:impossible_lse}
	One basic model underlying the construction of Proposition \ref{prop:impossibility_LSE} is the following: consider the (one-dimensional) linear regression model with heteroscedastic errors
	\begin{align*}
	Y_i = \alpha_0 X_i+\xi_i,\quad i=1,\ldots,n
	\end{align*}
	where $\xi_i = \epsilon_i X_i$ for some independent Rademacher random variables $\epsilon_i$'s. Clearly $\E[\xi_i|X_i]=0$, but $\xi_i$ is (highly) dependent on $X_i$. The least squares estimator $\hat{\alpha}_n\equiv \arg\min_{\alpha \in \R} n^{-1}\sum_{i=1}^n(Y_i-\alpha X_i)^2$ has a closed form:
	\begin{align*}
	\hat{\alpha}_n\equiv \frac{\sum_{i=1}^n X_i Y_i}{\sum_{i=1}^n X_i^2}=\alpha_0+ \frac{\sum_{i=1}^n \epsilon_i X_i^2}{\sum_{i=1}^n X_i^2}.
	\end{align*}
	Suppose $X_i$'s have a finite second moment, then by the SLLN, $\hat{\alpha}_n\to \alpha_0$ a.s., but the convergence rate of $\pnorm{\hat{f}_n-f_0}{L_2(P)}=\abs{\hat{\alpha}_n-\alpha_0}\pnorm{X_1}{2}$ can be as slow as any $n^{-\delta}$: note that $\sum_{i=1}^n X_i^2 = \mathcal{O}(n)$ under the assumed second moment condition on $X_i$'s, while the sum of the centered random variables $\sum_{i=1}^n \epsilon_i X_i^2$ may have a growth rate $\mathcal{O}(n^{1-\delta})$ if $\epsilon_1X_1^2$ is in the domain of attraction of a symmetric stable law with index close to 1 (recall Remark \ref{rmk:impossible}). 
	
	A simple modification of the construction along the lines of the proof of Proposition \ref{prop:impossibility} allows the situation where $\xi_i$'s have a finite $p$-th moment ($p\geq 2$), while the convergence rate of the LSE can be as slow as $n^{-\delta}$.
	
	So in order to derive the rate-optimality of the LSE under any universal moment condition on the errors $\xi_i$'s, in a framework that allows arbitrary dependence between the $\xi_i$'s and the $X_i$'s, it is necessary to impose conditions on the model $\mathcal{F}$ to exclude the counter-examples (as in \cite{lecue2016regularization,mendelson2014learning,mendelson2015local,mendelson2015aggregation,mendelson2016multiplier}).
	
\end{remark}

\section{ Sparse linear regression: Lasso revisited}\label{section:sparse_linear_reg}

In this section we consider the sparse linear regression model:
\begin{align}\label{def:linear_model}
Y=X\theta_0 +\xi
\end{align}
where $X \in \R^{n\times d}$ is a (random) design matrix and $\xi=(\xi_1,\ldots,\xi_n)$ is a mean-zero noise vector independent of $X$. When the true signal $\theta_0 \in \R^d$ is sparse, one popular estimator is the Lasso \cite{tibshirani1996regression}:
\begin{align}\label{def:lasso}
\hat{\theta}(\lambda)\equiv \arg\min_{\theta \in \R^d} \bigg(\frac{1}{n}\pnorm{Y-X\theta}{2}^2+\lambda \pnorm{\theta}{1}\bigg).
\end{align}
The lasso estimator has been thoroughly studied in an already vast literature; we refer readers to the monograph \cite{buhlmann2011statistics} for a comprehensive overview. 

Our main interest here concerns the following question: \emph{under what moment conditions on the distributions of $X$ and $\xi$ can the lasso estimator enjoy the optimal rate of convergence?} In particular, neither $X$ nor $\xi$ need be light tailed apriori (i.e. not sub-Gaussian), and the components $\xi_1,\ldots,\xi_n$ of the vector $\xi$ need not be independent.  

Previous work guaranteeing rate-optimality of the Lasso estimator typically assumes that both $X$ and $\xi$ are sub-Gaussian, see \cite{buhlmann2011statistics,nickl2013confidence,van2014asymptotically}. 
Relaxing the sub-Gaussian conditions in the Lasso problem is challenging: \cite{lecue2016regularization} showed how to remove the sub-Gaussian assumption on $\xi$ in the case $X$ is sub-Gaussian.
The problem is even more challenging if we relax the sub-Gaussian assumption on the design matrix $X$. Our goal in this section is to demonstrate how the new multiplier inequality in Theorem \ref{thm:local_maximal_generic}, combined with (essentially) existing techniques, can be used to give a systematic treatment to the above question, in a rather straightforward fashion.

Before stating the result, we need some notion of the \emph{compatibility condition}: For any $L>0$ and $S\subset \{1,\ldots, d\}$, define
\begin{align*}
\phi(L,S) = \sqrt{\abs{S}}\min \left\{ \frac{1}{\sqrt{n}}\pnorm{X\theta_S - X\theta_{S^c}}{2}:\pnorm{\theta_S}{1}=1, \pnorm{\theta_{S^c}}{1}\leq L\right\}.
\end{align*}
Here for any $\theta=(\theta_i)\in \R^d$, $\theta_S\equiv (\theta_i\bm{1}_{i \in S})$ and $\theta_{S^c}\equiv (\theta_i \bm{1}_{i \notin S})$. Let $B_0(s)$ be the set of $s$-sparse vectors in $\R^d$, i.e. $\theta \in B_0(s)$ if and only if $\abs{\{i:\theta_i\neq 0\}}\leq s$.  Further let $\Sigma = \E \hat{\Sigma}$ where $\hat{\Sigma}=X^\top X/n$ is the sample covariance matrix, and $\underline{\sigma}_d = \sigma_{\mathrm{min}}(\Sigma)$ and $\bar{\sigma}_d=\sigma_{\mathrm{max}}(\Sigma)$ be the smallest and largest singular value of the population covariance matrix, respectively. Here $d=d_n$ and $s=s_n$ can either stay bounded or blow up to infinity in asymptotic statements.

\begin{theorem}\label{thm:lasso_l2_uniform}
	Let $X$ be a design matrix with i.i.d. mean-zero rows, and $0<\liminf \underline{\sigma}_d\leq \limsup\bar{\sigma}_d<\infty$. Suppose that 
	\begin{align}\label{cond:lasso_re}
	\min_{\abs{S}\leq s} \phi(3,S)\geq c_0
	\end{align}
	holds for some $c_0>0$ with probability tending to $1$ as $n \to \infty$, and that  for some $1/4\leq \alpha\leq 1/2$, 
	\begin{align}\label{cond:lasso}
	\limsup_{n \to \infty} \frac{\log d\cdot \left(M_4({X}) \vee  \log^2 d\right)}{n^{2-4\alpha}}  <\infty,
	\end{align}
	where $M_4({X}) \equiv \E \max_{1\leq j\leq d}\abs{X_{1j}}^4$. Then for $\hat{\theta}^L\equiv \hat{\theta}(2L\pnorm{\bm{\xi}_n}{1/\alpha,1}\sqrt{\log d/n})$, 
	\begin{align}\label{conc:lasso}
	\lim_{L \to \infty}\limsup_{ n\to \infty} \sup_{\theta_0 \in B_0(s)}&\Prob_{\theta_0} \bigg(\frac{1}{n}\pnorm{X(\hat{\theta}^L-\theta_0)}{2}^2 > \frac{16L^2\pnorm{\bm{\xi}_n}{1/\alpha,1}^2}{c_0^2}\cdot \frac{s \log d}{n}\bigg)=0.
	\end{align}
	Here $\pnorm{\bm{\xi}_n}{1/\alpha,1}\equiv \int_0^\infty \big(\frac{1}{n}\sum_{i=1}^n \Prob(\abs{\xi_i}>t)\big)^\alpha \ \d{t}$.
\end{theorem}

The rate $\sqrt{s\log d/n}$ in the above theorem is well-known to be (nearly) minimax optimal for prediction in the sparse linear regression model (e.g. \cite{raskutti2011minimax}). The quantity $\pnorm{\bm{\xi}_n}{1/\alpha,1}$ should be thought as the `noise level' of the regression problem. For instance, if the $\xi_i$'s are i.i.d, and $\alpha=1/2$, then $\pnorm{\bm{\xi}_n}{1/\alpha,1}=\pnorm{\xi_1}{2,1}$.

Although in Theorem \ref{thm:lasso_l2_uniform} we only consider prediction error, the estimation error $\pnorm{\hat{\theta}^L-\theta_0}{1}$ can be obtained using completely similar arguments by noting that Lemma \ref{lem:lasso_l2_bound_compatibility} below also holds for estimation error.

\begin{remark}
Two technical remarks.
\begin{enumerate}
\item As in Theorem \ref{thm:lse_rate_optimal}, we assume in Theorem \ref{thm:lasso_l2_uniform} that the rows of $X$ have zero-mean as vectors in $\R^d$ so that arbitrary dependence structure among $\xi_i$'s can be allowed. For i.i.d. errors, the zero-mean assumption is not needed.
\item (\ref{conc:lasso}) is of an asymptotic nature mainly due to the weak asymptotic assumptions made in (\ref{cond:lasso_re}) and (\ref{cond:lasso}). It is clear from the proof that concrete probability estimates can be obtained if a probability estimate for (\ref{cond:lasso_re}) is available.
\end{enumerate}
\end{remark}

As an illustration of the scope of Theorem \ref{thm:lasso_l2_uniform}, we consider several different scaling regimes for the parameter space $(d,n,s)$. For simplicity of discussion we assume that the errors $\xi_1,\ldots,\xi_n$ have the same marginal distributions and the design matrix $X$ has i.i.d. entries such that $X_{11}$ has a Lebesgue density bounded away from $\infty$ and $\E X_{11}^2=1$.

\begin{example}\label{ex:lasso_moderate_dim}
	Consider the scaling regime $d/n \to \lambda \in (0,1)$. We claim that $\E \abs{X_{11}}^{4+\epsilon}\vee \pnorm{\xi}{4,1}<\infty$ for some $\epsilon>0$ guarantees the validity of (\ref{conc:lasso}). First, (\ref{cond:lasso_re}) holds under the finite fourth moment condition, see \cite{bai1993limit}. Second, (\ref{cond:lasso}) holds under the assumed moment conditions. Note that a fourth moment condition on $X_{11}$ is necessary: if $\E X_{11}^4=\infty$, then $\limsup \bar{\sigma}_d=\infty$ a.s., see \cite{bai1988note}. This corollary of Theorem \ref{thm:lasso_l2_uniform} appears to be a new result; \cite{lugosi2017regularization} considered a different `tournament' Lasso estimator with best tradeoff between confidence statement and convergence rate under heavy-tailed designs and errors.
\end{example}

\begin{example}
	If $\pnorm{X_{11}}{p}\lesssim p^\beta$ for some $\beta\geq 1/2$ and all $p\lesssim \log n$, then Theorem E of \cite{lecue2014sparse} showed that the compatibility condition (\ref{cond:lasso_re}) holds under $n\gtrsim s\log d\vee (\log d)^{(4\beta-1)}$. Condition (\ref{cond:lasso}) is satisfied if $\pnorm{\xi}{2+\epsilon}<\infty$ and $\log d\lesssim \log n$. 
	
    The condition $\log d\lesssim \log n$ requires polynomial growth of $d$ with $n$; this can be improved if $X_{11}$ is light tailed. In particular, if $\E \exp(\mu\abs{X_{11}}^\gamma)<\infty$ for some $\mu,\gamma>0$, then we can take $\beta = 1/\gamma$ so that (\ref{cond:lasso_re}) holds under $n\gtrsim s\log d \vee (\log d)^{(4/\gamma)-1}$, while (\ref{cond:lasso}) is satisfied if $\pnorm{\xi}{2+\epsilon}<\infty$ and $d\leq \exp(n^{c_{\epsilon,\gamma}})$ for some constant $c_{\epsilon,\gamma}>0$. Different choices of $\gamma$ lead to:
    \begin{itemize}
    	\item If the entries of $X$ have sub-exponential tails, then we may take $\gamma=1$. In this case, (\ref{conc:lasso}) is valid under $\pnorm{\xi}{2+\epsilon}<\infty$ subject to $n\gtrsim s\log d\vee \log^3 d$ and  $d\leq \exp(n^{c_{\epsilon,1}})$ for some constant $c_{\epsilon,1}>0$. This seems to be a new result; the recent result of \cite{sivakumar2015beyond} considered the similar tail condition on $X$ along with a sub-exponential tail for the errors $\xi_i$'s, while their rates come with additional logarithmic factors.
    	\item If the entries of $X$ have sub-Gaussian tails, then we may take $\gamma=2$. In this case, (\ref{conc:lasso}) is valid under $\pnorm{\xi}{2+\epsilon}<\infty$ subject to $n\gtrsim s \log d$ and  $d\leq \exp(n^{c_{\epsilon,2}})$ for some constant $c_{\epsilon,2}>0$. This recovers a recent result of \cite{lecue2016regularization} in the case where $X$ and $\xi$ are independent (up to the mild dimension constraint on $d$). 
    \end{itemize}
\end{example}

Now we prove Theorem \ref{thm:lasso_l2_uniform}. The following reduction (basic inequality) is well-known, cf. Theorem 6.1 of \cite{buhlmann2011statistics}.

\begin{lemma}\label{lem:lasso_l2_bound_compatibility}
	On the event $
	\mathcal{E}_L\equiv \{\max_{1\leq j\leq d}\abs{\frac{2}{n}\sum_{i=1}^n \xi_i X_{ij}}\leq L\sqrt{\log d/n}\}$,
	with tuning parameter $\lambda\equiv 2L\sqrt{\log d/n}$, it holds that $
	n^{-1}\pnorm{X(\hat{\theta}^L-\theta_0)}{2}^2\leq 16 L^2 \phi^{-2}(3,S_0)\cdot s_0\log d/n$
	where $S_0=\{i:(\theta_0)_i\neq 0\}$ and $s_0=\abs{S_0}$.
\end{lemma}

The difficulty involved here is that \emph{both} $X$ and $\xi$ can be heavy tailed. By Theorem \ref{thm:local_maximal_generic}, to account for the effect of the $\xi_i$'s, we only need to track the size of
$\E\max_{1\leq j\leq d}\abs{\sum_{i=1}^k \epsilon_i X_{ij}}$
at each scale $k\leq n$. This is the content of the following Gaussian approximation lemma.
\begin{lemma}\label{lem:gaussian_approx}
	Let $X_1,\ldots,X_n$ be i.i.d. random vectors in $\R^d$ with covariance matrix $\Sigma$. If $\sup_d \sigma_{\max}(\Sigma)<\infty$,
	then for all $k,d \in \N$, 
	\begin{align*}
	\E \max_{1\leq j\leq d} \biggabs{\sum_{i=1}^k \epsilon_i X_{ij}}\lesssim \left( k\log^3 d\cdot \big(M_4({X})\vee \log^2 d\big)\right)^{1/4}+(k\log d)^{1/2}.
	\end{align*}
\end{lemma}
The proof of the lemma is inspired by the recent work \cite{chernozhukov2013gaussian} who considered Gaussian approximation of the maxima of high-dimensional random vectors by exploiting \emph{second} moment information for the $X_i$'s. We modify their method by taking into account the \emph{third} moment information of $X_i$'s induced by the symmetric Rademacher $\epsilon_i$'s; such a modification proves useful in identifying certain sharp moment conditions considered in the examples (in particular Example \ref{ex:lasso_moderate_dim}). See Section \ref{section:proof_gaussian_approx} for a detailed proof.

\begin{proof}[Proof of Theorem \ref{thm:lasso_l2_uniform}]
	By Lemma \ref{lem:lasso_l2_bound_compatibility} and the assumption on the compatibility condition (\ref{cond:lasso_re}), we see that with the choice for tuning parameter $\lambda\equiv 2L\pnorm{\bm{\xi}_n}{1/\alpha,1}\sqrt{\log d/n}$, the left side of (\ref{conc:lasso}) can be bounded by
	\begin{align}\label{ineq:lasso_l2_1}
	& \Prob_{\theta_0}\left(\frac{1}{n}\pnorm{X(\hat{\theta}^L-\theta_0)}{2}^2> \frac{16 L^2 \pnorm{\bm{\xi}_n}{1/\alpha,1}^2 }{\phi^2(3,S_0)}\cdot \frac{s \log d}{n}\right)+\mathfrak{o}(1)\\
	& \leq  \Prob\bigg(\max_{1\leq j\leq d}\biggabs{\frac{2}{n}\sum_{i=1}^n \xi_i X_{ij}}> L \pnorm{\bm{\xi}_n}{1/\alpha,1}\sqrt{\frac{\log d}{n}}\bigg)+\mathfrak{o}(1).\nonumber
	\end{align}
	By Lemma \ref{lem:gaussian_approx}, we can apply Theorem \ref{thm:local_maximal_generic} with $\mathcal{F}_1=\cdots=\mathcal{F}_n\equiv \{\pi_j : \R^d \to \R, j=1,\ldots,d\}$ where $\pi_j(x)=x_j$ for any $x=(x_l)_{l=1}^d \in \R^d$, and
	\begin{align*}
	\psi_n(k)\equiv C\left(  k^{\alpha} \left(\log^3 d\cdot \big(M_4\vee \log^2 d\big)\right)^{1/4}+k^{1/2}\sqrt{\log d} \right)
	\end{align*}
	for any $1/4\leq \alpha\leq 1/2$ such that (\ref{cond:lasso}) holds and $\pnorm{\bm{\xi}_n}{1/\alpha,1}<\infty$, to conclude that
	\begin{align*}
	\E \max_{1\leq j\leq d} \biggabs{\sum_{i=1}^n \xi_i X_{ij}} &\lesssim n^\alpha \left(\log^3 d\cdot \big(M_4\vee \log^2 d\big)\right)^{1/4} \pnorm{\bm{\xi}_n}{1/\alpha,1} +n^{1/2}\sqrt{\log d}  \pnorm{\bm{\xi}_n}{2,1} \\
	&\lesssim \left(n^\alpha \left(\log^3 d\cdot \big(M_4\vee \log^2 d\big)\right)^{1/4}+n^{1/2}\sqrt{\log d}\right)\pnorm{\bm{\xi}_n}{1/\alpha,1}.
	\end{align*}
	By Markov's inequality, (\ref{ineq:lasso_l2_1}) can be further bounded (up to constants) by 
	\begin{align*}
	\frac{1}{L} \left(\frac{\log d\cdot (M_4 \vee \log^2 d)}{n^{2-4\alpha}}\vee 1\right)^{1/4}+\mathfrak{o}(1).
	\end{align*}
	The claim of Theorem \ref{thm:lasso_l2_uniform} therefore follows from the assumption (\ref{cond:lasso}).
\end{proof}

\section{Proofs for the main results: main steps}\label{section:proof_multiplier_ineq}

In this section, we outline the main steps for the proofs of our main theorems. Proofs for many technical lemmas will be deferred to later sections.

\subsection{Preliminaries}

Let
\begin{align}\label{def:uniform_entropy}
J(\delta,\mathcal{F},L_2) \equiv   \int_0^\delta  \sup_Q\sqrt{1+\log \mathcal{N}(\epsilon\pnorm{F}{Q,2},\mathcal{F},L_2(Q))}\ \d{\epsilon}
\end{align}
denote the \emph{uniform} entropy integral, where the supremum is taken over all discrete probability measures, and
\begin{align}\label{def:bracketing_entropy}
J_{[\,]}(\delta,\mathcal{F},\pnorm{\cdot}{}) \equiv \int_0^\delta \sqrt{1+\log \mathcal{N}_{[\,]}(\epsilon,\mathcal{F},\pnorm{\cdot}{})}\ \d{\epsilon}
\end{align}
denote the \emph{bracketing} entropy integral. The following local maximal inequalities for the empirical process play a key role throughout the proof.
\begin{proposition}\label{prop:local_maximal_ineq}
	Suppose that $\mathcal{F} \subset L_\infty(1)$, and $X_1,\ldots,X_n$'s are i.i.d. random variables with law $P$. Then with $\mathcal{F}(\delta)\equiv \{f \in \mathcal{F}:Pf^2<\delta^2\}$,
	\begin{enumerate}
		\item If the uniform entropy integral (\ref{def:uniform_entropy}) converges, then
		\begin{align}\label{ineq:local_maximal_uniform}
		\E \biggpnorm{\sum_{i=1}^n \epsilon_i f(X_i)}{\mathcal{F}(\delta)}  \lesssim \sqrt{n}J(\delta,\mathcal{F},L_2)\bigg(1+\frac{J(\delta,\mathcal{F},L_2)}{\sqrt{n} \delta^2 \pnorm{F}{P,2}}\bigg)\pnorm{F}{P,2}.
		\end{align}
		\item If the bracketing entropy integral (\ref{def:bracketing_entropy}) converges, then
		\begin{align}\label{ineq:local_maximal_bracketing}
		\E \biggpnorm{\sum_{i=1}^n \epsilon_i f(X_i)}{\mathcal{F}(\delta)} \lesssim \sqrt{n} J_{[\,]}(\delta,\mathcal{F},L_2(P))\bigg(1+\frac{J_{[\,]}(\delta,\mathcal{F},L_2(P))}{\sqrt{n} \delta^2}\bigg).
		\end{align}
	\end{enumerate}
\end{proposition}

\begin{proof}
	(\ref{ineq:local_maximal_uniform}) follows from \cite{van2011local}; see also Section 3 of \cite{gine2006concentration}, or Theorem 3.5.4 of \cite{gine2015mathematical}. (\ref{ineq:local_maximal_bracketing}) follows from Lemma 3.4.2 of \cite{van1996weak}.
\end{proof}

We will primarily work with $F\equiv 1$ in the above inequalities. A two-sided estimate for the empirical process will be important for proving lower bounds in Theorems \ref{thm:lower_bound_mep} and \ref{thm:lse_lower_bound}. The following definition is from \cite{gine2006concentration}, page 1167.

\begin{definition}
	A function class $\mathcal{F}$ is \emph{$\alpha$-full} $(0<\alpha<2)$ if and only if there exists some constant $K_1,K_2>1$ such that both
	\begin{align*}
	\log \mathcal{N}\big(\epsilon\pnorm{F}{L_2(\Prob_n)},\mathcal{F},L_2(\Prob_n)\big)\leq K_1\epsilon^{-\alpha},\qquad a.s.
	\end{align*}
	for all $\epsilon>0, n \in \N$, and
	\begin{align*}
	\log \mathcal{N}\big(\sigma\pnorm{F}{L_2(P)}/K_2, \mathcal{F},L_2(P)\big)\geq K_2^{-1}\sigma^{-\alpha}
	\end{align*}
	hold. Here $\sigma^2\equiv \sup_{f \in \mathcal{F}}Pf^2$, $F$ denotes the envelope function for $\mathcal{F}$, and $\Prob_n$ is the empirical measure for i.i.d. samples $X_1,\ldots,X_n$ with law $P$.
\end{definition}

The following lemma, giving a sharp two-sized control for the empirical process under the $\alpha$-full assumption, is proved in Theorem 3.4 of \cite{gine2006concentration}.
\begin{lemma}\label{lem:gine_koltchinskii_matching_bound_ep}
	Suppose that $\mathcal{F}\subset L_\infty(1)$ is $\alpha$-full with $\sigma^2\equiv \sup_{f \in \mathcal{F}}Pf^2$. If $
	n\sigma^2\gtrsim_{\alpha} 1$ and $\sqrt{n} \sigma \left(\frac{\pnorm{F}{L_2(P)}}{\sigma}\right)^{\alpha/2}\gtrsim_{\alpha} 1$,
	then there exists some constant $K>0$ depending only on $\alpha,K_1,K_2$ such that
	\begin{align*}
	K^{-1} \sqrt{n} \sigma \bigg(\frac{\pnorm{F}{L_2(P)}}{\sigma}\bigg)^{\alpha/2}\leq \E \biggpnorm{\sum_{i=1}^n \epsilon_if(X_i)}{\mathcal{F}}\leq K\sqrt{n} \sigma \bigg(\frac{\pnorm{F}{L_2(P)}}{\sigma}\bigg)^{\alpha/2}.
	\end{align*}
\end{lemma}
Note that the right side of the inequality can also be derived from (\ref{ineq:local_maximal_uniform}) (taking supremum over all finitely discrete probability measures only serves to get rid of the random entropy induced by $L_2(\Prob_n)$ norm therein).

The following lemma guarantees the existence of a particular type of $\alpha$-full class that serves as the basis of the construction in the proof of Theorems \ref{thm:lower_bound_mep} and \ref{thm:lse_lower_bound}. The proof can be found in Section \ref{section:remaining_proof_II}.
\begin{lemma}\label{lem:existence_alpha_full}
	Let $\mathcal{X}, P$ be as in Theorem \ref{thm:lower_bound_mep}. Then for each $\alpha>0$, there exists some function class $\mathcal{F}$ defined on $\mathcal{X}$ which is $\alpha$-full and contains $\mathcal{G}\equiv \{\bm{1}_{[a,b]}:0\leq a\leq b\leq 1\}$.
\end{lemma}

\subsection{Proof of Theorem \ref{thm:local_maximal_generic}}

The key ingredient in the proof of Theorem \ref{thm:local_maximal_generic} is the following, which may be of independent interest.

\begin{proposition}\label{prop:key_prop}
	Suppose Assumption \ref{assump:multiplier_function} holds. For any function class $\mathcal{F}$,
	\begin{align}\label{ineq:interpolation_cont_prin}
	\E \biggpnorm{\sum_{i=1}^n \xi_i f(X_i)}{\mathcal{F}}\leq  \E \left[ \sum_{k=1}^n  (\abs{\eta_{(k)}}-\abs{\eta_{(k+1)}}) \E\biggpnorm{\sum_{i=1}^k \epsilon_i f(X_i)}{\mathcal{F}}\right]
	\end{align}
	where $\abs{\eta_{(1)}}\geq \cdots \geq \abs{\eta_{(n)}}\geq \abs{\eta_{(n+1)}}\equiv 0$ are the reversed order statistics for: (i) (under (A1)) $\{2\abs{\xi_i}\}_{i=1}^n$, (ii) (under (A2)) $\{\abs{\xi_i-\xi_i'}\}_{i=1}^n$ with $\{\xi_i'\}$ being an independent copy of $\{\xi_i\}$.
\end{proposition}

\begin{proof}[Proof of Proposition \ref{prop:key_prop}]
	We drop $\mathcal{F}$ from the notation for supremum norm over $\mathcal{F}$ and write $\pnorm{\cdot}{}$ for $\pnorm{\cdot}{\mathcal{F}}$. We first consider the condition (A1). Note that for $(X_1',\ldots,X_n')$ being an independent copy of $(X_1,\ldots,X_n)$, we have
	\begin{align*}
	\E \biggpnorm{\sum_{i=1}^n \xi_if(X_i)}{}& = \E_{\bm{\xi},\bm{X}} \biggpnorm{\sum_{i=1}^n \xi_i \big(f(X_i)-\E_{\bm{X}'} f(X_i')\big)}{} \leq \E \biggpnorm{\sum_{i=1}^n \xi_i \big(f(X_i)-f(X_i')\big)}{}.
	\end{align*}
	Here in the first equality we used the centeredness assumption on the function class $\mathcal{F}$ in (A1). Now conditional on $\bm{\xi}$, for fixed $\epsilon_1,\ldots,\epsilon_n$, the map $
	(X_1,\ldots,X_n,X_1',\ldots,X_n')\mapsto \pnorm{\sum_{i=1}^n \xi_i \epsilon_i\big(f(X_i)-f(X_i')\big)}{} $
	is a permutation of the original map (without $\epsilon_i$'s). Since $(X_1,\ldots,X_n,X_1',\ldots,X_n')$ is the coordinate projection of a product measure, it follows by taking expectation over $\epsilon_1,\ldots,\epsilon_n$ that
	\begin{align}\label{ineq:multiplier_ineq_0_sym}
	\E_{\bm{X},\bm{X}'} \biggpnorm{\sum_{i=1}^n \xi_i \big(f(X_i)-f(X_i')\big)}{} & = \E_{\bm{\epsilon},\bm{X},\bm{X}'} \biggpnorm{\sum_{i=1}^n \xi_i \epsilon_i\big(f(X_i)-f(X_i')\big)}{}.
	\end{align}
	This entails that
	\begin{align}\label{ineq:multiplier_ineq_0}
	\E \biggpnorm{\sum_{i=1}^n \xi_if(X_i)}{}\leq 2 \E_{\bm{\xi},\bm{\epsilon},\bm{X}} \biggpnorm{\sum_{i=1}^n \abs{\xi_i}\mathrm{sgn}(\xi_i)\epsilon_i f(X_i)}{}=2 \E \biggpnorm{\sum_{i=1}^n \abs{\xi_i}\epsilon_i f(X_i)}{}
	\end{align}
	where the equality follows since the random vector $(\mathrm{sgn}(\xi_1)\epsilon_1,\ldots,\mathrm{sgn}(\xi_n)\epsilon_n)$ has the same distribution as that of $(\epsilon_1,\ldots,\epsilon_n)$ and is independent of $\xi_1,\ldots,\xi_n$. 
	We will simply write $\abs{\xi_i}$ without the absolute value in the sequel for notational convenience. Let $\pi$ be a permutation over $\{1,\ldots,n\}$ such that $\xi_{i}=\xi_{(\pi(i))}$. Then the right hand side of (\ref{ineq:multiplier_ineq_0}) equals
	\begin{align}\label{ineq:multiplier_ineq_1}
	\E \biggpnorm{ \sum_{i=1}^n \xi_{(\pi(i))} \epsilon_i f(X_i) }{} & = \E \biggpnorm{\sum_{i=1}^n \xi_{(i)} \epsilon_{\pi^{-1}(i)} f(X_{\pi^{-1}(i)}) }{} \quad \textrm{(by
		relabelling)}\\
	& = \E \biggpnorm{\sum_{i=1}^n \xi_{(i)}\epsilon_i f(X_i) }{} \quad \textrm{(by invariance of }(P_X\otimes P_\epsilon)^n).\nonumber
	\end{align}
	Now write $\xi_{(i)} = \sum_{k\geq i} (\xi_{(k)}-\xi_{(k+1)})$ where $\xi_{(n+1)}\equiv 0$. The above display can be rewritten as
	\begin{align}\label{ineq:multiplier_ineq_2}
	\E \biggpnorm{\sum_{i=1}^n \sum_{k=i}^n  (\xi_{(k)}-\xi_{(k+1)}) \epsilon_if(X_{i}) }{} =  \E \biggpnorm{ \sum_{k=1}^n  (\xi_{(k)}-\xi_{(k+1)}) \sum_{i=1}^k \epsilon_i f(X_i)}{}.
	\end{align}
	The claim under (A1) follows by combining (\ref{ineq:multiplier_ineq_0})-(\ref{ineq:multiplier_ineq_2}). For (A2), let $\xi_i'$'s be an independent copy of $\xi_i$'s. Then the analogy of (\ref{ineq:multiplier_ineq_0}) becomes
	\begin{align*}
	\E \biggpnorm{\sum_{i=1}^n \xi_if(X_i)}{}&=\E \biggpnorm{\sum_{i=1}^n (\xi_i-\E \xi_i')f(X_i)}{}\leq \E \pnorm{\sum_{i=1}^n (\xi_i- \xi_i')f(X_i)}{}\\
	&= \E \biggpnorm{\sum_{i=1}^n \epsilon_i\abs{\xi_i-\xi_i'}f(X_i)}{}=\E \biggpnorm{\sum_{i=1}^n \epsilon_i\abs{\eta_i}f(X_i)}{}
	\end{align*}
	where $\eta_i\equiv \xi_i-\xi_i'$. The claim for (A2) follows by repeating the arguments in (\ref{ineq:multiplier_ineq_1}) and (\ref{ineq:multiplier_ineq_2}).
\end{proof}

\begin{proof}[Proof of Theorem \ref{thm:local_maximal_generic}]
	First consider (A1). Using Proposition \ref{prop:key_prop} we see that,
	\begin{align}\label{ineq:local_maximal_mul_1}
	\E \biggpnorm{\sum_{i=1}^n \xi_i f(X_i)}{\mathcal{F}_n}\leq 2 \E \bigg[\sum_{k=1}^n  (\abs{\xi_{(k)}}-\abs{\xi_{(k+1)}}) \E \biggpnorm{\sum_{i=1}^k \epsilon_i f(X_i)}{\mathcal{F}_n}\bigg].
	\end{align}
	By the assumption that $\mathcal{F}_k\supset \mathcal{F}_n$ for any $1\leq k\leq n$,
	\begin{align}\label{ineq:local_maximal_mul_3}
	\E \biggpnorm{\sum_{i=1}^k \epsilon_i f(X_i)}{\mathcal{F}_n} \leq \E \biggpnorm{\sum_{i=1}^k \epsilon_i f(X_i)}{\mathcal{F}_k}\leq \psi_n(k).
	\end{align}
	Collecting (\ref{ineq:local_maximal_mul_1})-(\ref{ineq:local_maximal_mul_3}), we see that
	\begin{align*}
	\E \biggpnorm{\sum_{i=1}^n \xi_i f(X_i)}{\mathcal{F}_n}&\leq 2 \E \bigg[\sum_{k=1}^n  (\abs{\xi_{(k)}}-\abs{\xi_{(k+1)}}) \psi_n(k)\bigg] = 2 \E \sum_{k=1}^n \int_{\abs{\xi_{(k+1)}} }^{ \abs{\xi_{(k)}}}\psi_n(k)\ \d{t}\\
	& \leq 2 \E \int_0^\infty \psi_n\left(\abs{\{i:\abs{\xi_i}\geq t\}}\right)\ \d{t}\leq 2\int_0^\infty  \psi_n\bigg(\sum_{i=1}^n \Prob(\abs{\xi_i}>t)\bigg) \ \d{t}
	\end{align*}
	where the last inequality follows from Fubini's theorem and Jensen's inequality, completing the proof for the upper bound for (A1). For (A2), mimicking the above proof, we have
	\begin{align*}
	\E \biggpnorm{\sum_{i=1}^n \xi_i f(X_i)}{\mathcal{F}_n} &\leq \int_0^\infty \psi_n\bigg(\sum_{i=1}^n\Prob(\abs{\xi_i-\xi_i'}\geq t)\bigg)\ \d{t}\\
	&\leq \int_0^\infty \psi_n\bigg(\sum_{i=1}^n\Prob(\abs{\xi_i}\geq t/2)+\Prob(\abs{\xi_i'}\geq t/2)\bigg)\ \d{t}\\
	& = \int_0^\infty \psi_n\bigg(2\sum_{i=1}^n\Prob(\abs{\xi_i}\geq t/2)\bigg)\ \d{t}\\
	&=2\int_0^\infty  \psi_n\bigg(2\sum_{i=1}^n \Prob(\abs{\xi_i}>t)\bigg)\ \d{t}.
	\end{align*}
	The proof of the claim for (A2) is completed by noting that $\psi_n(2x)\leq 2\psi_n(x)$ due to the concavity of $\psi_n$ and $\psi_n(0)=0$.
\end{proof}

\subsection{Proof of Theorem \ref{thm:lower_bound_mep}}

We need the following lemma.

\begin{lemma}\label{lem:lower_bound_mep_ep}
	Suppose that $\xi_1,\ldots,\xi_n$ are i.i.d. mean-zero random variables independent of i.i.d. $X_1,\ldots,X_n$. Then
	\begin{align*}
	\pnorm{\xi_1}{1}\E \biggpnorm{\sum_{i=1}^n \epsilon_i f(X_i)}{\mathcal{F}}\leq 2\E\biggpnorm{\sum_{i=1}^n \xi_i f(X_i)}{\mathcal{F}}.
	\end{align*}
\end{lemma}

\begin{proof}
The proof follows that of the left hand side inequality in Lemma 2.9.1 of \cite{van1996weak}, so we omit the details.
\end{proof}

\begin{lemma}\label{lem:order_n_dist_uniform}
	Let $X_1,\ldots,X_n$ be i.i.d. random variables distributed on $[0,1]$ with a probability law $P$ admitting a Lebesgue density bounded away from $\infty$. Let $\{I_i\}_{i=1}^n$ be a partition of $[0,1]$ such that $I_i\cap I_j=\emptyset$ for $i\neq j$ and $\cup_{i=1}^n I_i=[0,1]$, and $L^{-1}n^{-1}\leq \abs{I_i}\leq L n^{-1}$ for some absolute value $L>0$. Then there exists some $\tau\equiv\tau_{L,P}\in (0,1)$ such that for $n$ sufficiently large,
	\begin{align*}
	\Prob\left(X_1,\ldots,X_n \textrm{ lie in at most } \tau n \textrm{ intervals among } \{I_i\}_{i=1}^n\right)\leq 0.5^{n-1}.
	\end{align*}
\end{lemma}

The proofs of Lemma \ref{lem:order_n_dist_uniform} can be found in Section \ref{section:remaining_proof_II}. Now we are in position to prove Theorem \ref{thm:lower_bound_mep}.	
\begin{proof}[Proof of Theorem \ref{thm:lower_bound_mep}]
	The proof will proceed in two steps. The first step aims at establishing a lower bound for the multiplier empirical process on the order of $n^{1/\gamma}$.
	
	Let $\alpha=2/(\gamma-1)$, and $\tilde{\mathcal{F}}$ be an $\alpha$-full class on $\mathcal{X}$ in Lemma \ref{lem:existence_alpha_full}. Further let $\delta_k = k^{-1/(2+\alpha)}$ and $\tilde{\mathcal{F}}_k\equiv \tilde{\mathcal{F}}(\delta_k)=\{f \in \tilde{\mathcal{F}}:Pf^2<\delta_k^2\}$. Then it follows from Lemma \ref{lem:gine_koltchinskii_matching_bound_ep} that there exists some constant $K>0$, 
	\begin{align*}
	K^{-1} k^{\alpha/(2+\alpha)}\leq \E\biggpnorm{\sum_{i=1}^k \epsilon_i f(X_i)}{\tilde{\mathcal{F}}_k}\leq K k^{\alpha/(2+\alpha)}.
	\end{align*}
	Lemma \ref{lem:lower_bound_mep_ep} now guarantees that $
	\E\pnorm{\sum_{i=1}^n \xi_i f(X_i)}{\tilde{\mathcal{F}}_n}$ can be bounded from below by a constant multiple of $n^{\alpha/(2+\alpha)}=n^{1/\gamma}$ where the constant depends on $\pnorm{\xi_1}{1}$. This completes the first step of the proof.
	
	In the second step, we aim at establishing a lower bound of order $n^{1/p}$. To  this end, let $\{I_j\}_{j=1}^n$ be a partition of $ \mathcal{X}$ such that $L^{-1}n^{-1}\leq \abs{I_j}\leq Ln^{-1}$. On the other hand, let $f_j\equiv \bm{1}_{I_j}\in \tilde{\mathcal{F}}_{n}$ for $1\leq j\leq n$ (increase $\delta_n$ by constant factors if necessary), and $\mathcal{E}_n$ denote the event that $X_1,\ldots,X_n$ lie in $N\geq \tau n$ sets among $\{I_j\}_{j=1}^n$. Then Lemma \ref{lem:order_n_dist_uniform} entails that $\Prob(\mathcal{E}_n)\geq 1-0.5^n\geq 1/2$ for $n$ sufficiently large. Furthermore, let $\mathcal{I}_j\equiv \{i:X_i \in I_j\}$ and pick any $X_{\iota(j)} \in I_j$. Note that $\mathcal{I}_j$'s are disjoint, and hence conditionally on $\bm{X}$ we have
	\begin{align*}
	\E\max_{1\leq j\leq \tau n}\abs{\xi_j}&\leq \E \max_{1\leq j\leq N}\abs{\xi_{\iota(j)}}\quad(\textrm{by i.i.d. assumption on }\xi_i\textrm{'s})\\
	&\leq \E \max_{1\leq j\leq N}\biggabs{\xi_{\iota(j)}+\E\sum_{i \in \mathcal{I}_j\setminus \iota(j)}\xi_i} \quad (\mathcal{I}_j\textrm{'s are disjoint and }\E\xi_i=0)\\
	&\leq \E \max_{1\leq j\leq N} \biggabs{\sum_{i \in \mathcal{I}_j}\xi_i}\quad (\textrm{by Jensen's inequality}).
	\end{align*}
	Then
	\begin{align*}
	\E \biggpnorm{\sum_{i=1}^n \xi_i f(X_i)}{\tilde{\mathcal{F}}_n}&\geq \E\bigg[ \max_{1\leq j\leq n}\biggabs{\sum_{i=1}^n \xi_i f_j(X_i)}\bigg]  \geq \E_{\bm{X}} \bigg[\E_{\bm{\xi}} \max_{1\leq j\leq N} \biggabs{\sum_{i \in \mathcal{I}_j}\xi_i} \bm{1}_{\mathcal{E}_n}\bigg]\\
	& \geq \E_{\bm{X}} \big[ \E_{\bm{\xi}}\max_{1\leq j\leq \tau n}\abs{\xi_j}\bm{1}_{\mathcal{E}_n}\big]\geq \frac{1}{2}\E_{\bm{\xi}}\max_{1\leq j\leq \tau n}\abs{\xi_j}
	\end{align*}
	for $n$ sufficiently large. Now the second step follows from the assumption, and hence completing the proof.
\end{proof}

\subsection{Proof of Theorem \ref{thm:lse_rate_optimal}}

We first prove Proposition \ref{prop:general_lse}.

\begin{proof}[Proof of Proposition \ref{prop:general_lse}]
	Let $
	\mathbb{M}_n f \equiv \frac{2}{n} \sum_{i=1}^n (f-f_0)(X_i)\xi_i - \frac{1}{n}\sum_{i=1}^n (f-f_0)^2(X_i)$, and $
	 Mf \equiv \E \left[\mathbb{M}_n(f)\right]=-P(f-f_0)^2$.
	Here we used the fact that $\E\xi_i=0$ and the independence assumption between $\{\xi_i\}$ and $\{X_i\}$. Then it is easy to see that
	\begin{align*}
	&\abs{\mathbb{M}_n f -\mathbb{M}_n f_0 -(Mf - Mf_0)}\leq \biggabs{\frac{2}{n}\sum_{i=1}^n (f-f_0)(X_i)\xi_i}+ \abs{(\Prob_n-P)(f-f_0)^2}.
	\end{align*}
	The first claim (i.e. convergence rate in probability) follows by standard symmetrization and contraction principle for the  empirical process indexed by a uniformly bounded function class, followed by an application of Theorem 3.2.5 of \cite{van1996weak}. 
	
	Now assume that $\xi_1,\ldots,\xi_n$ are i.i.d. mean-zero errors with $\pnorm{\xi_1}{p}<\infty$ for some $p\geq 2$. Fix $t\geq 1$. For $j\in \N$, let $\mathcal{F}_{j}\equiv \{f \in \mathcal{F}: 2^{j-1}t\delta_n\leq  \pnorm{f-f_0}{L_2(P)}< 2^j t\delta_n\}$. Then by a standard peeling argument, we have
	\begin{align*}
	\Prob\left(\pnorm{\hat{f}_n-f_0}{L_2(P)}\geq t \delta_n\right)&\leq \sum_{j\geq 1} \Prob\big(\sup_{f \in \mathcal{F}_j}\left(\mathbb{M}_n(f)-\mathbb{M}_n(f_0)\right)\geq 0 \big).
	\end{align*}
	Each probability term in the above display can be further bounded by
	\begin{align*}
	& \Prob\big(\sup_{f \in \mathcal{F}_j}\left(\mathbb{M}_n(f)-\mathbb{M}_n(f_0)-(Mf-Mf_0)\right)\geq 2^{2j-2}t^2\delta_n^2 \big)\\
	&\leq \Prob\bigg(\sup_{f \in \mathcal{F}-f_0: \pnorm{f}{L_2(P)}\leq 2^j t\delta_n }\biggabs{\frac{1}{\sqrt{n}}\sum_{i=1}^n \xi_i f(X_i)}\geq 2^{2j-4} t^2 \sqrt{n}\delta_n^2\bigg)\\
	&\qquad\qquad + \Prob\bigg(\sup_{f \in \mathcal{F}-f_0: \pnorm{f}{L_2(P)}\leq 2^j t\delta_n }\biggabs{\frac{1}{\sqrt{n}}\sum_{i=1}^n \left(f^2(X_i)-Pf^2\right)}\geq 2^{2j-3} t^2 \sqrt{n}\delta_n^2\bigg).
	\end{align*}
	By the contraction principle and moment inequality for the empirical process (Lemma \ref{lem:p_moment_estimate}), we have
	\begin{align*}
	&\E \bigg(\sup_{ \substack{f \in \mathcal{F}-f_0: \\ \pnorm{f}{L_2(P)}\leq 2^j t\delta_n} }\biggabs{\frac{1}{\sqrt{n}}\sum_{i=1}^n \xi_i f(X_i)}^2\bigg)\vee \E \bigg(\sup_{  \substack{ f \in \mathcal{F}-f_0:\\ \pnorm{f}{L_2(P)}\leq 2^j t\delta_n} }\biggabs{\frac{1}{\sqrt{n}}\sum_{i=1}^n \epsilon_i f^2(X_i)}^2\bigg)\\
	&\qquad \lesssim \big[\phi_n(2^j t \delta_n)\big]^2+ (1\vee \pnorm{\xi_1}{2})^2 2^{2j} t^2\delta_n^2+(1\vee \pnorm{\xi_1}{p})^2n^{-1+2/p}.
	\end{align*}
	In the above calculation we used the fact that $\E \max_{1\leq i\leq n}\abs{\xi_i}^2\leq \pnorm{\xi_1}{p}^2n^{2/p}$ under $\pnorm{\xi_1}{p}<\infty$. By Chebyshev's inequality,
	\begin{align*}
	\Prob\left(\pnorm{\hat{f}_n-f_0}{L_2(P)}\geq t \delta_n\right)
	&\leq C_{\xi} \sum_{j\geq 1} \bigg[ \left(\frac{\phi_n(2^j t\delta_n)}{2^{2j}t^2 \sqrt{n}\delta_n^2}\right)^2\vee \frac{1}{2^{2j}t^2 n\delta_n^2}\vee \frac{1}{2^{4j} t^4 n^{2-2/p}\delta_n^4 }\bigg].
	\end{align*}
	Under the assumption that $\delta_n\geq n^{-\frac{1}{2}+\frac{1}{2p}}$, and noting that $\phi_n(2^j t \delta_n)\leq 2^j t \phi_n(\delta_n)$ by the assumption that $\delta \mapsto \phi_n(\delta)/\delta$ is non-increasing, the right side of the above display can be further bounded up to a constant by $
	\sum_{j\geq 1}  \left(\frac{\phi_n(\delta_n)}{2^{j}t \sqrt{n}\delta_n^2}\right)^2+\frac{1}{t^2}\lesssim \frac{1}{t^2}$ 
	for $t\geq 1$. The expectation bound follows by integrating the tail estimate.
\end{proof}

The following lemma calculates an upper bound for the multiplier empirical process at the target rate in Theorem \ref{thm:lse_rate_optimal}. The proof can be found in Section \ref{section:remaining_proof_II}.

\begin{lemma}\label{lem:upper_bound_lse}
	Suppose that Assumption \ref{assump:multiplier_function} holds with i.i.d. $X_1,\ldots,X_n$'s with law $P$, and $\mathcal{F}\subset L_\infty(1)$ satisfies the entropy condition (F) with $\alpha \in (0,2)$. Further assume for simplicity that $\xi_i$'s have the same marginal distributions with $\pnorm{\xi_1}{p,1}<\infty$. Then with $
	\delta_n \equiv n^{-\frac{1}{2+\alpha}} \vee n^{-\frac{1}{2}+\frac{1}{2p}}$, 
	we have 
	\begin{align*}
	&\E \sup_{Pf^2\leq \rho^2 \delta_n^2} \biggabs{\sum_{i=1}^n \xi_i f(X_i)}  \vee \E \sup_{Pf^2\leq \rho^2 \delta_n^2} \biggabs{\sum_{i=1}^n \epsilon_i f(X_i)} \\
	&\leq \bar{K}_\alpha  (\rho^{1-\alpha/2} \vee \rho^{-\alpha}) 
	\begin{cases}
		n^{\frac{\alpha}{2+\alpha}}\big(1\vee\pnorm{\xi_1}{1+2/\alpha,1}\big), & p\geq 1+2/\alpha,\\
		n^{\frac{1}{p}}\big(1\vee \pnorm{\xi_1}{p,1}\big), & 1\leq p<1+2/\alpha.
	\end{cases}
	\end{align*}
\end{lemma}

\begin{proof}[Proof of Theorem \ref{thm:lse_rate_optimal}]
	The claim follows immediately from Lemma \ref{lem:upper_bound_lse} by noting that the rate $\delta_n$ chosen therein corresponds to the condition (\ref{cond:thm_general_lse_1}) in Proposition \ref{prop:general_lse}, along with Proposition \ref{prop:local_maximal_ineq} handling (\ref{cond:thm_general_lse_2}).
\end{proof}

\subsection{Proof of Theorem \ref{thm:lse_lower_bound}}

We will prove the following slightly more general version of Theorem \ref{thm:lse_lower_bound}.

\begin{theorem}\label{thm:lse_lower_bound_general}
	Let $\mathcal{X}=[0,1]$ and $P$ be a probability measure on $\mathcal{X}$ with Lebesgue density bounded away from $0$ and $\infty$. Suppose that $\xi_1,\ldots,\xi_n$ are i.i.d. mean-zero random variables with $\pnorm{\xi_1}{p,1}<\infty$ for some $p\geq 2$. Then: 
	\begin{enumerate}
		\item For each $\alpha \in (0,2)$, there exists a function class $\mathcal{F}$ and some $f_0 \in \mathcal{F}$ with $\mathcal{F}-f_0$ satisfying the entropy condition (F), such that for $p\geq 1+2/\alpha$, 
		there exists some least squares estimator $f_n^\ast$ over $\mathcal{F}$ satisfying
		\begin{align*}
		\E \pnorm{f^\ast_n-f_0}{L_2(P)}\geq  \rho\cdot n^{-\frac{1}{2+\alpha}}.
		\end{align*}
		Here $\rho>0$ is a (small) constant independent of $n$.
		\item 
		For each $\alpha \in (0,2)$, there exists a function class $\mathcal{F}\equiv \mathcal{F}_n$, some $f_0 \in \mathcal{F}$ with $\mathcal{F}-f_0$ satisfying the entropy condition (F), such that the following holds:  suppose $\sqrt{\log n}\leq p \leq (\log n)^{1-\delta}$ for some $\delta \in (0,1/2)$. Then there exists some law for the error $\xi_1$ with $\pnorm{\xi_1}{p,1}\lesssim \log n$, such that for $n$ sufficiently large, there exists some least squares estimator $f_n^\ast$ over $\mathcal{F}_n$ satisfying
		\begin{align*}
		\E \pnorm{f^\ast_n-f_0}{L_2(P)}\geq \rho' \cdot n^{-\frac{1}{2}+\frac{1}{2p}}(\log n)^{-2}.
		\end{align*}
		Here $\rho'>0$ is a (small) constant independent of $n$.
	\end{enumerate}
\end{theorem}

\subsubsection{The strategy}
The proof of Theorem \ref{thm:lse_lower_bound_general} is technically rather involved; here we give a brief outline. There are two main steps:
\begin{enumerate}
	\item We first show that (see Lemma \ref{lem:characterization_lse}): the risk of \emph{some} LSE corresponds to the extreme value $\delta_n^\ast$ of the map
	\begin{align}\label{def:F_n_E_n}
	\delta\mapsto F_n(\delta)\equiv \sup_{f \in \mathcal{F}-f_0: Pf^2\leq \delta^2} (\Prob_n-P)(2\xi f-f^2)-\delta^2\equiv E_n(\delta)-\delta^2.
	\end{align}
	This step is similar in spirit to \cite{chatterjee2014new,van2015concentration} (e.g. Theorem 1.1 of \cite{chatterjee2014new}; Lemma 3.1 of \cite{van2015concentration}); we will deal with the fact that the LSE may not be unique, and the map does not enjoy good geometric properties such as convexity in \cite{chatterjee2014new}, or uniqueness of the extreme value of the map (\ref{def:F_n_E_n}) as in \cite{van2015concentration}.
	\item Step 1 reduces the problem of finding $\delta_n^\ast$ to that of finding $\delta_1<\delta_2$ such that $E_n(\delta_1)<F_n(\delta_2)$ [This means, $F_n(\delta)<F_n(\delta_2)$ for $\delta\leq \delta_1$, implying that the extreme value $\delta_n^\ast\geq \delta_1$]. Hence our task will be to find $\delta_1,\delta_2$ with matching order such that $E_n(\delta_1)$ is smaller than $F_n(\delta_2)$ up to a constant order under a specific function class. The construction of such an underlying regression function is inspired by the one used in Theorem \ref{thm:lower_bound_mep}. The main technical job involves (i) developing a problem-specific approach to derive an upper bound for $E_n(\rho \delta_1)$ for \emph{small} $\rho>0$ (corresponding to the Poisson (small-sample) domain of the empirical process where general tools fail), (ii) using a Paley-Zygmund moment argument to produce a \emph{sharp} lower bound for $F_n(\delta_2)$
	and (iii) handling the delicate fact that the $L_{p,1}$ norm is slightly stronger than the $L_p$ norm. 
\end{enumerate}

\subsubsection{The reduction scheme}

\begin{lemma}\label{lem:characterization_lse}
	Fix $\epsilon>0$. Let $\delta^\ast_n \equiv\inf\{\delta^\ast\geq 0: F_n(\delta^\ast)\geq \sup_{\delta \in [0,\infty)} F_n(\delta)-\epsilon\}$. Then there exists a $2\epsilon$-approximate LSE $f^\ast_n$ such that
	\begin{align*}
	\pnorm{f^\ast_n-f_0}{L_2(P)}^2 \geq (\delta_n^\ast)^2-\epsilon.
	\end{align*}
\end{lemma}
\begin{proof}
	Without loss of generality we assume $f_0=0$. Let $f^\ast_n$ be such that $\delta_0^2\equiv P(f_n^\ast)^2\leq (\delta_n^\ast)^2$ and $ E_n(\delta_n^\ast) \leq (\Prob_n-P)(2\xi f_n^\ast -(f_n^\ast)^2)+\epsilon$. Note that for any $f \in \mathcal{F}$, 
	\begin{align*}
	(\Prob_n-P)(2\xi f-f^2)-Pf^2&\leq F_n(\pnorm{f}{L_2(P)})\leq F_n(\delta^\ast_n)+\epsilon=E_n(\delta^\ast_n)-(\delta_n^\ast)^2+\epsilon\\
	&\leq (\Prob_n-P)(2\xi f^\ast_n-(f^\ast_n)^2)-P(f^\ast_n)^2+2\epsilon,
	\end{align*}
	where in the last inequality we used the definition of $f^\ast_n$ and the fact that $(\delta^\ast_n)^2\geq P(f^\ast_n)^2$. This implies that for any $f \in \mathcal{F}$,
	\begin{align*}
	\pnorm{Y-f}{n}^2 &=\Prob_n \xi^2-\Prob_n(2\xi f-f^2)\\
	&\geq \Prob_n \xi^2-\Prob_n(2\xi f^\ast_n- (f^\ast_n)^2)-2\epsilon=\pnorm{Y-f^\ast_n}{n}^2-2\epsilon.
	\end{align*}
	Hence $f^\ast_n$ is a $2\epsilon$-approximate LSE. 	The claim follows if we can show $\delta_0^2\geq (\delta_n^\ast)^2-\epsilon$. This is valid: if $(\delta_n^\ast)^2>\delta_0^2+\epsilon$, then
	\begin{align*}
	F_n(\delta_0)&\geq (\Prob_n-P)(2\xi f_n^\ast-(f_n^\ast)^2)-\delta_0^2\\
	&\geq E_n(\delta_n^\ast)-\epsilon- \delta_0^2=F_n(\delta_n^\ast) -\epsilon+ ((\delta_n^\ast)^2-\delta_0^2)>F_n(\delta_n^\ast),
	\end{align*}
	a contradiction to the definition of $\delta_n^\ast$ by noting that $\delta_0\leq \delta_n^\ast$.
\end{proof}

Since $\epsilon>0$ in Lemma \ref{lem:characterization_lse} can be arbitrarily small, in the following analysis we will assume without loss of generality that $\epsilon=0$. 

The following simple observation summarizes the strategy for finding a lower bound on the rate of \emph{some} least squares estimator.
\begin{proposition}\label{prop:proof_route_lower_bound_lse}
	Suppose that $0<\delta_1<\delta_2$ are such that $E_n(\delta_1)<F_n(\delta_2)$. Then there exists a LSE $f^\ast_n$ such that $\pnorm{f^\ast_n-f_0}{L_2(P)}\geq \delta_1$.
\end{proposition}
\begin{proof}
	The condition implies that $F_n(\delta)=E_n(\delta)-\delta^2\leq E_n(\delta)\leq E_n(\delta_1)<F_n(\delta_2)$ for any $\delta\leq \delta_1$ and hence a maximizer $\delta_n^\ast$ of the map $\delta\mapsto F_n(\delta)$ cannot lie in $[0,\delta_1]$, i.e. $\delta_n^\ast \geq \delta_1$. The claim now follows by Lemma \ref{lem:characterization_lse}. 
\end{proof}

In the next few subsections, we first prove claim (1) of Theorem \ref{thm:lse_lower_bound_general}. The proof of claim (2) follows the similar proof strategy as that of claim (1); the details will be delayed until the last subsection.
\subsubsection{Upper bound}

The regression function class ${\mathcal{F}}$ we consider will be the H\"older class constructed in Lemma \ref{lem:existence_alpha_full} with $\mathcal{X}=[0,1]$, and $f_0\equiv 0$. We first handle the \emph{upper bound} part of the problem.  Lemma \ref{lem:upper_bound_lse} is awkward in this regard because general tools (Proposition \ref{prop:local_maximal_ineq}) cannot handle the Poisson (small-sample) domain of the empirical process, and hence the resulting bound is \emph{insensitive} with respect to small $\rho>0$. The following lemma remedies this for our special function class ${\mathcal{F}}$.

\begin{lemma}\label{lem:upper_bound_lower_bound_lse}
	Suppose that $\mathcal{X}=[0,1]$ and $P$ is a probability measure on $\mathcal{X}$ with Lebesgue density bounded away from $0$ and $\infty$. Suppose further that the $\xi_i$'s are i.i.d. mean-zero and $\pnorm{\xi_1}{p,1}<\infty$ and $p\geq 1+2/\alpha$. Then for any $\rho \in (0,1)$, if 
	\begin{align}\label{cond:sample_size_upper_bound}
	n\geq \min\{n \geq 3: \rho^2\geq \log n (n^{-\alpha/(2+\alpha)})\},
	\end{align}
	then with $\delta_n\equiv  n^{-\frac{1}{2+\alpha}} $ we have,
	\begin{align*}
	\E \sup_{f \in {\mathcal{F}}: Pf^2\leq \rho^2 \delta_n^2} \biggabs{\sum_{i=1}^n \big(2\xi_i f(X_i)-f^2(X_i)+Pf^2\big)} \leq \bar{K}_{P,\alpha} \rho^{1-\alpha/2} n^{\frac{\alpha}{2+\alpha}} (1 \vee \pnorm{\xi_1}{p,1}).
	\end{align*}
\end{lemma}

Note that in the above lemma we may choose $\rho$ small as long as the sample size condition (\ref{cond:sample_size_upper_bound}) is satisfied. The key idea of the proof is to compare $\sup_{f \in {\mathcal{F}}: Pf^2<\sigma^2} \Prob_n f^2$ with $\sigma^2$ directly for \emph{(nearly) the whole range of $\sigma^2$ including the Poisson (small-sample) domain} by exploiting the geometry of ${\mathcal{F}}$. Details can be found in Section \ref{section:remaining_proof_II}.

\subsubsection{Lower bound}
Next we turn to the \emph{lower} bound part of the problem. We will first consider a lower bound in expectation, then a Paley-Zygmund type argument translates the claim from in expectation to in probability.

\begin{lemma}\label{lem:lower_bound_exp_lse}
	Let $\mathcal{X}=[0,1]$, and $P$ be a probability measure on $\mathcal{X}$ with Lebesgue density bounded away from $0$ and $\infty$. Let $\xi_1,\ldots,\xi_n$ be i.i.d. mean-zero random variables such that $\pnorm{\xi_1}{1}>0$, and ${\mathcal{F}}$ be the H\"older class constructed in Lemma \ref{lem:existence_alpha_full}. 
	Then with $
	\delta_n \equiv n^{-\frac{1}{2+\alpha}}$, if $\pnorm{\xi_1}{1}> \mathfrak{K}_{\alpha,P}$ and $\vartheta\geq  1$ for some $\mathfrak{K}_{\alpha,P}>0$ depending only on $\alpha,P$,
	\begin{align*}
	\E \sup_{Pf^2\leq \vartheta^2 \delta_n^2} \biggabs{\sum_{i=1}^n \big(2\xi_i f(X_i)-f^2(X_i)+Pf^2\big)}\geq \underline{K}_{P,\alpha}\pnorm{\xi_1}{1}\vartheta^{1-\alpha/2} \cdot n^{\frac{\alpha}{2+\alpha}}.
	\end{align*}
\end{lemma}
The proof uses Lemmas \ref{lem:gine_koltchinskii_matching_bound_ep} and \ref{lem:lower_bound_mep_ep}, and the $\alpha$-fullness of $\tilde{\mathcal{F}}$; see Section \ref{section:remaining_proof_II}. The following Paley-Zygmund lower bound is standard.
\begin{lemma}[Paley-Zygmund]\label{lem:paley_zygmund}
	Let $Z$ be any non-negative random variable. Then for any $\epsilon>0$,  
	$
	\Prob (Z>\epsilon \E Z) \geq \left(\frac{(1-\epsilon) \E Z}{(\E Z^q)^{1/q}}\right)^{q'},
	$
	where $q,q' \in (1,\infty)$ are conjugate indices: $1/q+1/q'=1$.
\end{lemma}

Now we turn the lower bound in expectation in Lemma \ref{lem:lower_bound_exp_lse} to a probability bound by a Paley-Zygmund argument.

\begin{lemma}\label{lem:lower_bound_in_prob_lse}
	Consider the same setup as in Lemma \ref{lem:lower_bound_exp_lse} with $p\geq 2$, and let 
	$
	Z
	= \sup_{Pf^2\leq \vartheta^2 \delta_n^2}\abs{\sum_{i=1}^n 2\xi_i f(X_i)-f^2(X_i)+Pf^2}.
	$
	Suppose $p\geq 1+2/\alpha$. If $\pnorm{\xi_1}{p}<\infty$, $\pnorm{\xi_1}{1}>\mathfrak{K}_{\alpha}$, $\vartheta\geq 1$ and $1< q\leq p$. Then
	\begin{align*}
	&\Prob\bigg( Z\geq \frac{1}{2}\underline{K}_{P,\alpha}\pnorm{\xi_1}{1}\vartheta^{1-\alpha/2} \cdot n^{\frac{\alpha}{2+\alpha}}\bigg)\geq 2^{-q/(q-1)}\bar{L}_{\alpha,\xi,\vartheta,q,P}^{-1/(q-1)}>0.
	\end{align*}
	The constant in the probability estimate is defined below.
\end{lemma}
\begin{proof}
	Lemma \ref{lem:lower_bound_exp_lse} entails that $\E Z\geq \underline{K}_{P,\alpha}\pnorm{\xi_1}{1}\vartheta^{1-\alpha/2} \cdot n^{\frac{\alpha}{2+\alpha}}$. By the moment inequality Lemma \ref{lem:p_moment_estimate}, if the $\xi_i$'s have finite $p$-th moments, and $q\leq p$, 
	\begin{align*}
	\frac{\E Z^q}{(\E Z)^q} &\leq C_q \bigg[1 +\frac{(\sqrt{n}(\pnorm{\xi_1}{2}\vee 1)\vartheta \delta_n)^q}{(\E Z)^q}+\frac{1\vee \E \max_{1\leq i\leq n}\abs{\xi_i}^q}{(\E Z)^q}\bigg]\\
	&\leq C_q\bigg[1+2\left(\underline{K}_{P,\alpha} \pnorm{\xi_1}{1}\wedge 1\right)^{-q}\\
	&\qquad\qquad\qquad\times \left(\vartheta^{\alpha q/2} (\pnorm{\xi_1}{2}\vee 1)^q n^{-\frac{q\alpha}{2(2+\alpha)}}\vee \vartheta^{q(\alpha/2-1)}\pnorm{\xi_1}{p}^q n^{\frac{q}{p}-\frac{q\alpha}{2+\alpha}}\right)\bigg]\\
	&\leq C_q\left[1+2\left(\underline{K}_{P,\alpha} \pnorm{\xi_1}{1}\wedge 1\right)^{-q} \vartheta^{\alpha q/2} (\pnorm{\xi_1}{p}\vee 1)^q \right]\equiv \bar{L}_{\alpha,\xi,\vartheta,q,P}.
	\end{align*}
	In the second inequality we used $\pnorm{\max_i \abs{\xi_i}}{q}\leq \pnorm{\max_i \abs{\xi_i}}{p}\leq n^{1/p} \pnorm{\xi_1}{p}$, and the third inequality follows by noting $\vartheta\geq 1$ and the assumption $p\geq 1+2/\alpha$. The proof is complete.
\end{proof}

\subsubsection{Putting the pieces together}
\begin{proposition}\label{prop:rate_lower_bound_lse_generic}
	Suppose $\pnorm{\xi_1}{p,1}<\infty$ for $p\geq \max\{2,1+2/\alpha\}$. If $\pnorm{\xi_1}{1}>\mathfrak{K}_{\alpha,P}$ for some (large) constant $\mathfrak{K}_{\alpha,P}>0$, then for  $n$ sufficiently large, there exist constants $\rho_{\xi,\alpha,P}<\vartheta_{\xi,\alpha,P}$ such that on an event with positive probability $\mathfrak{p}_1=\mathfrak{p}_1(\alpha,\xi,P)>0$ independent of $n$, 
	\begin{align*}
	F_n\big(\vartheta_{\xi,\alpha,P}\cdot  n^{-\frac{1}{2+\alpha}}\big) > E_n\big(\rho_{\xi,\alpha,P}\cdot  n^{-\frac{1}{2+\alpha}}\big).
	\end{align*}
\end{proposition}
\begin{proof}
	Lemma \ref{lem:upper_bound_lower_bound_lse} and Markov's inequality entail that for any $\rho>0$, if $n\geq \min\{n \geq 3: \rho^2\geq \log n (n^{-\alpha/(2+\alpha)})\}$, then on an event with probability at least $1-1/M$,
	\begin{align}\label{ineq:lse_1}
	nE_n( \rho \delta_n)\leq M\bar{C}_{\xi,\alpha,P} \cdot \rho^{1-\alpha/2} n^{\frac{\alpha}{2+\alpha}}.
	\end{align}
	We will choose $\rho, M$ later on. On the other hand, apply Lemma \ref{lem:lower_bound_in_prob_lse} with $q=2$ (since $p\geq 2$) and $\vartheta = (\underline{K}_{P,\alpha}\pnorm{\xi_1}{1}/4)^{\frac{2}{2+\alpha}}$ (we may increase $\pnorm{\xi_1}{1}$ to ensure $\vartheta\geq 1$ if necessary) we see that on an event with probability at least $2\mathfrak{p}_1\equiv 2\mathfrak{p}_1(\alpha,\xi,P)$, we have
	\begin{align}\label{ineq:lse_2}
	nF_n(\vartheta\delta_n)&\geq \frac{1}{2}\underline{K}_{P,\alpha}\pnorm{\xi_1}{1}\vartheta^{1-\alpha/2} \cdot n^{\frac{\alpha}{2+\alpha}}-\vartheta^2 n^{\frac{\alpha}{2+\alpha}}\\
	& \geq \frac{1}{4}\underline{K}_{P,\alpha}\pnorm{\xi_1}{1}\vartheta^{1-\alpha/2} \cdot n^{\frac{\alpha}{2+\alpha}}\equiv \underline{C}_{\xi,\alpha,P}\cdot  n^{\frac{\alpha}{2+\alpha}}.\nonumber
	\end{align}
	First we choose $M=1/\mathfrak{p}_1$ so that with probability at least $\mathfrak{p}_1$, (\ref{ineq:lse_1}) and (\ref{ineq:lse_2}) hold simultaneously. Then we choose $\rho=\min\{(\mathfrak{p}_1 \underline{C}_{\xi,\alpha,P}/2\bar{C}_{\xi,\alpha,P})^{\frac{2}{2-\alpha}},\vartheta/2\}$ to conclude $F_n(\vartheta \delta_n)>E_n(\rho\delta_n)$ with probability at least $\mathfrak{p}_1$. 
\end{proof}

Now we have completed the program outlined in Proposition \ref{prop:proof_route_lower_bound_lse}.

\begin{proof}[Proof of Theorem \ref{thm:lse_lower_bound_general}: claim (1)]
	Recall that the regression function class is taken from Lemma \ref{lem:existence_alpha_full} with $f_0\equiv 0$.  Combining the proof outline Proposition \ref{prop:proof_route_lower_bound_lse}, with Proposition \ref{prop:rate_lower_bound_lse_generic}, we see that there exists an event with probability at least $\mathfrak{p}_1=\mathfrak{p}_1(\alpha,\xi,P)>0$, on which at least one 
	 least squares estimator $f_n^\ast$ over $\mathcal{F}$ satisfies $
	\pnorm{f^\ast_n-f_0}{L_2(P)}\geq  \rho\cdot n^{-\frac{1}{2+\alpha}}$, where $\rho>0$ is a (small) constant independent of $n$. The claim now follows by bounding the expectation from below on this event.
\end{proof}


\subsubsection{Remaining proofs for Theorem \ref{thm:lse_lower_bound}}

Here we prove the second claim of Theorem \ref{thm:lse_lower_bound_general}. Without loss of generality, we only consider $d=1$, and the probability measure $P$ is assumed to be uniform for simplicity. To this end, let $ \tilde{\mathcal{F}}_n\equiv \{\bm{1}_{[a,b]}: a,b \in [0,1]\cap \Q,  b-a\geq \delta_n^2\}\cup \{0\}$, and $\tilde{\mathcal{G}}_n\equiv \{g \in C^{1/\alpha}([0,1]): Pg^2\geq \delta_n^2\}$, and $f_0=0$, where $\delta_n=\rho n^{-1/2+1/2p'}$ with $\frac{1}{p}-\frac{1}{p'}=\epsilon$ for some numeric constants $\epsilon, \rho>0$ to be specified later. Let ${\mathcal{F}}\equiv \tilde{\mathcal{F}}_n\cup \tilde{\mathcal{G}}_n$.

In the current case $E_n(\delta)=0$ for $\delta< \delta_n$ and hence our goal is to give a lower bound for $F_n(\delta_n)$. We mimic the proof strategy of the first claim of Theorem \ref{thm:lse_lower_bound_general} by (i) giving a lower bound for the multiplier empirical process in expectation, and then (ii) using the Paley-Zygmund moment argument to translate the lower bound in probability. The arguments are somewhat delicate due to the fact that the $L_{p,1}$ norm is stronger than the $L_p$ norm. For notational simplicity, let $Z\equiv \sup_{f \in \tilde{\mathcal{F}}_n: Pf^2\leq \delta_n^2}\abs{\sum_{i=1}^n \xi_i f(X_i)}$, and $\tilde{Z}\equiv \sup_{f \in \tilde{\mathcal{F}}_n: Pf^2\leq \delta_n^2}\abs{\sum_{i=1}^n 2\xi_i f(X_i)-f^2(X_i)+Pf^2 }$.

\begin{lemma}\label{lem:lower_bound_in_prob_noise_dom}
	Suppose that ${\xi}_1,\ldots,{\xi}_n$ are i.i.d. symmetric random variables with $ \Prob(\abs{{\xi}_1}>t)=1/(1+\abs{t}^{p'})$. Further suppose that $\rho \leq (64/e^3)^{1/6}$, $\epsilon^{-1/2}\vee 3\leq p \leq \log n/\log \log n$ and $
	n\geq \min\{n\geq 2: \delta_n =\rho n^{-1/2+1/2p'}\geq \sqrt{\log n/n}\}$. 
	Then there exists some absolute constant $C_1>0$, and for any $C_2>0$, there exists some constant $C_3=C_3(C_2)>0$ such that
	\begin{align*}
	\Prob\bigg(\tilde{Z}\geq \frac{1}{16}n^{1/p'-1/(p')^2}-C_3 n^{1/2p'}\sqrt{\log n}\bigg)\geq C_1 \left((\epsilon p')\wedge 1\right)^2 n^{-4\epsilon}- e^{-C_2\log n}.
	\end{align*}
\end{lemma}

We need the following before the proof of Lemma \ref{lem:lower_bound_in_prob_noise_dom}.

\begin{lemma}\label{lem:lower_bound_mep_noise_dom}
	Suppose $\rho \leq (64/e^3)^{1/6}$ and $\xi_1,\ldots,\xi_n$ are i.i.d. mean-zero random variables. Then for $n\geq 2$, we have
	$
	\E Z\geq \frac{1}{2}\E \max_{1\leq j\leq 4^{-1} n^{1-1/p'}} \abs{\xi_j}.
	$
\end{lemma}
\begin{proof}
	Let $I_j\equiv [(j-1) \delta_n^2 ,j \delta_n^2]\subset [0,1]$ for $j=1,\ldots,N$ where $N=\delta_n^{-2}\leq \rho^{-2}n$. Note that for any $c \in (0,1)$,
	\begin{align*}
	&\Prob\left(X_1,\ldots,X_n \textrm{ lie in at most } cN \textrm{ intervals among } \{I_j\}_{j=1}^N\right)\\
	& =\Prob\left(\cup_{\abs{\mathcal{I}}=cN} \{ X_1,\ldots,X_n \in \cup_{i \in \mathcal{I}} I_i\}\right)\\
	&\leq \binom{N}{cN}c^n\leq e^{cN\log(e/c)-n\log(1/c)}\leq e^{\left(c\rho^{-2} \log(e/c)-\log(1/c)\right) n}.
	\end{align*}
	By choosing $c=\rho^2/4$, the exponent in the above display can be further bounded by $
	\frac{1}{4}\log(e/c)-\log(1/c)=\frac{1}{4}\log(c^3 e)=\frac{1}{4}\log(e \rho^6/64)\leq -\frac{1}{2}$
	where the last inequality follows by the assumption that $\rho \leq (64/e^3)^{1/6}$. Hence we conclude that on an event $\mathcal{E}$ with probability at least $1-0.61^n$, the samples $X_1,\ldots,X_n$ must occupy at least $\rho^2 N/4$ many intervals among $\{I_j\}_{j=1}^N$. This implies that
	\begin{align*}
	\E Z=\E \sup_{f \in \tilde{\mathcal{F}}_n: Pf^2\leq \delta_n^2}\biggabs{\sum_{i=1}^n \xi_i f(X_i)}& \geq \E \max_{1\leq j\leq N} \biggabs{\sum_{i \in I_j} \xi_i \bm{1}_{I_j}(X_i)}\bm{1}_{\mathcal{E}}\geq \frac{1}{2}\E \max_{1\leq j\leq \rho^2 N/4} \abs{\xi_j}
	\end{align*}
	where we used the same arguments as in the proof of Theorem \ref{thm:lower_bound_mep}. The claim now follows by noting $\rho^2 N/4 = \rho^2 \delta_n^{-2}/4= n^{1-1/p'}/4$.
\end{proof}

We also need some auxiliary results.

\begin{lemma}\label{lem:size_maxima_multiplier}
	Suppose that ${\xi}_1,\ldots,{\xi}_n$ are i.i.d. symmetric random variables with $ p(t)\equiv \Prob(\abs{{\xi}_1}>t)=1/(1+\abs{t}^p)$. Then for any $1\leq q<p$, and $n\geq 2$,
	\begin{align*}
	\frac{1}{4} n^{1/p}\leq \E \max_{1\leq i\leq n}\abs{\xi_i}\leq \big(\E \max_{1\leq i\leq n}\abs{\xi_i}^q\big)^{1/q}\leq \left(\frac{ p+q}{p-q} \right)^{1/q} n^{1/p}.
	\end{align*}
\end{lemma}

We need the following exact characterization concerning the size of maxima of a sequence of independent random variables due to \cite{gine1983central}, see also Corollary 1.4.2 of \cite{de2012decoupling}.

\begin{lemma}\label{lem:characterization_maxima}
	Let $\xi_1,\ldots,\xi_n$ be a sequence of independent non-negative random variables such that $\pnorm{\xi_i}{r}<\infty$ for all $1\leq i\leq n$. For $\lambda>0$, set $
	\delta_0(\lambda)\equiv \inf\left\{t>0: \sum_{i=1}^n \Prob(\xi_i>t)\leq \lambda\right\}$.
	Then
	\begin{align*}
	\frac{1}{1+\lambda}\sum_{i=1}^n \E \xi_i^r\bm{1}_{\xi_i>\delta_0}\leq \E\max_{1\leq i\leq n} \xi_i^r\leq \frac{1}{1\wedge \lambda}\sum_{i=1}^n \E \xi_i^r\bm{1}_{\xi_i>\delta_0}.
	\end{align*}
\end{lemma}

\begin{proof}[Proof of Lemma \ref{lem:size_maxima_multiplier}]
	For $\lambda\equiv 1$ in Lemma \ref{lem:characterization_maxima}, $
	\delta_0 = \inf\{t>0: np(t)\leq 1\}=(n-1)^{1/p}$. Lemma \ref{lem:characterization_maxima} now yields that for $q<p$,
	\begin{align*}
	\E \max_{1\leq i\leq n} \abs{{\xi}_i}^q &\leq  n \E \abs{{\xi}_1}^q \bm{1}_{\abs{{\xi}_1}>\delta_0}\\
	&= n \left[\Prob(\abs{{\xi}_1}>\delta_0)\int_0^{\delta_0} qu^{q-1}\ \d{u}+\int_{\delta_0}^{\infty} qu^{q-1} \Prob(\abs{{\xi}_1}>u)\ \d{u}\right]\\
	&\leq \frac{n\delta_0^q}{1+\delta_0^p}+qn \int_{\delta_0}^{\infty} \frac{1}{u^{p-q+1}}\ \d{u}\\
	&= (n-1)^{q/p}+ \frac{q}{p-q} \frac{n}{n-1} (n-1)^{q/p}\leq \frac{p+q}{p-q}n^{q/p}
	\end{align*}
	since $n\geq 2$. For a lower bound for $\E \max_{1\leq i\leq n}\abs{{\xi}_i}$, we proceed similarly as above by using $1+u^p\leq 2u^p$ on $[\delta_0,\infty)$ for $n\geq 2$:
	\begin{align*}
	\E \max_{1\leq i\leq n} \abs{{\xi}_i}&
	\geq \frac{n}{2}\left[\frac{\delta_0}{1+\delta_0^p}+ \int_{\delta_0}^{\infty} \frac{1}{2u^{p}}\ \d{u} \right]\geq \frac{(n-1)^{1/p}}{2}\geq \frac{1}{4}n^{1/p}.
	\end{align*}
	This completes the proof.
\end{proof}

We also need Talagrand's concentration inequality \cite{talagrand1996new} for the empirical process in the form given by Bousquet \cite{bousquet2003concentration}, recorded as follows.

\begin{lemma}\label{lem:talagrand_conc_ineq}[Theorem 3.3.9 of \cite{gine2015mathematical}]
	Let $\mathcal{F}$ be a countable class of real-valued measurable functions such that $\sup_{f \in \mathcal{F}} \pnorm{f}{\infty}\leq b$. Then
	\begin{align*}
	\Prob\bigg(\sup_{f \in \mathcal{F}}\abs{\G_n f} \geq \E\sup_{f \in \mathcal{F}}\abs{\G_n f} +\sqrt{2\bar{\sigma}^2 x}+b x/3\sqrt{n} \bigg)\leq e^{-x},
	\end{align*}
	where $\bar{\sigma}^2\equiv \sigma^2+2b n^{-1/2} \E \sup_{f \in \mathcal{F}} \abs{\G_n f}$ with $\sigma^2\equiv \sup_{f \in \mathcal{F}} \mathrm{Var}_P f$, and $\mathbb{\G}_n\equiv \sqrt{n}(\Prob_n-P)$.
\end{lemma}
In applications, since
\begin{align*}
\sqrt{2\bar{\sigma}^2 x}&\leq \sqrt{2\sigma^2 x}+\sqrt{4(bx/\sqrt{n}) \E \sup_{f \in \mathcal{F}} \abs{\G_n f}}\\
&\leq \sqrt{2\sigma^2 x}+\delta^{-1} (bx/\sqrt{n})+\delta \E \sup_{f \in \mathcal{F}} \abs{\G_n f}
\end{align*}
by the elementary inequality $2ab\leq \delta^{-1} a^2+\delta b^2$, we have for any $\delta>0$,
\begin{align}\label{ineq:talagrand_ineq_weak_var}
\Prob\bigg(\sup_{f \in \mathcal{F}}\abs{\G_n f} \geq (1+\delta)\E\sup_{f \in \mathcal{F}}\abs{\G_n f} +\sqrt{2\sigma^2 x}+(3^{-1}+\delta^{-1})b x/\sqrt{n} \bigg)\leq e^{-x}.
\end{align}
We will mainly use the above form (\ref{ineq:talagrand_ineq_weak_var}) in the proofs.

\begin{proof}[Proof of Lemma \ref{lem:lower_bound_in_prob_noise_dom}]
	The proof is divided into two steps. In the first step, we handle $Z$, i.e. the multiplier empirical process part. In the second step we handle the residual term, i.e. the purely empirical process part.
	
	\noindent\textbf{(Step 1)} We first claim that
	there exists some absolute constant $C_1>0$ such that
	\begin{align}\label{ineq:lower_bound_in_prob_noise_dom_1}
	\Prob\big(Z\geq \frac{1}{32}n^{1/p'-1/(p')^2}\big)\geq C_1 \left((\epsilon p')\wedge 1\right)^2 n^{-4\epsilon}.
	\end{align}
	Let $\mathcal{G}$ be the class of indicators functions $\bm{1}_{[a,b]}$ with $0\leq a\leq b\leq 1$. Then since $\delta_n\geq \sqrt{\log n/n}$, by local maximal inequalities for the empirical process (Proposition \ref{prop:local_maximal_ineq}), 
	\begin{align}\label{ineq:lower_bound_in_prob_noise_dom_2}
	\E \sup_{g \in \mathcal{G}: Pg^2\leq \delta_n^2} \biggabs{\sum_{i=1}^n \epsilon_i g(X_i)}\lesssim \sqrt{n} \delta_n\sqrt{\log(1/\delta_n)}\lesssim  n^{1/2p'}\sqrt{\log n}
	\end{align}
	where in the last inequality we used $\rho\lesssim 1$. Applying Theorem \ref{thm:local_maximal_generic} and noting that $p\leq \log n/\log \log n$ implying $n^{1/2p'}\sqrt{\log n}\leq n^{1/2p}\sqrt{\log n}\leq n^{1/p}$, we see that for some absolute constant $C>0$,
	$
	\E Z \leq \E \sup_{g \in \mathcal{G}: Pg^2\leq \delta_n^2} \abs{\sum_{i=1}^n \xi_i g(X_i)}\leq C n^{1/p} \pnorm{\xi_1}{p,1}.
	$
	By the moment inequality Lemma \ref{lem:p_moment_estimate}, we have for any $q\geq 1$, 
	\begin{align*}
	\E Z^q &\leq C_q\big( (\E Z)^q+(\sqrt{n}\pnorm{\xi_1}{2}\delta_n)^q+\E \max_{1\leq i\leq n}\abs{\xi_i}^q\big)\\
	&\leq C_q'\big(n^{q/p} \pnorm{\xi_1}{p,1}^q+n^{q/2p'}\pnorm{\xi_1}{2}^q+\E \max_{1\leq i\leq n}\abs{\xi_i}^q\big)\\
	&\leq 2C_q'\big(n^{q/p} (\pnorm{\xi_1}{p,1}\vee\pnorm{\xi_1}{2})^q+\E \max_{1\leq i\leq n}\abs{\xi_i}^q\big).
	\end{align*}
	Now using the Paley-Zygmund inequality (Lemma \ref{lem:paley_zygmund}) and Lemma \ref{lem:lower_bound_mep_noise_dom}, we see that 
	\begin{align*}
	\Prob \big(Z>\frac{1}{2}\E Z\big)&\geq 2^{-q'} \bigg(\frac{\E Z}{(\E Z^q)^{1/q}}\bigg)^{q'}\\
	&\geq C_q''\bigg(\frac{\E \max_{1\leq j\leq 4^{-1}n^{1-1/{p'}}}\abs{\xi_j} }{n^{1/p}(\pnorm{\xi_1}{p,1}\vee\pnorm{\xi_1}{2})+ \left(\E \max_{1\leq i\leq n}\abs{\xi_i}^q\right)^{1/q}   }\bigg)^{q'}.
	\end{align*}
	By Lemma \ref{lem:size_maxima_multiplier}, the above display can be further estimated from below by
	\begin{align*}
	&C_q'''\bigg(\frac{ n^{1/p'-1/(p')^2}}{ n^{1/p}(\pnorm{\xi_1}{p,1}\vee\pnorm{\xi_1}{2}) +[ (p'+q)/(p'-q) ]^{1/q} n^{1/p'}}\bigg)^{q'}\\
	& \geq C' \bigg(\frac{ n^{1/p'-1/(p')^2}}{ n^{1/p}(\pnorm{\xi_1}{p,1}\vee\pnorm{\xi_1}{2}) + n^{1/p'}}\bigg)^{2} \quad (\textrm{choose } q=q'=2)\\
	&\geq C'' \left((\epsilon p')\wedge 1\right)^2 n^{-\frac{2}{p}+\frac{2}{p'}-\frac{2}{(p')^2}}\geq C'' \left((\epsilon p')\wedge 1\right)^2 n^{-4\epsilon}
	\end{align*}
	where in the last line we use the following facts: (i) $\frac{1}{p}-\frac{1}{p'}=\epsilon$, (ii) $p'\geq p\geq \epsilon^{-1/2}$ and (iii)
	\begin{align*}
	\pnorm{{\xi}_1}{p,1}&=\int_0^\infty \Prob(\abs{{\xi}_1}>t)^{1/p}\ \d{t}= \int_0^\infty \frac{1}{(1+ t^{p'})^{1/p}}\ \d{t}\\
	&\asymp 1+ \int_1^\infty \frac{\d{t}}{t^{p'/p}}= 1 + \frac{1}{p'/p-1}=1+\frac{1}{\epsilon p'},\\
	\pnorm{{\xi}_1}{2}& \leq \pnorm{{\xi}_1}{2,1} \asymp 1+\frac{1}{(p'/2)-1}\asymp 1.
	\end{align*}
	The proof of (\ref{ineq:lower_bound_in_prob_noise_dom_1}) is complete by noting that 
	\begin{align*}
	\E Z\geq \frac{1}{2} \E \max_{1\leq j\leq 4^{-1}n^{1-1/p'}}\abs{\xi_j}\geq \frac{1}{8} 4^{-1/p'} n^{1/p'-1/(p')^2}\geq 16^{-1}n^{1/p'-1/(p')^2}.
	\end{align*}
	
	\noindent \textbf{(Step 2)} We next claim that for any $C_2>0$,
	there exists some absolute constant $C_3>0$ such that
	\begin{align}\label{ineq:lower_bound_in_prob_noise_dom_3}
	\Prob\bigg(\sup_{f \in \tilde{\mathcal{F}}_n: Pf^2\leq \delta_n^2 } \bigabs{\sum_{i=1}^n (f^2(X_i)-Pf^2)}\leq C_3 n^{1/2p'}\sqrt{\log n}\bigg)\geq 1- e^{-C_2\log n}.
	\end{align}
	Note that by contraction principle and (\ref{ineq:lower_bound_in_prob_noise_dom_2}),
	\begin{align*}
	\E \sup_{f \in \tilde{\mathcal{F}}_n: Pf^2\leq \delta_n^2 } \biggabs{\sum_{i=1}^n (f^2(X_i)-Pf^2)}\lesssim \E \sup_{g \in \mathcal{G}: Pg^2\leq \delta_n^2} \biggabs{\sum_{i=1}^n \epsilon_i g(X_i)}\lesssim  n^{1/2p'}\sqrt{\log n}.
	\end{align*}
	The claim (\ref{ineq:lower_bound_in_prob_noise_dom_3}) now follows from the above display combined with Talagrand's concentration inequality (Lemma \ref{lem:talagrand_conc_ineq}) applied with $x=C_2\log n$. 
	
	Now the claimed inequality in the lemma follows by considering the event that is the intersection of the events indicated in (\ref{ineq:lower_bound_in_prob_noise_dom_1}) and (\ref{ineq:lower_bound_in_prob_noise_dom_3}). 
\end{proof}

Now we are in a good position to prove the second claim of Theorem \ref{thm:lse_lower_bound_general}.

\begin{proof}[Proof of Theorem \ref{thm:lse_lower_bound_general}: claim (2)]
	Suppose that $n^{1/(p')^2}\leq 2, p'< \frac{\log n}{2\log(64C_3)+\log\log n}$ and $\rho=1/8$. Then
	\begin{align*}
	&\frac{1}{16}n^{1/p'-1/(p')^2}-C_3 n^{1/2p'}\sqrt{\log n}-n\delta_n^2 \\
	&= \bigg(\frac{1}{16 n^{1/(p')^2}}-\rho^2\bigg) n^{1/p'}-C_3 n^{1/2p'}\sqrt{\log n}\geq \frac{1}{64} n^{1/p'} -C_3 n^{1/2p'}\sqrt{\log n}> 0.
	\end{align*}
	Since  $
	\left\{\tilde{Z}\geq \frac{1}{16}n^{1/p'-1/(p')^2}-C_3 n^{1/2p'}\sqrt{\log n}\right\}\subset \left\{\tilde{Z}-n\delta_n^2>0\right\}\subset \left\{F_n(\delta_n)>E_n(\delta_n/2)\right\}$, 
	it follows from Lemma \ref{lem:lower_bound_in_prob_noise_dom} that
	\begin{align*}
	\Prob(\mathcal{E}_n)&\equiv \Prob\left(F_n(\delta_n)>E_n(\delta_n/2)\right)\geq  C_1 \left((\epsilon p')\wedge 1\right)^2 n^{-4\epsilon}- e^{-C_2\log n},
	\end{align*}
	provided further $\epsilon\geq 1/p^2$, $p\geq 3$, $n\geq 2$ and $ n^{1/p'}\geq 64\log n$. Equivalently, 
	\begin{align*}
	\epsilon\geq 1/p^2, \quad p\geq 3, \quad n\geq 2, \quad \sqrt{\log_2 n}\leq p'\leq\frac{\log n}{\log(64\log n)+2\log(64 C_3)}.
	\end{align*}
	Furthermore, since $p'=p/(1-p\epsilon)\leq 2p$ if $\epsilon\leq 1/2p$, it suffices to require
	\begin{align*}
	1/p^2 \leq \epsilon\leq 1/2p, \quad n\geq n_\delta \vee e^{64C_3^2},\quad \sqrt{\log_2 n}\leq p\leq (\log n)^{1-\delta}\bigg[\leq \frac{\log n}{2\log(64\log n)}\bigg],
	\end{align*}
	where $n_\delta \equiv\min\{n \geq 2: (\log n)^\delta\geq 2\log(64 \log n)\}$. Hence for $n$ in the indicated range (i.e. sufficiently large depending on $\delta,C_3$), we have
	\begin{align*}
	\Prob(\mathcal{E}_n)\geq C_1 (\log n)^{-2} n^{-4\epsilon}- e^{-C_2\log n}\geq C (\log n)^{-2} n^{-4\epsilon},
	\end{align*}
	where the first inequality follows from $\epsilon p'\geq \epsilon p\geq 1/p\geq 1/\log n$, and the second inequality follows for $n$ sufficiently large by choosing $C_2=3$ since $n^{-4\epsilon}\geq n^{-2}$. By Proposition \ref{prop:proof_route_lower_bound_lse},
	\begin{align*}
	\E \pnorm{f^\ast_n-f_0}{L_2(P)}&\geq \E \pnorm{f^\ast_n-f_0}{L_2(P)}\bm{1}_{\mathcal{E}_n}\\
	&\geq \frac{\delta_n}{2}\cdot C (\log n)^{-2} n^{-4\epsilon} \geq C' n^{-1/2+1/2p - 4.5\epsilon}(\log n)^{-2}.
	\end{align*}
	For any $1/(1-\delta)<a< 2$ so that $p= (\log n)^{1/a}$, we may choose $\epsilon= p^{-a}$. The claim then follows by noting $n^{-\epsilon}=n^{-1/\log n}= e^{-1}$, and $\pnorm{\xi_1}{p,1}\asymp 1+(\epsilon p)^{-1} = 1+(\log n)^{1-1/a} \lesssim \log n$.
\end{proof}

\section{Proof of impossibility results}

In this section we prove Propositions \ref{prop:impossibility} and \ref{prop:impossibility_LSE}.

\begin{proof}[Proof of Proposition \ref{prop:impossibility}]
	Let $X_i$'s be i.i.d. symmetric random variables with the tail probability $\Prob(\abs{X_1}>x)=x^{-(2+\delta)}$. Let $\xi_i =\abs{X_i}^{2/p} \epsilon_i$, where $\epsilon_i$'s are Rademacher random variables independent of all other random variables. Then it is easy to check that (i)-(ii) hold (in particular, $\delta>0$ guarantees $\pnorm{\xi_1}{p,1}<\infty$). Now take any $\mathcal{F}$ satisfying the entropy condition (F) with exponent $\alpha=2/(\gamma-1) \in (0,2)$, and let $\mathcal{F}_k\equiv \{f \in \mathcal{F}: Pf^2<\delta_k^2\}\cup \{\delta_k \mathfrak{e}\}$ where $\delta_k\equiv k^{-1/(2+\alpha)}$ and $\mathfrak{e}(x)=x$. Then by Proposition \ref{prop:local_maximal_ineq}, we have $\E\pnorm{\sum_{i=1}^k \epsilon_i f(X_i)}{{\mathcal{F}}_k}\lesssim k^{1/\gamma}$. On the other hand, since $p<4/\delta$, applying Theorem 3.7.2 of \cite{durrett2010probability} we see that for $n$ large enough,
	\begin{align*}
	\E \biggpnorm{\sum_{i=1}^n \xi_i f(X_i)}{{\mathcal{F}}_n} &\geq \delta_n \E \biggabs{\sum_{i=1}^n \xi_i \mathfrak{e}(X_i)}= \delta_n \E \biggabs{\sum_{i=1}^n \epsilon_i\abs{X_i}^{1+2/p}} \\
	&\gtrsim \delta_n n^{\frac{1+2/p}{2+\delta}} \equiv r_n.
	\end{align*}
	The equality in the first line of the above display follows from
	\begin{align*}
	\xi_i \mathfrak{e}(X_i) =  \abs{X_i}^{2/p} \epsilon_i X_i=_d \abs{X_i}^{2/p} \epsilon_i \epsilon_i'\abs{X_i}=_d \epsilon_i \abs{X_i}^{1+2/p}
	\end{align*}
	by symmetry of $X_i$'s ($\epsilon_i'$'s are independent copies of $\epsilon_i$'s). 
	
	Now in order that $r_n\gg n^{1/\gamma}$, 
	\begin{align*}
	\frac{1+2/p}{2+\delta}-\frac{1}{2+\alpha}>\frac{1}{\gamma}=\frac{\alpha}{2+\alpha}&\Leftrightarrow \frac{1+2/p}{2+\delta}>\frac{1+\alpha}{2+\alpha}\\
	&\Leftrightarrow p<\frac{2(1+2/\alpha)}{1+(1+1/\alpha) \delta}.
	\end{align*}
	Hence it suffices to require that $p<2\gamma/(1+\gamma\delta)$. On the other hand, in order that $r_n\gg n^{1/p}$, it suffices to require that 
	\begin{align*}
	\frac{1+2/p}{2+\delta}-\frac{1}{2+\alpha}>\frac{1}{p} \Leftrightarrow \frac{1}{2+\delta}-\frac{1}{2+\alpha}>\frac{\delta}{2+\delta}\frac{1}{p}.
	\end{align*}
	Since $p\geq 2$, we only need to check that 
	\begin{align*}
	\frac{1}{2+\delta}-\frac{1}{2+\alpha}>\frac{\delta/2}{2+\delta},
	\end{align*}
	which holds since we choose $\alpha>4\delta$ (which is equivalent to $\gamma=1+2/\alpha<1+1/(2\delta)$). This completes the proof.
\end{proof}

\begin{proof}[Proof of Proposition \ref{prop:impossibility_LSE}]
	We only sketch the proof here. Let $X_i$'s, $\xi_i$'s and $\mathfrak{e}$ be defined as in the proof of Proposition \ref{prop:impossibility}, $f_0=0$, and $\mathcal{F}$ be any function class defined on $[0,1]$ satisfying the entropy condition (F) with exponent $\alpha \in (0,2)$. Let $\delta_n>0$ be determined later on, and $\mathcal{F}_n\equiv \{f  \in \mathcal{F}: Pf^2\geq \delta_n^2\}\cup \{\delta_n \mathfrak{e}\}\cup \{0\}$. Then $E_n(\delta_n/2)=0$ and we only need to show that with positive probability, $F(\delta_n)>0$. To see this, note that we have shown in Proposition \ref{prop:impossibility} that $\abs{\sum_{i=1}^n \xi_i\mathfrak{e}(X_i)}\gtrsim  n^{\frac{1+2/p}{2+\delta}}$ with positive probability for $n$ large enough. Furthermore, by Talagrand's concentration inequality (cf. Lemma \ref{lem:talagrand_conc_ineq}), we have that with overwhelming probability,
	\begin{align*}
	n^{-1/2}\sup_{f \in \mathcal{F}_n}\biggabs{\sum_{i=1}^n (f^2(X_i)-Pf^2)}\lesssim \E \sup_{f \in \mathcal{F}_n} \abs{\G_n f^2}+\delta_n+n^{-1/2}\lesssim_\alpha \delta_n^{1-\alpha/2},
	\end{align*}
	since $
	\E \sup_{f \in \mathcal{F}_n} \abs{\G_n f^2}\lesssim \E \sup_{f \in \mathcal{F}_n} \abs{n^{-1/2}\sum_{i=1}^n \epsilon_i f(X_i)}\lesssim \delta_n^{1-\alpha/2}$
	by using the standard contraction principle and Proposition \ref{prop:local_maximal_ineq}. Hence if $\delta_n\gtrsim n^{-\frac{1}{2+\alpha}}$, then for $n$ large enough, we have with positive probability,
	\begin{align*}
	nF_n(\delta_n)&\geq \delta_n\biggabs{\sum_{i=1}^n \xi_i \mathfrak{e}(X_i)}-\sup_{f \in \mathcal{F}_n}\biggabs{\sum_{i=1}^n (f^2(X_i)-Pf^2)}-n\delta_n^2\\
	&\geq C_1 \delta_n n^{\frac{1+2/p}{2+\delta}} -C_2 n^{1/2} \delta_n^{1-\alpha/2}-n\delta_n^2\\
	&\geq C_1 n^{\beta}(n^{\frac{\alpha}{2+\alpha}}\vee n^{\frac{1}{p}})-C_2 n^{1/2} \delta_n^{1-\alpha/2}-n\delta_n^2
	\end{align*}
	for some $\beta\equiv \beta(\delta,\alpha,p)>0$, where the last inequality follows from the arguments in the proof of Proposition \ref{prop:impossibility}, by assuming that $2<\gamma<1+1/(2\delta)$ and $2\leq p<\min\{4/\delta,2\gamma/(1+\gamma\delta)\}$ where $\gamma=1+2/\alpha$. This condition is equivalent to $4\delta<\alpha<2$ and $2\leq p< \min\{4/\delta, (2+4/\alpha)/(1+(1+2/\alpha)\delta)\}$. Hence we may choose $\delta_n=C_3 n^{\beta/2}(n^{-\frac{1}{2+\alpha}}\vee n^{-1/2+1/(2p)})$ for some constant $C_3>0$ to ensure that the last line of the above display is $>0$ for $n$ large enough.
\end{proof}

\section{Remaining proofs I}\label{section:proof_remaining}

\subsection{Proof of Lemma \ref{lem:asymp_concavitify}}

\begin{proof}[Proof of Lemma \ref{lem:asymp_concavitify}]
	Without loss of generality we assume that $a_n\leq 1$ for all $n=0,1,\ldots$. For any $\epsilon \in (0,1)$, since $a_n$ vanishes asymptotically, there exists some $N_\epsilon$ for which $a_n\leq \epsilon$ as long as $n\geq N_\epsilon$. Consider
	\begin{align*}
	\psi_\epsilon(t)\equiv 
	\begin{cases}
		\varphi(t), & t\leq N_\epsilon;\\
		(1-\epsilon)\varphi(N_\epsilon)+\epsilon \varphi(t), &t>N_\epsilon.
	\end{cases}
	\end{align*}
	Then it is easy to verify that $\psi_\epsilon$ is a concave function and majorizes $n\mapsto a_n \varphi(n)$ [since $\psi_\epsilon(n)\geq \epsilon \varphi(n)\geq a_n\varphi(n)$ for $n\geq N_\epsilon$ and $\psi_\epsilon(n)=\varphi(n)\geq a_n\varphi(n)$ for $n<N_\epsilon$ by the assumption that $a_n\leq 1$]. Hence by definition of $\psi$, it follows that $\limsup_{t \to \infty}\psi(t)/\varphi(t)\leq \lim_{t\to \infty}\psi_\epsilon(t)/\varphi(t)=\epsilon$. The claim now follows by taking $\epsilon \to 0$. 
\end{proof}

\subsection{Proof of Lemma \ref{lem:gaussian_approx}}\label{section:proof_gaussian_approx}

We need some auxiliary lemmas before the proof of Lemma \ref{lem:gaussian_approx}.

\begin{lemma}\label{lem:dim_free_F_beta}
	Let $F_\beta:\R^d \to \R$ be the soft-max function defined by $
	F_\beta(x) = \beta^{-1}\log \big(\sum_{i=1}^d \exp(\beta x_i)\big)$.
	Then $
	\sup_{x \in \R^d}\sum_{j,k,l,m=1}^d \abs{\partial_{jklm} F_\beta(x)}\leq 25\beta^3$.
\end{lemma}
\begin{proof}
	Let $\pi_j(x) = \partial_j F_\beta(x) = \exp(\beta x_j)/\sum_{i=1}^d \exp(\beta x_i)$ and $\delta_{ij}=\bm{1}_{i=j}$. Then it is easy to verify that (see Lemma 4.3 of \cite{chernozhukov2014gaussian})
	\begin{align*}
	\partial_{jk} F_\beta &=\partial_j (\pi_k)= \beta(\delta_{jk}\pi_j-\pi_j\pi_k),\\
	\partial_{jkl} F_\beta &= \beta^2 \big[\delta_{jk}\delta_{jl} \pi_l -\delta_{jk}\pi_j\pi_l +2\pi_j\pi_k\pi_l-(\delta_{jl}+\delta_{kl})\pi_k\pi_l\big].
	\end{align*}
	Furthermore taking the derivative with respect to  $x_m$, we have
	\begin{align*}
	\partial_{jklm} F_\beta & = \beta^2 \big[\delta_{jk}\delta_{jl} \partial_m(\pi_l) -\delta_{jk}\partial_m(\pi_j\pi_l) +2\partial_m(\pi_j\pi_k\pi_l)-(\delta_{jl}+\delta_{kl})\partial_m(\pi_k\pi_l)\big]\\
	& = \beta^2\big[(I)-(II)+2(III)-(IV)\big],
	\end{align*}
	where 
	$
	(I) = \beta \delta_{jk}\delta_{jl}\big(\delta_{lm}\pi_l-\pi_l\pi_m\big)$, 
	$(II) = \beta \delta_{jk}\big[(\delta_{jm}+\delta_{lm})\pi_j\pi_l-2\pi_j\pi_l\pi_m\big]$, 
	$
	(III) = \beta \big[(\delta_{jm}+\delta_{km}+\delta_{lm})\pi_j\pi_k\pi_l-3\pi_j\pi_k\pi_l\pi_m\big]$, and
	$
	(IV) = \beta (\delta_{jl}+\delta_{kl})\big[(\delta_{jm}+\delta_{km})\pi_j\pi_k-2\pi_j\pi_k\pi_m\big]$. 
	Now summing over $j,k,l,m$ by noting that $\sum_j \pi_j=1$ yields the claim.
\end{proof}

\begin{proof}[Proof of Lemma \ref{lem:gaussian_approx}]
	In the proof we write $k\equiv n$ to line up with the notation used in \cite{chernozhukov2013gaussian}. Slightly abusing the notation, we use simply $X_i$'s to denote $\epsilon_i X_i$'s. Let $Y_1,\ldots,Y_n$ be centered i.i.d. Gaussian random vectors in $\R^d$ such that $\E Y_1Y_1^\top = \E X_1 X_1^\top$. We first claim that it suffices to prove that,
	\begin{align}\label{ineq:gaussian_approx_0}
	\abs{\E F_\beta(X)-F_\beta(Y)} \lesssim \beta^3 n^{-1} \E\big(\max_{1\leq j\leq d}\abs{X_{1j}}^4\vee \abs{Y_{1j}}^4\big)\equiv \beta^3 n^{-1} \bar{M}_4,
	\end{align}
	where $X=\frac{1}{\sqrt{n}}\sum_{i=1}^n X_i \in \R^d$, $Y=\frac{1}{\sqrt{n}}\sum_{i=1}^n Y_i \in \R^d$, and $F_\beta:\R^d \to \R$ is defined by $F_\beta(x) = \beta^{-1}\log \big(\sum_{i=1}^d \exp(\beta x_i)\big)$. Once (\ref{ineq:gaussian_approx_0}) is proved, we use the inequality $	0\leq F_\beta(x) - \max_{1\leq j\leq d}x_j\leq \beta^{-1}\log d$
	to obtain that
	\begin{align*}
	&\biggabs{\E \biggabs{\max_{1\leq j\leq d}\frac{1}{\sqrt{n}} \sum_{i=1}^n X_{ij}}-\E \biggabs{\max_{1\leq j\leq d}\frac{1}{\sqrt{n}} \sum_{i=1}^n Y_{ij}} }\lesssim \beta^3 n^{-1} \bar{M}_4+ \beta^{-1}\log d.
	\end{align*}
	The conclusion of Lemma \ref{lem:gaussian_approx} follows by taking $\beta =\big(n\log d/\bar{M}_4\big)^{1/4}$ and controlling the size of the Gaussian maxima. 
	
	The proof of (\ref{ineq:gaussian_approx_0}) proceeds in similar lines as in Lemma I.1 of \cite{chernozhukov2013gaussian} by a fourth moment argument instead of a third moment one used therein. We provide details below. Let  $Z(t)=\sqrt{t}X+\sqrt{1-t}Y=\sum_{i=1}^n Z_i(t)$ be the Slepian's interpolation between $X$ and $Y$, where $Z_i(t) = \frac{1}{\sqrt{n}}(\sqrt{t}X_i+\sqrt{1-t}Y_i)$. Let $Z^{(i)}(t) = Z(t)-Z_{i}(t)$. Then,
	\begin{align}\label{ineq:gaussian_approx_1}
	\E F_\beta(X)-\E F_\beta(Y) &= \E F_\beta(Z(1))- \E F_\beta(Z(0))\\
	&=\int_0^1 \frac{\d{}}{\d{t}} \E F_\beta(Z(t))\ \d{t}= \int_0^1 \sum_{i=1}^n \sum_{j=1}^d\E \big[\partial_j F_\beta(Z(t))\dot{Z}_{ij}(t)\big]\ \d{t}\nonumber
	\end{align}
	where $
	\dot{Z}_{ij}(t)=\frac{1}{2\sqrt{n}}\big(\frac{1}{\sqrt{t}}X_{ij}-\frac{1}{\sqrt{1-t}}Y_{ij}\big)$. Now using Taylor expansion for $\partial_j F_\beta(\cdot)$ at $Z^{(i)}(t)$, we have
	\begin{align}\label{ineq:gaussian_approx_2}
	\partial_j F_\beta(Z(t))& = \partial_j F_\beta(Z^{(i)}(t))+\sum_k\partial_{jk} F_\beta(Z^{(i)}(t))Z_{ik}(t)\\
	&\qquad +\sum_{k,l}\partial_{jkl}F_\beta(Z^{(i)}(t)) Z_{ik}(t)Z_{il}(t) \nonumber \\
	&\qquad + \sum_{k,l,m}\int_0^1 \partial_{jklm}F_\beta(Z^{(i)}(t)+\tau Z_i(t))Z_{ik}(t)Z_{il}(t)Z_{im}(t)\ \d{\tau}.\nonumber
	\end{align}
	Hence (\ref{ineq:gaussian_approx_1}) can be split into four terms according to (\ref{ineq:gaussian_approx_2}). Now the key observation here is that ${Z}^{(i)}(t)$ is independent of $Z_{i\cdot},\dot{Z}_{i\cdot}$. Since $\E \dot{Z}_{ij}(t)=0$, the contribution of the first order term in (\ref{ineq:gaussian_approx_2}) vanishes. Similar observation holds for the second and third order terms. 
	For the second order term, we only need to verify $\E\dot{Z}_{ij}(t) Z_{ik}(t)=0$; this follows from the construction of $Y$ that matches the second moments of $X$: $
	\E\dot{Z}_{ij}(t) Z_{ik}(t) = \frac{1}{2n}\E \left(\frac{1}{\sqrt{t}}X_{ij}-\frac{1}{\sqrt{1-t}}Y_{ij}\right)\left(\sqrt{t}X_{ik}+\sqrt{1-t}Y_{ik}\right)= \frac{1}{2n}\left(\E X_{ij}X_{ik}-\E Y_{ij}Y_{ik}\right)=0.$ 
	For the third order term,
	\begin{align*}
	&\E\dot{Z}_{ij}(t) Z_{ik}(t) Z_{il}(t)\\
	& =\frac{1}{2n^{3/2}}\E \big(\frac{1}{\sqrt{t}}X_{ij}-\frac{1}{\sqrt{1-t}}Y_{ij}\big)\big(\sqrt{t}X_{ik}+\sqrt{1-t}Y_{ik}\big) \big(\sqrt{t}X_{il}+\sqrt{1-t}Y_{il}\big)\\
	&=(2n^{3/2})^{-1} \big(\sqrt{t}\E X_{ij}X_{ik}X_{il} -\sqrt{1-t} \E Y_{ij}Y_{ik}Y_{il}\big).
	\end{align*}
	Cross terms in the calculation of the last line in the above display all vanish by the independence and centeredness of $X$ and $Y$. The first term of the above display is $0$ since (recall $X_i$ stands for $\epsilon_i X_i$ throughout the proof) $\E \epsilon_i^3 X_{ij}X_{ik}X_{il}= \E \epsilon_i^3 \cdot \E X_{ij}X_{ik}X_{il}= 0$ by the independence between the  Rademacher $\epsilon_i$ and $X_i$. The second term is also zero by a similar argument: since $Y_i=_d \epsilon_i Y_i$ for a Rademacher random variable $\epsilon_i$ independent of $Y_i$, $\E Y_{ij}Y_{ik}Y_{il}= \E \epsilon_i^3 \cdot \E Y_{ij}Y_{ik}Y_{il}= 0$. Hence the only non-trivial contribution of (\ref{ineq:gaussian_approx_2}) in (\ref{ineq:gaussian_approx_1}) is the fourth order term:
	\begin{align*}
	&\abs{\E F_\beta(X)-F_\beta(Y)} \\
	& \leq \sum_{i=1}^n \sum_{j,k,l,m=1}^d\int_0^1 \int_0^1 \E \abs{\partial_{jklm}F_\beta(Z^{(i)}(t)+\tau Z_i(t))\dot{Z}_{ij}(t)Z_{ik}(t)Z_{il}(t)Z_{im}(t)} \d{\tau}\ \d{t}\\
	&\leq \sum_{i=1}^n \int_0^1 \int_0^1 \E \bigg[\sum_{j,k,l,m=1}^d\abs{\partial_{jklm}F_\beta(Z^{(i)}(t)+\tau Z_i(t))}\\
	&\qquad\qquad\qquad\qquad\qquad\qquad\times\max_{1\leq k,l,m\leq d}\abs{\dot{Z}_{ij}(t)Z_{ik}(t)Z_{il}(t)Z_{im}(t)}\bigg] \d{\tau}\ \d{t}\\
	&\leq 25\beta^3 \sum_{i=1}^n \int_0^1 \E \max_{1\leq j,k,l,m\leq d}\abs{\dot{Z}_{ij}(t)Z_{ik}(t)Z_{il}(t)Z_{im}(t)}\ \d{t}
	\end{align*}
	where the last inequality follows from the dimension free property of the third derivatives of soft max function $F_\beta$ (Lemma \ref{lem:dim_free_F_beta}). Now the claim (\ref{ineq:gaussian_approx_0}) follows by noting that
	\begin{align*}
	&\E \max_{1\leq j,k,l,m\leq d}\abs{\dot{Z}_{ij}(t)Z_{ik}(t)Z_{il}(t)Z_{im}(t)}
	\leq  \big(\E \max_{1\leq j\leq d} \abs{\dot{Z}_{ij}}^4\big)^{1/4} \big(\E \max_{1\leq j\leq d} \abs{Z_{ij}}^4\big)^{3/4}\\
	&\qquad\qquad \lesssim n^{-2} \bigg(\frac{1}{\sqrt{t}}\vee \frac{1}{\sqrt{1-t}}\bigg) \big(\E \max_{1\leq j\leq d} \abs{X_{1j}}^4\vee \abs{Y_{1j}}^4\big)
	\end{align*}
	and the fact that the integral $\int_0^1  \big(\frac{1}{\sqrt{t}}\vee \frac{1}{\sqrt{1-t}}\big) \ \d{t}<\infty$ converges.
\end{proof}

\section{Remaining proofs II}\label{section:remaining_proof_II}

\subsection{Proof of Lemma \ref{lem:existence_alpha_full}}

\begin{proof}[Proof of Lemma \ref{lem:existence_alpha_full}]
	Without loss of generality we assume $P$ is uniform on $\mathcal{X}\equiv [0,1]$. Take $\mathcal{F}=C^{1/\alpha}([0,1])$ to be a $1/\alpha$-H\"older class on $[0,1]$ (see Section 2.7 of \cite{van1996weak}). Let  $\tilde{\mathcal{F}}\equiv \mathcal{F}\cup \mathcal{G}$. For any discrete probability measure $Q$ on $\mathcal{X}= [0,1]$,
	\begin{align*}
	\mathcal{N}(\epsilon, \tilde{\mathcal{F}},L_2(Q))&\leq \mathcal{N}(\epsilon, \mathcal{F},L_2(Q))+\mathcal{N}(\epsilon, \mathcal{G},L_2(Q))\\
	&\leq \mathcal{N}(\epsilon, \mathcal{F},L_\infty([0,1]))+\sup_Q \mathcal{N}(\epsilon,\mathcal{G},L_2(Q)),
	\end{align*}
	where the last inequality follows from the fact that any $\epsilon$-cover of $\mathcal{F}$ in $L_\infty$ metric on $[0,1]$ induces an $\epsilon$-cover on the function class $\mathcal{F}$ under any $L_2(Q)$ on $\mathcal{X}$. Now by Theorem 2.7.1 of \cite{van1996weak} and the fact that $\mathcal{G}$ is a bounded VC-subgraph function class (see Section 2.6 of \cite{van1996weak}), we have the following entropy estimate:
	\begin{align}\label{ineq:lower_bound_mep_1}
	\sup_Q \log \mathcal{N}(\epsilon,\tilde{\mathcal{F}},L_2(Q))\lesssim \epsilon^{-\alpha}.
	\end{align}
	where the supremum is taken over all discrete probability measures supported on $\mathcal{X}$. On the other hand, for some small $c>0$, 
	\begin{align*}
	\mathcal{N}(c\sigma, C^{1/\alpha}([0,1])\cap L_2(\sigma),L_2([0,1]))\gtrsim \exp(c'\sigma^{-\alpha})
	\end{align*}
	holds for another constant $c'>0$ for all $\sigma>0$, due to the classical work of \cite{carl1981entropy,triebel1975interpolation} in the context of more general Besov spaces. The connection here is $C^{1/\alpha}_1([0,1])=B_{\infty,\infty}^{1/\alpha}(1)$ (in the usual notation for Besov space, see Proposition 4.3.23 of \cite{gine2015mathematical}). See also \cite{tsybakov2009introduction}, page 103-106 for an explicit construction for a (local) minimax lower bound in $L_2$ metric for the H\"older class (which is essentially the same problem), where a set of testing functions $\{f_i:i\leq M\}$ is constructed such that $M\geq 2^{m/8}$, $\pnorm{f_j-f_k}{L_2}\gtrsim m^{-1/\alpha}$ and $\pnorm{f_j}{L_2}\lesssim m^{-1/\alpha}$. Hence we see that
	\begin{align}\label{ineq:lower_bound_mep_2}
	\log \mathcal{N}(c\sigma,\tilde{\mathcal{F}}\cap L_2(\sigma),L_2([0,1]))\gtrsim \sigma^{-\alpha}.
	\end{align}
	The claim follows by combining (\ref{ineq:lower_bound_mep_1}) and (\ref{ineq:lower_bound_mep_2}).
\end{proof}

\subsection{Proof of Lemma \ref{lem:order_n_dist_uniform}}

\begin{proof}[Proof of Lemma \ref{lem:order_n_dist_uniform}]
	Note that the event in question equals 
	\begin{align*}
	\cup_{\abs{\mathcal{I}}\leq \tau n} \left\{X_1,\ldots,X_n \in \cup_{i \in \mathcal{I}} I_i\right\}.
	\end{align*}
	Hence with $K\equiv \sup_{x \in [0,1]} \abs{(\d{P}/\d{\lambda})(x)}$, the probability in question can be bounded by 
	\begin{align*}
	\sum_{k\leq \tau n}\binom{n}{k} \left(k \cdot L n^{-1}\cdot K\right)^n&\leq \sum_{k\leq \tau n}\exp\left(k \log(en/k)\right)(k/n)^n\cdot (L K)^n\\
	&=(LK)^n \sum_{k\leq \tau n} \exp\big( k \log (en/k)-n\log(n/k)\big)\\
	&\leq (e^\tau LK)^n \sum_{k\leq \tau n} \exp\big( -(n-k)\log(n/k)\big)\\
	&\leq (e^\tau LK)^n \sum_{k\leq \tau n} \exp\big( -(1-\tau)n\log(n/k)\big)\\
	& = (e^\tau LK)^n \sum_{k\leq \tau n} \bigg(\frac{k}{n}\bigg)^{(1-\tau)n}\\
	&\leq (e^\tau LK)^n n \int_0^{\tau+1/n} x^{(1-\tau)n}\ \d{x}\\
	& = \frac{n}{(1-\tau)n+1} (\tau+1/n)\cdot \bigg[e^\tau LK (\tau+1/n)^{1-\tau}\bigg]^n\\
	&\leq 0.5^{n-1},
	\end{align*}
	for $\tau<\min\{1/2, 1/8e(LK)^2\}$ and $n\geq \max\{2,8e(LK)^2\}$. The first line uses the standard inequality $\binom{n}{k}\leq n^k/k!\leq   (en/k)^k$, since $e^k=\sum_{i=0}^\infty k^i/i!\geq k^k/k!$. The last line follows since $\frac{n}{(1-\tau)n+1} (\tau+1/n)\leq \frac{n}{n/2+1}(1/2+1/2)\leq 2$ and $e^\tau LK (\tau+1/n)^{1-\tau}\leq \sqrt{e(LK)^2(\tau+1/n)}\leq 1/2$ by the conditions on $\tau$ and $n$.
\end{proof}

\subsection{Proof of Lemma \ref{lem:upper_bound_lse}}

\begin{proof}[Proof of Lemma \ref{lem:upper_bound_lse}]
	If $p\geq 1+2/\alpha$, then $\delta_n =n^{-\frac{1}{2+\alpha}}$. By local maximal inequalities for empirical processes (see Proposition \ref{prop:local_maximal_ineq}), we have
	\begin{align}\label{ineq:upper_bound_lse_0}
	\E \sup_{Pf^2\leq \rho^2 \delta_k^2} \biggabs{\sum_{i=1}^k \epsilon_i f(X_i)} &\leq C\sqrt{k}(\rho\delta_k)^{1-\alpha/2}\bigg(1\vee \frac{(\rho\delta_k)^{1-\alpha/2}}{\sqrt{k}(\rho \delta_k)^2}\bigg)\\
	&\leq C k^{\frac{\alpha}{2+\alpha}} \rho^{1-\alpha/2}  ( 1\vee \rho^{-(1+\alpha/2)})\nonumber\\
	&\leq C(\rho^{1-\alpha/2} \vee \rho^{-\alpha}) \cdot k^{\frac{\alpha}{2+\alpha}}.\nonumber
	\end{align}
	Applying Corollary \ref{cor:local_maximal_ineq_barrier} we see that
	\begin{align*}
	&\E \sup_{Pf^2\leq \rho^2 \delta_n^2} \biggabs{\sum_{i=1}^n \xi_i f(X_i)}  \leq C(\rho^{1-\alpha/2} \vee \rho^{-\alpha}) \cdot n^{\frac{\alpha}{2+\alpha}}\big(1 \vee \pnorm{\xi_1}{1+2/\alpha,1}\big).
	\end{align*}
	If $p<1+2/\alpha$, then $\delta_n = n^{-\frac{1}{2}+\frac{1}{2p}}$. In this case,
	\begin{align}\label{ineq:upper_bound_lse_1}
	\E \sup_{Pf^2\leq \rho^2 \delta_k^2} \biggabs{\sum_{i=1}^k \epsilon_i f(X_i)} &\leq C\sqrt{k}(\rho\delta_k)^{1-\alpha/2}\bigg(1\vee \frac{(\rho\delta_k)^{1-\alpha/2}}{\sqrt{k}(\rho \delta_k)^2}\bigg)\\
	&\leq C\rho^{1-\alpha/2} \cdot k^{\frac{1}{2}\left(\frac{1}{p}+\frac{\alpha}{2}\cdot \frac{p-1}{p}\right)}\bigg(1+\rho^{-(1+\alpha/2)}k^{\frac{1}{2}\left(-\frac{1}{p}+\frac{\alpha}{2}\cdot \frac{p-1}{p}\right)}\bigg)\nonumber\\
	&\leq C(\rho^{1-\alpha/2}\vee \rho^{-\alpha}) k^{\frac{1}{p}}\nonumber
	\end{align}
	where the last inequality follows from $\frac{1}{p}> \frac{\alpha}{2}\cdot \frac{p-1}{p}$ by the assumed relationship between $p$ and $\alpha$. Now apply Corollary \ref{cor:local_maximal_ineq_barrier} we have
	\begin{align*}
	\E \sup_{Pf^2\leq \rho^2 \delta_n^2} \biggabs{\sum_{i=1}^n \xi_i f(X_i)}  
	&\leq C (\rho^{1-\alpha/2}\vee \rho^{-\alpha})\cdot n^{\frac{1}{p}}(\pnorm{\xi_1}{p,1} \vee 1)
	\end{align*}
	as desired.
\end{proof}

\subsection{Proof of Lemma \ref{lem:upper_bound_lower_bound_lse}}

We first need the following.

\begin{lemma}\label{lem:interval_sample}
	Let $X_1,\ldots,X_n$ be i.i.d. $P$ on $[0,1]$ where $P$ has Lebesgue density bounded away from $0$ and $\infty$. Set $\gamma_n = \kappa_P\log n/n$ where $\kappa_P\geq 1$ is a constant depending only on $P$. Let $I_j\equiv [(j-1)\gamma_n,j\gamma_n)$ for $j=1,\ldots,n/(\kappa_P\log n)\equiv N$. Then for some $c_P>0, n_P$ sufficiently large depending on $P$, if $n\geq n_P$, with probability at least $1-2n^{-2}$, all intervals $\{I_j\}$ contain at least one and at most $c_P\log n$ samples. 
\end{lemma}
\begin{proof}
	Without loss of generality we assume that $P$ is uniform on $[0,1]$. The general case where $P$ has Lebesgue density bounded away from $0$ and $\infty$ follows from minor modification. Let $\mathcal{E}_1(\mathcal{E}_2)$ be the event that all intervals $\{I_j\}$ contain at least one sample(at most $c\log n$ samples). Then for $\kappa_P=6$,
	\begin{align*}
	\Prob\left(\mathcal{E}_1^c\right)&= \Prob\left(\cup_{1\leq j\leq N} \{I_j\textrm{ contains no samples}\}\right)\\
	&\leq N\cdot \left(1-\frac{\kappa_P\log n}{n}\right)^n\leq N e^{-\kappa_P \log n}\leq n^{-5}.
	\end{align*}
	On the other hand, 
	\begin{align*}
	\Prob(\mathcal{E}_2^c)&= \Prob\bigg(\max_{1\leq j\leq N} \biggabs{\sum_{i=1}^n \bm{1}_{I_j}(X_i)}>c\log n\bigg)\\
	&\leq \sum_{j=1}^N \Prob\bigg(\biggabs{\sum_{i=1}^n \left(\bm{1}_{I_j}(X_i)-\gamma_n\right) }>(c-6)\log n\bigg).
	\end{align*}
	Now we use Bernstein inequality in the following form (cf. (2.10) of \cite{boucheron2013concentration}): for $S=\sum_{i=1}^n (Z_i-\E Z_i)$, $v = \sum_{i=1}^n \E Z_i^2$ where $\abs{Z_i}\leq b$ for all $1\leq i\leq n$, we have $\Prob(S>t)\leq \exp\left(-\frac{t^2}{2(v+bt/3)}\right)$. We apply this with $Z_i\equiv \bm{1}_{I_j}(X_i)$ and hence $\gamma_n = \E Z_i$ and $v=\sum_{i=1}^n \gamma_n = 6\log n$, $b=1$, to see that right side of the above display can be further bounded by
	\begin{align*}
	\sum_{j=1}^N \exp\left(-\frac{(c-6)^2 \log^2 n}{2(6\log n+(c-6)\log n/3)}\right)\leq Ne^{-3 \log n}\leq n^{-2}
	\end{align*}
	by choosing $c=14$. Combining the two cases completes the proof.
\end{proof}

We also need Dudley's entropy integral bound for sub-Gaussian processes, recorded below for the convenience of the reader.

	\begin{lemma}[Theorem 2.3.7 of \cite{gine2015mathematical}]\label{lem:dudley_entropy_integral}
		Let $(T,d)$ be a pseudo metric space, and $(X_t)_{t \in T}$ be a sub-Gaussian process such that $X_{t_0}=0$ for some $t_0 \in T$. Then
		\begin{align*}
		\E \sup_{t \in T} \abs{X_t} \leq C\int_0^{\mathrm{diam}(T)} \sqrt{\log \mathcal{N}(\epsilon,T,d)}\ \d{\epsilon}.
		\end{align*}
		Here $C$ is a universal constant.
	\end{lemma}

\begin{proof}[Proof of Lemma \ref{lem:upper_bound_lower_bound_lse}]
	By the contraction principle, we only need to handle 
	\begin{align*}
	\E \sup_{Pf^2\leq \rho^2 \delta_n^2} \biggabs{\sum_{i=1}^n \xi_i f(X_i)}, \quad \E \sup_{Pf^2\leq \rho^2 \delta_n^2} \biggabs{\sum_{i=1}^n \epsilon_i f(X_i)}.
	\end{align*}
	Let ${\mathcal{F}}$ be the H\"older class constructed in Lemma \ref{lem:existence_alpha_full}. We first claim that on an event $\mathcal{E}_n$ with probability at least $1-2n^{-2}$, for any $f \in \mathcal{F}$,
	\begin{align}\label{ineq:upper_bound_lower_bound_lse_1}
	\Prob_n f^2\leq \mathfrak{c}_P \left(Pf^2+\frac{\log n}{n}\right).
	\end{align}
	By Lemma \ref{lem:interval_sample}, we see that on an event $\mathcal{E}_n$ with probability at least $1-2n^{-2}$, 
	\begin{align*}
	\frac{1}{n}\sum_{i=1}^n f^2(X_i)=\frac{1}{n}\sum_{j=1}^N \sum_{X_i \in I_j} f^2(X_i)\leq \frac{1}{n}\sum_{j=1}^N c_P \log n\cdot  \max_{X_i \in I_j} f^2(X_i).
	\end{align*}
	Here $N=n/(\kappa_P \log n)$ is the number of intervals $\{I_j\}$. The trick now is to observe that since $f$ is at least $1/2$-H\"older, we have $\max_{X_i \in I_j} f(X_i)\leq \min_{x \in I_j} f(x)+\sqrt{\gamma_n}$, where $\gamma_n=\kappa_P \log n/n$ is the length for each interval $I_j$. Hence on the same event as above, the right side of the above display can be further bounded by
	\begin{align*}
	\frac{2c_P \log n}{n}\sum_{j=1}^N \big(\min_{x \in I_j} f^2(x)+\gamma_n\big)
	&= \frac{2c_P}{\kappa_P} \sum_{j=1}^N \gamma_n \min_{x \in I_j} f^2(x) + \frac{2c_P \log n}{n}\\
	&\leq \frac{2c_P}{\kappa_P}\int_0^1 f^2(x)\ \d{x}+\frac{2c_P \log n}{n},
	\end{align*}
	where the inequality follows from the definition of Riemann integral. The  claim (\ref{ineq:upper_bound_lower_bound_lse_1}) is thus proven by noting that the intergal in the above display is equivalent to $Pf^2$ up to a constant depending on $P$ only. Now using Dudley's entropy integral (see Lemma \ref{lem:dudley_entropy_integral}) and (\ref{ineq:upper_bound_lower_bound_lse_1}), we have for the choice $\sigma_n = \rho (n^{-\frac{1}{2+\alpha}} )\geq \sqrt{\log n/n}$ [the inequality holds when $n\geq \min\{n \geq 3: \rho^2\geq \log n (n^{-\alpha/(2+\alpha)})\}$],
	\begin{align*}
	&\E\sup_{f \in \mathcal{F}: Pf^2\leq \sigma_n^2}\biggabs{\frac{1}{\sqrt{n}}\sum_{i=1}^n \epsilon_i f(X_i)}\\
	& \leq C\E \int_0^{2 \sqrt{\sup_{f \in \mathcal{F}}\Prob_n f^2} }\sqrt{\log \mathcal{N}(\epsilon,\mathcal{F},L_2(\Prob_n))}\ \d{\epsilon}\\
	&\lesssim  \int_0^{2\sqrt{\mathfrak{c}_P(\sigma_n^2+\log n/n)}} \epsilon^{-\alpha/2}\ \d{\epsilon} + J(\infty, \mathcal{F},L_2)\Prob(\mathcal{E}^c_n)\lesssim_{P,\alpha}\big( \sigma_n^{1-\alpha/2}+n^{-2}\big).
	\end{align*}
	Since 
	$
	\sqrt{n}\sigma_n^{1-\alpha/2} =\rho^{1-\alpha/2} n^{\frac{1}{2}-\frac{1}{2}\frac{2-\alpha}{2+\alpha}}=\rho^{1-\alpha/2} n^{\frac{\alpha}{2+\alpha}}$ and $\sqrt{n}\cdot n^{-2}\leq n^{-1}\leq \frac{\rho^2}{n^{\frac{2}{2+\alpha}}(\log n)^2 }\leq \rho^2\leq \rho^{1-\alpha/2} n^{\frac{\alpha}{2+\alpha}}$, 
	in this case Corollary \ref{cor:local_maximal_ineq_barrier} along with the assumption $p\geq 1+2/\alpha$ yields that
	\begin{align*}
	&\E\sup_{f \in \mathcal{F}: Pf^2\leq \sigma_n^2}\biggabs{\sum_{i=1}^n \xi_i f(X_i)}
	\lesssim_{P,\alpha} \rho^{1-\alpha/2} n^{\frac{\alpha}{2+\alpha}}\pnorm{\xi_1}{1+2/\alpha,1}\leq \rho^{1-\alpha/2} n^{\frac{\alpha}{2+\alpha}}\pnorm{\xi_1}{p,1}.
	\end{align*}
	The proof is complete.
\end{proof}

\subsection{Proof of Lemma \ref{lem:lower_bound_exp_lse}}

\begin{proof}[Proof of Lemma \ref{lem:lower_bound_exp_lse}]
	By Lemmas \ref{lem:gine_koltchinskii_matching_bound_ep} and \ref{lem:lower_bound_mep_ep}, and the $\alpha$-fullness of $\tilde{\mathcal{F}}$, we have
	\begin{align*}
	\E \sup_{Pf^2\leq \vartheta^2 \delta_n^2} \biggabs{\sum_{i=1}^n \xi_i f(X_i)}&\geq \frac{1}{2}\pnorm{\xi_1}{1} \E \sup_{Pf^2\leq \vartheta^2 \delta_n^2} \biggabs{\sum_{i=1}^n \epsilon_i f(X_i)}\geq C_1\pnorm{\xi_1}{1} \vartheta^{1-\alpha/2} n^{\frac{\alpha}{2+\alpha}}.
	\end{align*}
	On the other hand, during the proof of Lemma \ref{lem:upper_bound_lse} (see (\ref{ineq:upper_bound_lse_0})) we see that $
	\E \sup_{Pf^2\leq \vartheta^2 \delta_n^2} \abs{\sum_{i=1}^n \epsilon_i f^2(X_i)} \leq  C_2(\vartheta^{1-\alpha/2} \vee 1) \cdot n^{\frac{\alpha}{2+\alpha}}.$ By de-symmetrization,
	\begin{align*}
	\E \sup_{Pf^2\leq \vartheta^2 \delta_n^2} \biggabs{\sum_{i=1}^n (f^2(X_i)-Pf^2)} \leq  2C_2(\vartheta^{1-\alpha/2} \vee 1) \cdot n^{\frac{\alpha}{2+\alpha}}.
	\end{align*}
	Here $C_1,C_2$ are constants depending on $\alpha,P$ only. Now for $\pnorm{\xi_1}{1}\geq 2C_2/C_1$, since $\vartheta\geq 1$, by the triangle inequality we see that
	\begin{align*}
	\E \sup_{Pf^2\leq \vartheta^2 \delta_n^2} \biggabs{\sum_{i=1}^n \big(2\xi_i f(X_i)-f^2(X_i)+Pf^2\big)}\geq C_1\pnorm{\xi_1}{1}\vartheta^{1-\alpha/2} n^{\frac{\alpha}{2+\alpha}},
	\end{align*}
	as desired.
\end{proof}

\section*{Acknowledgements}
The authors would like to thank Vladimir Koltchinskii, Richard Samworth, two referees and an Associate Editor for helpful comments and suggestions on an earlier version of the paper. We also thank Shahar Mendelson for sending us a copy of his paper \cite{mendelson2017extending}.

\bibliographystyle{abbrv}
\bibliography{mybib}

\end{document}